\documentclass{compositio}
\usepackage{amsmath}
\usepackage{amssymb}
\usepackage{amsthm}
\usepackage{latexsym}
\usepackage{hyperref}
\usepackage{enumerate}
\usepackage{graphicx}
\usepackage[all]{xy}
\usepackage{tikz-cd} 
\usepackage{color}
\usepackage{verbatim}
\usepackage{mathtools}

\usepackage[mathscr]{euscript}

\newtheorem{thm}{Theorem}[section]
\newtheorem{cor}[thm]{Corollary}
\newtheorem{lem}[thm]{Lemma}
\newtheorem{prop}[thm]{Proposition}
\newtheorem{conj}[thm]{Conjecture}

\DeclareMathAlphabet\mathbfcal{OMS}{cmsy}{b}{n}
\theoremstyle{definition}
\newtheorem{defn}[thm]{Definition}
\newtheorem{rem}[thm]{Remark}

\newtheorem{ex}[thm]{Example}

\newcommand{\fo}{{\mathfrak{o}}}

\newcommand{\fs}{{\mathfrak{s}}}

\newcommand{\fB}{{\mathfrak{B}}}

\newcommand{\CC}{{\mathbb{C}}}

\newcommand{\bZ}{{\mathbb{Z}}}
\newcommand{\bR}{{\mathbb{R}}}

\newcommand{\Fq}{{\mathbb{F}_q}}

\newcommand{\ZZ}{{\mathbb{Z}}}
\newcommand{\GL}{{\mathsf{GL}}}

\newcommand{\sign}{{\mathrm{sign}}}

\newcommand{\SL}{{\mathsf{SL}}}
\newcommand{\Ind}{{\mathrm{Ind}}}

\newcommand{\cH}{{\mathcal{H}}}
\newcommand{\Hom}{{\mathrm{Hom}}}
\newcommand{\Gal}{{\mathrm{Gal}}}
\newcommand{\PGL}{{\mathsf{PGL}}}
\newcommand{\Sp}{{\mathsf{Sp}}}
\newcommand{\SO}{{\mathsf{SO}}}
\newcommand{\Spin}{{\mathsf{Spin}}}
\newcommand{\SU}{{\mathsf{SU}}}

\newcommand{\FT}{{\mathrm{FT}}}

\newcommand{\rN}{{\mathrm{N}}}
\newcommand{\rO}{{\mathrm{O}}}

\newcommand{\Irr}{{\mathrm{Irr}}}

\newcommand{\reg}{{\mathrm{reg}}}

\newcommand{\Nor}{{\mathrm{N}}}

\newcommand{\rZ}{{\mathrm{Z}}}

\newcommand{\bG}{{\mathbf{G}}}
\newcommand{\bH}{{\mathbf{H}}}

\newcommand{\cA}{{\mathcal{A}}}
\newcommand{\cC}{{\mathcal{C}}}
\newcommand{\cE}{{\mathcal{E}}}
\newcommand{\cF}{{\mathcal{F}}}
\newcommand{\cL}{{\mathcal{L}}}
\newcommand{\cG}{{\mathcal{G}}}
\newcommand{\cR}{{\mathcal{R}}}
\newcommand{\cS}{{\mathcal{S}}}
\newcommand{\cV}{{\mathcal{V}}}

\newcommand{\cY}{{\mathcal{Y}}}

\newcommand{\matje}[4]{\left(\begin{smallmatrix} #1 & #2 \\ 
#3 & #4 \end{smallmatrix}\right)}

\newcommand{\ellip}{{\mathrm{ell}}}
\newcommand{\Aut}{{\mathrm{Aut}}}
\newcommand{\End}{{\mathrm{End}}}
\newcommand{\Id}{{\mathrm{Id}}}
\newcommand{\St}{{\mathrm{St}}}
\newcommand{\triv}{{\mathrm{triv}}}
\newcommand{\iso}{{\mathrm{iso}}}

\newcommand{\Ad}{{\mathrm{Ad}}}
\newcommand{\ad}{{\mathrm{ad}}}
\newcommand{\cusp}{{\mathrm{cusp}}}
\newcommand{\cus}{{\mathrm{c}}}
\newcommand{\der}{{\mathrm{der}}}

\newcommand{\rH}{{\mathrm{H}}}
\DeclareMathOperator{\Inn}{{\mathrm{Inn}}}

\newcommand{\InnT}{{\mathrm{InnT}}}

\newcommand{\EP}{{\mathrm{EP}}}
\newcommand{\Ext}{{\operatorname{Ext}}}
\newcommand{\tr}{{\mathrm{tr}}}

\newcommand{\sgn}{{\mathrm{sgn}}}

\newcommand{\sep}{{\mathrm{s}}}
\def\sc{{\mathrm{sc}}}

\newcommand{\un}{{\mathrm{un}}}
\newcommand{\Frob}{{\mathrm{Frob}}}
\newcommand{\Fr}{{\mathrm{Fr}}}
\newcommand{\isom}{\xrightarrow{\;\sim\;}}

\newcommand{\al}{{\alpha}}

\newcommand \wti[1]{{\widetilde {#1}}}

\newcommand{\bC}{\mathbb{C}}
\newcommand{\lam}{\lambda}
\DeclareMathOperator{\res}{res}

\newcommand{\bI}{\mathbf{I}}
\newcommand{\bJ}{\mathbf{J}}
\newcommand{\bS}{\mathbf{S}}
\newcommand{\cB}{\mathcal{B}}

\newcommand{\cM}{\mathcal{M}}
\newcommand{\cO}{\mathcal{O}}
\newcommand{\cP}{\mathcal{P}}
\newcommand{\cT}{\mathcal{T}}
\newcommand{\cU}{\mathcal{U}}

\newcommand{\enh}{{\mathrm{e}}}

\newcommand{\IC}{{\mathrm{IC}}}

\newcommand{\wnr}{{\mathrm{wur}}}

\newcommand{\temp}{{\mathrm{temp}}}

\newcommand{\cpt}{{\mathrm{cpt}}}

\newcommand{\fR}{{\mathfrak{R}}}
\newcommand{\bnu}{{\pmb{\nu}}}
\newcommand{\qG}{{\mathbb{G}}}
\newcommand{\qH}{{\mathbb{H}}}

\newcommand{\qT}{{\mathbb{T}}}
\usepackage{marginnote}

\usepackage{color}
	\definecolor{bethorange}{rgb}{0.98, 0.3, 0.0}
	\definecolor{amblue}{rgb}{0.32,0.09,0.98}

\numberwithin{equation}{section}

\begin{document}


\title[A tempered Fourier transform]{A nonabelian Fourier transform for tempered unipotent representations}

\author{Anne-Marie Aubert}
\email{anne-marie.aubert@imj-prg.fr}
\address[A.-M. Aubert]{Sorbonne Universit\'e and Universit\'e Paris Cit\'e, CNRS, IMJ-PRG,  
F-75005 Paris, France}

\author
{Dan Ciubotaru}
\email{dan.ciubotaru@maths.ox.ac.uk}
        \address[D. Ciubotaru]{Mathematical Institute, University of Oxford, Oxford OX2 6GG, UK}

\author
{Beth Romano}
 \email{beth.romano@kcl.ac.uk}
        \address[B. Romano]{Department of Mathematics, King's College London, WC2R 2LS, UK}
       
\begin{abstract} We define an involution on the elliptic space of tempered unipotent representations of inner twists of a split simple $p$-adic group $G$ and investigate its behaviour with respect to restrictions to reductive quotients of maximal compact open subgroups. In particular, we formulate a precise conjecture about the relation with a version of Lusztig's nonabelian Fourier transform on the space of unipotent representations of the (possibly disconnected) reductive quotients of maximal compact subgroups. We give evidence for the conjecture, including proofs for $\SL_n$ and $\PGL_n$.
\end{abstract}



\maketitle

\setcounter{tocdepth}{1}
\tableofcontents

\section{ Introduction}

The local Langlands correspondence (LLC) predicts that the depth-zero irreducible smooth representations of a reductive $p$-adic group $G$ are controlled by the geometry of the Langlands dual group $G^\vee$.  
This idea is most well developed for the class of unipotent representations 
 (or representations with unipotent reduction) of $G$ defined by Lusztig in \cite{LuI}, which contains in particular all of the irreducible representations of $G$ with vectors fixed under an Iwahori subgroup \cite{IM,KL}. 
A correspondence for unipotent representations has now been defined that satisfies many of the desired properties of the LLC. These results have come from many years of developments, starting with the seminal papers of Kazhdan and Lusztig \cite{KL} for Iwahori-spherical representations of split adjoint groups, Lusztig \cite{LuI,LuII} for unipotent representations of adjoint groups, Reeder \cite{Re} for Iwahori-spherical representations of split groups of arbitrary isogeny, and finally, the recent papers of Solleveld \cite{So1,So2} (building on \cite{AMS1,AMS3,FOS}) for all reductive $p$-adic groups. Yet not all desired properties of the LLC have been verified in full generality, and one of the main outstanding questions is to understand stability in $L$-packets. 

To be more precise, assume for simplicity that $G$ is the group of $F$-points of an absolutely simple, split connected reductive group over a non-Archimedean local field $F$ with finite residue field. 
In the correspondence mentioned above, the irreducible representations of $G$ (and of its inner forms) are partitioned into $L$-packets indexed by the conjugacy classes $G^\vee\cdot x$ for $x \in G^\vee$.  
From the perspective of abstract harmonic analysis, to understand these $L$-packets, the most basic case to consider is that of {\it tempered} $L$-packets, which correspond to the conjugacy classes $G^\vee\cdot x$ where the semisimple part of $x$ is compact. Many representation-theoretic questions can be reduced further to the case of {\it elliptic tempered} $L$-packets, as defined by Arthur: these consist of tempered representations that are not irreducibly parabolically induced from proper parabolic subgroups.  
As mentioned above, while most of the predicted properties of unipotent $L$-packets have now been verified, 
it is not yet known which linear combinations of Harish-Chandra distribution characters of the representations in a given (elliptic) tempered $L$-packet  are stable, in the sense of being constant on geometric (compact) semisimple conjugacy classes. 

To approach this question, a natural first step is to consider the restriction of unipotent representations to maximal compact open subgroups, as in M\oe glin--Waldspurger's tour de force \cite{MW}, which tackles the question of stability for elliptic tempered $L$-packets for the group $\SO_{2n + 1}$. These maximal compact subgroups allow us to pass from representations of $p$-adic groups to unipotent representations of certain finite reductive groups, which have a rich structure (see, e.g., \cite{Lubook}). 
In particular, while the characters of irreducible representations of a finite connected reductive group do not have good intrinsic stability properties, Lusztig's {\it almost characters}, certain class functions defined in terms of traces of character sheaves, do. The transition matrix between characters and almost characters is Lusztig's famous {\it nonabelian Fourier transform} (\cite{Lubook}, \cite{LualmostBirk}). If we can lift this Fourier transform to the setting of $p$-adic groups, we might be able to lift stability properties of combinations of almost characters of finite reductive groups, as in \cite[Theorem 4.3]{MW}.

With this idea in mind, in this paper, we formulate a conjecture that relates a nonabelian Fourier transform for pure inner twists of a (possibly disconnected) finite reductive group  
and an elliptic Fourier transform $\FT^\vee_{\ellip}$ (cf. \cite{CO-ell,Ciu}) for pure inner twists of $G$. 
In addition to M\oe glin--Waldspurger's work on
the elliptic representations of the special orthogonal groups  \cite{MW,Wa2}, our approach is also inspired by Lusztig's articles proposing a theory of almost characters for $p$-adic groups \cite{LualmostI, LualmostII}. We are also influenced by Reeder's \cite{Re3} and Waldspurger's \cite{Wa} ideas relating the classification of elliptic tempered unipotent representations  and the geometry of $G^\vee$.

We need two main innovations to formulate a precise conjecture. To understand the first, note that if we take a maximal compact open subgroup $K$ of $G$ with reductive quotient $\overline K$, then $\FT^\vee_{\ellip}$ does not necessarily induce a well-defined linear map on the unipotent representation space for $\overline K$. Instead we must look at maximal compact open subgroups in all pure inner twists of $G$ at the same time, and so we form the space $\cC(G)_{\cpt, \un}$ defined below. The reductive quotients $\overline K$ are not necessarily connected, and so the second innovation is to extend the definition of Lusztig's nonabelian Fourier transform to disconnected finite reductive groups. To do this, we must look at all pure inner twists of a disconnected finite reductive group, and the Fourier transform will mix the corresponding representation spaces. These two ideas are related: for every pure inner twist $H$ of $\overline K$, there is a pure inner twist $G'$ of $G$ and a maximal compact open subgroup $K'$ of $G'$ such that the reductive quotient $\overline K'$ is $H$.

\subsection{Main results} We now describe our work in more detail. As above, let us assume that $\bG$ is a simple, split group over $F$ and $G = \bG(F)$. Let $\InnT^p(G)$ denote the set of equivalence classes of pure inner twists of $G$. Then the local Langlands correspondence (see Section \ref{sec:LLC}) says that the $L$-packets of irreducible tempered  unipotent representations of the groups $G' \in \InnT^p(G)$ are in one-to-one correspondence with $G^\vee$-conjugacy classes of elements $x=su \in G^\vee$ (Jordan decomposition) such that $s$ is compact. The elements in the $L$-packet are parametrized by irreducible representations $\phi$ of the group of components $A_{G^\vee}(x)$ of the centralizer of $x$ in $G^\vee$. Hence an $L$-packet is a collection $\{\pi(su,\phi)\mid \phi\in \widehat {A_{G^\vee}(su)}\}$. Let $\Gamma_u$ denote the reductive part of the centralizer of $u$ in $G^\vee$. In \cite{Wa2,Ciu}, one considered the set $\cY(\Gamma_u)$ of pairs $(s,h)\in \Gamma_u^2$ of commuting semisimple elements  and the subset $\cY(\Gamma_u)_\ellip$ of elliptic pairs  (see Section \ref{s:ell-pairs}). These will play a role below in the Langlands parametrization.

Each group $G'\in \InnT^p(G)$ has a finite collection of conjugacy classes of maximal compact open subgroups $\max(G')$. These are classified in terms of the theory of \cite{BT1,IM} (see Section \ref{s:maximal}). A compact group $K'\in \max(G')$ has a finite quotient $\overline K'$ that is the group of $k$-points of a (possibly disconnected) reductive group over a finite field $k$. Write $R_\un(\overline{K}')$ for the $\bC$-vector space spanned by the irreducible unipotent representations of $\overline K'$. As mentioned above, for connected finite reductive groups, Lusztig \cite{Lubook} defined the nonabelian Fourier transform, which is the change-of-basis matrix between the basis of irreducible unipotent characters and the basis of unipotent almost characters. This is recalled in Section \ref{s:Lusztig}. We need to define an extension of this map to disconnected finite groups in the spirit of \cite{Lu86} and \cite[Section 5]{DM}. To fit with our picture, we define a nonabelian Fourier transform for the representations of the pure inner twists of the finite (possibly disconnected) reductive group $\overline K$, where $K \in \max(G)$. See Section \ref{s:disconn}. The point is that this transform gives an involution 
\begin{equation}
\FT_{\cpt,\un}\colon\cC(G)_{\cpt,\un}\to \cC(G)_{\cpt,\un}, 
\end{equation}
on the space 
\begin{equation*}
    \cC(G)_{\cpt,\un}=\bigoplus_{G' \in \InnT^p(G)} \bigoplus_{K' \in \max (G')} R_\un(\overline K'),
\end{equation*} which we can think of as the sum over $K\in\max(G)$ of the unipotent representation spaces of the pure inner twists of $\overline K$.
See (\ref{e:partition}) and Definition \ref{d:FT-compact}. It is important to notice that, in general, $\FT_{\cpt,\un}$ mixes the pure inner twists of a given $\overline K$.

Since parabolic induction of characters is generally well understood, of particular interest is the space of elliptic (unipotent) tempered representations for all pure inner twists
\[\cR^p_{\un,\ellip}(G)=\bigoplus_{G'\in\InnT^p(G)} \overline R_\un(G')
\]
(see Section \ref{s:ell-FT}).  
Generalizing \cite{Re3}, we prove the following theorem.

\begin{thm}[(Theorem \ref{t:main-elliptic})]
Suppose that $G$ is split and adjoint. The local Langlands correspondence induces an isometric isomorphism 
\begin{equation}
\overline{\mathsf{LLC}^p}_{\un}:\bigoplus_{u}\CC[\cY(\Gamma_u)_\ellip]^{\Gamma_u}\longrightarrow \cR^p_{\un,\ellip}(G),\ (s,h)\mapsto \Pi(u,s,h),
\end{equation}
where the left-hand side has a natural elliptic inner product while the right-hand side is endowed with the Euler--Poincar\' e product. The element $u$ ranges over representatives of unipotent conjugacy classes in $G^\vee$ and $\Pi(u,s,h)$ is defined in (\ref{eqn:pi_ush}). 
\end{thm}
We remark that $\Pi(u,s,1)$ is expected to be the stable combination of characters in the $L$-packet, while in general  $\Pi(u,s,h)$ are expected to satisfy the endoscopic identities. 

The proof of Theorem \ref{t:main-elliptic} in Section \ref{s:elliptic} applies in more generality, for example for Iwahori-spherical representations of groups of arbitrary isogeny (see Section \ref{s:extend}). Since the left-hand side has an obvious involution given by the flip $(s,h)\to (h,s)$, this defines an involution, the dual elliptic nonabelian transform 
\begin{equation}
\FT^\vee_\ellip\colon \cR_{\un,\ellip}^p(G) \to\cR_{\un,\ellip}^p(G).
\end{equation}
We notice that $\FT^\vee_\ellip$ mixes representations of the pure inner twists of $G$. We expect that there is a commutative diagram as follows.

\begin{conj}[(Conjecture \ref{c:elliptic})]\label{c:intro} Up to certain roots of unity (see Remark \ref{r:roots}), the following diagram commutes:
 \begin{displaymath}
 \xymatrix@+1pc{
    {\cR_{\un,\ellip}^p(G)} \ar[r]^{\mathrm{FT}^\vee_\ellip} \ar[d]_{\res_{\cpt,\un}}
    & {\cR_{\un,\ellip}^p(G)}\ar[d]^{\res_{\cpt,\un}}\\
    {  \cC(G)_{\cpt,\un}} \ar[r]_{\mathrm{FT_{\cpt,\un}}}
    & {  \cC(G)_{\cpt,\un}} }
\end{displaymath} 
where the vertical arrows are defined by taking invariants by the pro-unipotent radicals of maximal compact subgroups. 
\end{conj}
It is also natural to expect that the images of irreducible elliptic tempered characters under $\FT^\vee_\ellip$ are the ``almost characters'' (on elliptic elements) defined in \cite{LualmostI}: see for comparison \cite[Conjecture 2.2(c)]{LualmostII}.

Conjecture \ref{c:intro} is a generalization of \cite[Conjecture 1.3]{Ciu} with an important difference: we remark that the role of maximal compact subgroups (rather than maximal parahoric subgroups) and hence of a Fourier transform for pure inner twists of disconnected finite reductive groups in the conjecture is essential for treating all pure inner twists of $G$. 
 We verify the conjecture in some examples. In particular, we have the following theorem.
 \begin{thm}
 If $\bG=\SL_n$ or $\PGL_n$, Conjecture \ref{c:intro} holds.
 \end{thm}
See  Section \ref{s:type-A} and  Section \ref{s:pgl}.  We also verify the conjecture for $\bG= \Sp_4$ (Section \ref{s:sp4}).
The results of Waldspurger \cite{Wa2} show that this conjecture holds when $\bG=\SO_{2n+1}$.

\medskip

In future work, we will consider a generalization of Conjecture \ref{c:intro} to the space of compact/rigid tempered representations defined in \cite{CH,CH2}.

  \subsection{Structure of the paper}
 
 In Sections \ref{s:inner}, \ref{sec:GSC}, and \ref{sec:LLC}, we review relevant background about inner twists of $p$-adic groups, the generalized Springer correspondence, and the local Langlands correspondence. In Section \ref{s:Lusztig}, we recall Lusztig's parametrization of unipotent representations of a connected reductive group over a finite field and the definition of the nonabelian Fourier transform on the space spanned by these representations. We then extend Lusztig's parametrization: for the (possibly disconnected) groups $\overline{K}$ that arise as reductive quotients of subgroups $K \in \max (G)$ as defined above, we parametrize the union over all pure inner twists $\overline K'$ of $\overline K$ of the set of unipotent representations of $\overline K'$, and we then define a nonabelian Fourier transform on the space spanned by these representations (see Section \ref{s:disconn}).  
 
In Section \ref{s:maximal}, we return to the setting of $p$-adic groups.  We review the parametrization of maximal compact open subgroups of $G' \in \InnT^p(G)$, under the assumption $G$ is $F$-split.   We define the space $\cC(G)_{\cpt, \un}$ in terms of these subgroups, and we use the Fourier transform of Section \ref{s:disconn} to define an involution $\FT_{\cpt, \un}$ on $\cC(G)_{\cpt, \un}$. In Section \ref{s:ellipticpairs}, we review the definition of $\cY(\Gamma)_{\ellip}$ for a complex reductive group $\Gamma$. We also review the definition of the elliptic pairing on the Grothendieck group of a finite group.  

Section \ref{s:dualFT} contains the conjectures outlined above. We first review the Euler--Poincar{\'e} pairing and state Conjecture \ref{conj:main-ell}, which predicts that the local Langlands correspondence induces an isometric isomorphism at the level of elliptic spaces. We then define a restriction map $\res_{\cpt, \un}\colon \cR^p_{\un, \ellip}(G) \to \cC(G)_{\cpt, \un}$ and state Conjecture \ref{c:elliptic}, which predicts that the elliptic nonabelian Fourier transform $\FT_{\ellip}^\vee$ is compatible with $\FT_{\cpt, \un}$ under $\res_{\cpt, \un}$. We give evidence for this conjecture in Proposition \ref{p:regular}, which considers linear combinations of twists of Steinberg representations.

In Section \ref{sec:Lusztig}, we present an alternative definition of the elliptic nonabelian Fourier transform motivated by Lusztig's pairing \cite[\S1.3]{LualmostII}.

In Section \ref{s:elliptic}, we prove Conjecture \ref{conj:main-ell} in the case when $G$ is simple, split, and adjoint. In Section \ref{s:extend}, we indicate how the proof can be extended to the non-adjoint case. In the final three sections, we verify the conjectures for explicit examples: in Section \ref{s:sp4}, we consider the group $\Sp_4(F)$; in Section \ref{s:type-A}, we consider $\SL_n(F)$; and in Section \ref{s:pgl}, we consider $\PGL_n(F)$.

\subsection{Notation and conventions}

Given a complex Lie group $\cG$, we write $\rZ_{\cG}$ for the center of $\cG$. Given $x \in \cG$, we write $\rZ_{\cG}(x)$ for the centralizer of $x$ in $\cG$. Similarly, if $H$ is a subgroup of $\cG$, we write $\rZ_{\cG}(H)$ for the centralizer of $H$ in $\cG$, and if $\varphi$ is a homomorphism with image in $\cG$, we write $\rZ_{\cG}(\varphi)$ for $\rZ_{\cG}(\text{im } \varphi)$. We write $\cG^\circ$ for the identity component of $\cG$. If $x \in \cG$, we write $A_{\cG}(x) = \rZ_{\cG}(x)/\rZ_{\cG}(x)^\circ$ for the component group of $\rZ_{\cG}(x)$. If $u \in \cG$ is unipotent, we write $\Gamma_u$ for the reductive part of $\rZ_{\cG}(u)$.
Given a torus $\cT$, we write $X^*(\cT)$ for the character group of $\cT$. 

Given a finite group $A$, we write $\widehat{A}$ for the set of irreducible characters of $A$, and we write $R(A)$ for the $\bC$-vector space with basis given by (isomorphism classes of) irreducible representations of $A$. 
Given a finite set $S$, we write $\mathbb{C}[S]$ for the $\mathbb{C}$-vector space of functions $S \to \mathbb{C}$.

\medskip

\noindent{\bf Acknowledgements.} We thank M. Solleveld for his careful reading and many useful comments, particularly regarding Section \ref{s:elliptic}, and G. Lusztig and M. Reeder for their helpful suggestions. We also thank the referees for the thorough checking of the paper, for the corrections and suggestions for improvement. This research was supported in part by the EPSRC grant  EP/V046713/1 (2021). D.C. thanks Universit\'e Paris Cit\'e and Sorbonne Universit\'e for their hospitality while part of this work was completed.

\section{ Recollection on inner twists}\label{s:inner}

\subsection{Inner twists} Let $F$ be a non-Archimedean local field with finite residue field $k_F=\Fq$. We denote by $\mathfrak o_F$ the ring of integers of $F$. Let $F_\sep$ be a fixed separable closure of $F$, and let $\Gamma_F$ denote the Galois group of $F_\sep/F$. Let $F_\un\subset F_\sep$  be the maximal unramified extension of $F$. Let $\Frob$ be the geometric Frobenius element of $\Gal(F_\un/F)\simeq\widehat\ZZ$, i.e., the topological generator that induces the inverse of the automorphism $x\mapsto x^q$ of $k_F$. We denote by $\Fr_\bG$ the action of $\Frob$ on a connected reductive $F$-group $\bG$. We now review definitions related to inner twists and pure inner twists of a $p$-adic group. For details see, e.g., \cite[Section 2]{Vo}, \cite[Section 2]{KaLLC}, \cite[Section 1.3]{ABPSconj}. (Note that \cite{Vo} uses the term ``pure rational form" for what we call a pure inner twist.)

Let $G=\bG(F)$. Write $\Inn(\bG)$ for the group of inner automorphisms of $\bG$. Recall that given an algebraic group $\bH$ over $F$, an isomorphism $\alpha\colon \bH \to \bG$ defined over $F_s$ determines a $1$-cocycle
\begin{equation}\label{eqn:iso}
\gamma_\alpha \colon  \begin{array}{ccc} \Gamma_F & \to & \Aut (\bG) \\
\sigma & \mapsto & \alpha \sigma \alpha^{-1} \sigma^{-1} .                 
\end{array}
\end{equation}
An \textit{inner twist} of $G$ consists of a pair $(H,\alpha)$, where $H=\bH(F)$ for some connected reductive $F$-group $\bH$, and $\alpha\colon\bH\isom\bG$ is an isomorphism of algebraic groups defined over $F_s$ such that $\text{im } (\gamma_\alpha) \subset \Inn(\bG)$. Two inner twists $(H, \al), (H', \al')$ of $G$ are {\em equivalent} if there exists $f \in \Inn(\bG)$ such that
\begin{equation} \label{eq:cohomologous}
\gamma_\alpha (\sigma) = f^{-1} \gamma_{\alpha'} (\sigma) \; \sigma f \sigma^{-1}
\quad \text{ for all } \sigma \in \Gamma_F .
\end{equation}
Denote the set of equivalence classes of inner twists of $G$ by $\InnT(G)$.

An inner twist of $G$ is the same thing as an inner twist of the unique quasi-split inner form $G^* = \bG^* (F)$ of $G$.
Thus the equivalence classes of inner twists of $G$ are parametrized by the Galois cohomology group $H^1(F,\Inn(\bG^*))$:
\[\InnT(G)\longleftrightarrow H^1(F,\Inn(\bG^*)).
\]

\begin{ex}
For $G=\SL_n(F)$, there is a one-to-one correspondence 
\begin{equation}
\InnT(\SL_n(F))\longleftrightarrow\mathbb Z/n\mathbb Z.
\end{equation}
This is given as follows. Let $r$ be an integer mod $n$ and let $m=\gcd(r,n)$. Then $n=dm$ and $r/m$ is coprime to $d$. Therefore, there exists a division algebra $D_{d,r/m}$, central over $F$ and of dimension $\dim_F D_{d,r/m}=d^2$. The corresponding inner twist is $\SL_m(D_{d,r/m})$.
\end{ex}

A {\it pure inner twist} of $G$ is a triple $(H,\alpha,z)$, where $(H,\alpha)$ is an inner twist and  $z \in Z^1(F,G)$ such that $\alpha^{-1}\circ\gamma(\alpha)=\Ad(z(\gamma))$ for any $\gamma\in\Gamma_F$ \cite[Section 2.3]{KaLLC}. When $G$ splits over an unramified extension of $F$ such a cocycle is determined by the image $u:= z(\Frob) \in G$. The corresponding inner twist of $\bG$ is then defined by the functorial image $z_\ad \in Z^1(F,\Inn(\bG^*))$ of $z$. This pure inner twist is defined by the twisted Frobenius action $\Fr_u$ on $\bG$ given by $\Fr_u=\Ad(u)\circ\Fr_\bG$.

In cohomological terms, the short exact sequence 
\[1\longrightarrow \rZ_{\bG^*}\longrightarrow \bG^*\longrightarrow \Inn(\bG^*)\longrightarrow 1
\]
induces a map in cohomology $\rH^1(F,\Inn(\bG^*))\to \rH^2(F,\rZ_{\bG^*})$. An inner twist of $\bG^*$ has a corresponding pure inner twist if and only if the corresponding element of $\rH^2(F,\rZ_{\bG^*})$ is trivial \cite[Lemma 2.10]{Vo}. 
Denote by $\InnT^p(G^*)$ the set of equivalence classes of pure inner twists of $\bG^*$. We have \cite[Proposition 2.7]{Vo}
\begin{equation}
\InnT^p(G^*)\longleftrightarrow \rH^1(F,\bG^*).
\end{equation}

\begin{ex}
If $\bG^*$ is semisimple adjoint, every inner twist corresponds to a unique pure inner twist: $\InnT^p(G^*)=\InnT(G^*)$. If $\bG^*$ is semisimple and simply connected, $\rH^1(F,\Inn(\bG^*))\cong \rH^2(F,\rZ_{\bG^*})$ and therefore there is only one class of pure inner twists, the quasi-split form, $\InnT^p(G^*)= \{G^*\}$. When $G=\SL_n(F)$, the only pure inner twist is $G$ itself (see \cite[Example 2.12]{Vo}).
\end{ex}

\subsection{The \texorpdfstring{$L$}{L}-group}\label{s:Lgrp} 
Let $G^\vee$ denote the $\CC$-points of the dual group  of $\bG$. It is endowed with an action of $\Gamma_F$.
Let $W_F$ be the Weil group of $F$ (relative to $F_\sep/F$) and let ${}^L G:= G^\vee \rtimes W_F$ denote the $L$-group of $G$. 

Kottwitz proved in \cite[Proposition 6.4]{Ko} that there exists a natural isomorphism
\begin{equation} \label{eqn:Ko-iso}
\kappa_G \colon  \rH^1 (F,\bG) \isom \Irr\Big(\pi_0 \big(\rZ_{G^\vee}^{W_F} \big) \Big).
\end{equation}

Let $G^\vee_\sc$ denote the simply connected cover of the derived group $G_\der^\vee$ of $G^\vee$.
We have $G^\vee_{\sc}=(G_{\ad})^\vee$, and
\begin{equation}\label{eqn:Kot}
\kappa_{G^*_{\ad}} \colon \rH^1 (F,\Inn(\bG^*)) \isom \Irr \big(\rZ_{G^\vee_{\sc}}^{W_F}\big).
\end{equation}
All the inner twists of a given group $G$ share the same $L$-group, because the action of $W_F$ on $G^\vee$ is only uniquely defined up to inner automorphisms. 
This also works the other way around: from the Langlands dual group ${}^L G$ one can recover the inner-form class of $G$. 

\begin{ex} \label{ex:Sp}
If $G=\Sp_{2n}(F)$, then we have $G^\vee=\SO_{2n+1}(\CC)$ and $G_\sc^\vee=\Spin_{2n+1}(\CC)$, so $\rZ_{G^\vee_{\sc}}\simeq\ZZ/2\ZZ$. 
An inner twist of $G$ is determined by its Tits index \cite{Ti}. The group $G^*=G$ is split and its nontrivial inner twist is the group $\SU(n,h_r)$, where $h_r$ is 
a non-degenerate Hermitian form of index $r=\lfloor n/2\rfloor$ over the quaternion algebra $Q$ over $F$ (see for instance \cite[\S~9]{ArClass}).
\end{ex}

We will consider $G$ as an inner twist of $G^*$, so endowed with an isomorphism $\bG\to\bG^*$ over $F_s$. Via (\ref{eqn:Kot}), $G$ is labelled by a character $\zeta_G$ of $\rZ_{G_\sc^\vee}^{W_F}$. We choose an extension $\zeta$ of $\zeta_G$ to $\rZ_{G^\vee_\sc}$.

\section{ Generalized Springer correspondence for disconnected groups} \label{sec:GSC}
Let $\cG$ be a possibly disconnected complex Lie group. We denote by $\cG^\circ$ the identity component of $\cG$. Let $u$ be a unipotent element in $\cG^\circ$, and let $A_{\cG^\circ}(u)$ denote the  group of components of  $\rZ_{\cG^\circ}(u)$.

Let $\phi^\circ$ be an irreducible representation of {$A_{\cG^\circ}(u)$}. The pair $(u,\phi^\circ)$ is called \textit{cuspidal} if it determines a $\cG^\circ$-equivariant cuspidal local system on the $\cG^\circ$-conjugacy class of $u$ as defined in \cite{LuIC}. 
In particular,  if $(u,\phi^\circ)$ is cuspidal, then $u$ is a distinguished unipotent element in $\cG^\circ$ (that is, $u$ does not meet the unipotent variety of  any proper Levi subgroup  of $\cG^\circ$), \cite[Proposition~2.8]{LuIC}. However, in general not every distinguished unipotent element supports a cuspidal representation.

\begin{ex} \label{ex:cusps} For $\cG:=\SL_n(\CC)$, the unipotent classes in $\cG$ are in bijection with the partitions $\lambda=(\lambda_1,\lambda_2,\ldots,\lambda_r)$ of $n$: the corresponding $\cG$-conjugacy class $\cO_\lambda$ consists of unipotent matrices with Jordan blocks of sizes $\lambda_1$, $\lambda_2$, $\ldots$,$\lambda_r$. We identify the center $\rZ_{\cG}$ with the group $\mu_n$ of complex $n$th roots of unity. For $u\in\cO_\lambda$, the natural homomorphism $\rZ_\cG\to A_\cG(u)$ is surjective with kernel $\mu_{n/\gcd(\lambda)}$, where $\gcd(\lambda):=\gcd(\lambda_1,\lambda_2,\ldots,\lambda_r)$. Hence the irreducible $\cG$-equivariant local systems on $\cO_\lambda$ all have rank one, and they are distinguished by their central characters, which range over those $\chi\in\widehat{\mu_n}$ such that $\gcd(\lambda)$ is a multiple of the order of $\chi$. We denote these local systems by $\cE_{\lambda,\chi}$. The unique distinguished unipotent class in $\cG$ is the regular unipotent class $\cO_{(n)}$, consisting
of unipotent matrices with a single Jordan block. The cuspidal irreducible $\cG$-equivariant local systems are supported on $\cO_{(n)}$ and are of the form $\cE_{(n),\chi}$, with $\chi\in\widehat{\mu_n}$ of order $n$ (see \cite[(10.3.2)]{LuIC}). 
\end{ex}

The group $A_{\cG^\circ}(u)$ may be viewed as  a subgroup of the group $A_u:=A_{\cG}(u)$ of components of $\rZ_{\cG}(u)$. Let $\phi$ be an irreducible representation of $A_{\cG}(u)$. We say that $(u,\phi)$ is a \textit{cuspidal pair} if the restriction of $\phi$ to $A_{\cG^\circ}(u)$ is a direct sum of irreducible representations $\phi^\circ$ such that one (or equivalently any) of the pairs $(u,\phi^\circ)$ is cuspidal.
Let 
\[
\bI^{\cG}:=\{(U,\cE)\mid U\text{ unipotent conjugacy class in }\cG,\ \cE\text{ irred. $\cG$-eqvt. local system on }U\}.
\]
This set can be identified with the set of $\cG$-orbits of pairs $(u,\phi)$, where $u\in \cG$ is unipotent and $\phi\in\widehat A_u$.
  If $(\phi,V_\phi)$ is an irreducible $A_u$-representation, we can first regard it as an irreducible $\rZ_\cG(u)$-representation, and then the corresponding local system is $\cE=(\cG\times_{\rZ_\cG(u)} V_\phi\to \cG/\rZ_\cG(u)\cong U)$.
We denote by $\bI_\cus^\cG$ the subset of $\bI^\cG$ of cuspidal pairs. We write $\bI:=\bI^{\cG^\circ}$ and $\bI_\cus:=\bI_\cus^{\cG^\circ}$.

Let $\bJ^{\cG}$ denote the set of $\cG$-orbits  of triples $j=(\cM,U_\cus,\cE_\cus)$ such that $\cM^\circ$ is a Levi subgroup of $\cG^\circ$, 
\begin{equation} \label{eqn:disc_Levi}
\cM:=\rZ_{\cG}(\rZ_{\cM^\circ}^\circ),
\end{equation}
and $(U_\cus,\cE_\cus)\in \bI_\cus^{\cM^\circ}$. We observe that $\cM$ has identity component $\cM^\circ$ and that $\rZ_{\cM}^\circ = \rZ_{\cM^\circ}^\circ$.
We set $\bJ:=\bJ^{\cG^\circ}$. We notice that $\cM=\cM^\circ$ whenever $\cG=\cG^\circ$.

Let $\rZ^\circ_{\cM^\circ,\reg}=\{z\in Z^\circ_{\cM^\circ}\mid \rZ_\cG(z)=\cM^\circ\}$ and $Y_j(\cG)=\bigcup_{x\in \cG} x (\rZ^\circ_{\cM^\circ,\reg} U_\cus) x^{-1}$. Let $\overline Y_j(\cG)$ be the closure of $Y_j(\cG)$ in $\cG$. We set $Y_j=Y_j(\cG^\circ)$ and $\overline Y_j=\overline Y_j(\cG^\circ)$. For example, if $j_0=(T,1,\triv)$ is the trivial cuspidal pair on the maximal torus $T$ in $\cG^\circ$, then $Y_{j_0}$ is the variety of regular semisimple elements in $\cG^\circ$, hence $\overline Y_{j_0}=\cG^\circ$.

Set $W_j^\circ:=\Nor_{\cG^\circ}(\cM^\circ)/\cM^\circ$. This is a Coxeter group due to the particular nature of the Levi subgroups in $\cG^\circ$ that support cuspidal local systems (see \cite[Theorem~9.2]{LuIC}).

One constructs a $\cG^\circ$-equivariant semisimple perverse sheaf $K_j$ supported on $\overline Y_j$ that has a $W_j^\circ$-action and a decomposition (\cite[Theorem 6.5]{LuIC} and \cite[\S~5]{AMS1})
\[K_j=\bigoplus_{\rho^\circ\in \widehat W_j^\circ} V_{\rho^\circ}\otimes A_{j,\rho^\circ},
\]
where $(\rho^\circ,V_{\rho^\circ})$ ranges over the (equivalence classes of) irreducible $W_j^\circ$-representations and $A_{j,\rho^\circ}$ is an irreducible $\cG^\circ$-equivariant perverse sheaf. The perverse sheaf $A_{j,\rho^\circ}$ has the property that there exists a (unique) pair $(U,\cE^\circ)\in \bI$ such that its restriction to the variety $\cG^\circ_\un$ of unipotent elements in $\cG^\circ$ is: 
\begin{equation}\label{eqn:IC-sheaf}
(A_{j,\rho^\circ})|_{\cG^\circ_\un}[-\dim(\rZ^\circ_{\cM^\circ})]\cong \IC(\overline U,\cE^\circ)[\dim(U)].
\end{equation}
In particular, the hypercohomology of $A_{j^\circ,\rho^\circ}$ on $U$ is concentrated in one degree, namely 
\[\cH^{a_U}(A_{j,\rho^\circ})|_U\cong \cE^\circ,\text{ where } a_U=-\dim(U)-\dim(\rZ^\circ_{\cM^\circ}).
\]
If we set $\widetilde{\bJ}=\widetilde{\bJ}^{\cG^\circ}:=\{(j,\rho^\circ)\,:\, j\in \bJ^{\cG^\circ},\ \rho^\circ\in \widehat{{W_j^\circ}}\}$, the generalized Springer correspondence for $\cG^\circ$ is the bijection 
\begin{equation}\label{eqn:gen-Springer}
\nu^\circ\colon \bI^{\cG^\circ} \to \widetilde{\bJ}^{\cG^\circ},\quad (U,\cE)\mapsto (j,\rho^\circ),
\end{equation}
where the relation between $(j,\rho^\circ)$ and $(U,\cE)$ is given by (\ref{eqn:IC-sheaf}). 
Let $\nu^\circ_\cus\colon \bI\to \bJ$ denote the composition of $\nu^\circ$ with the projection from $\widetilde \bJ$ to $\bJ$.

\smallskip

We will now explain how, following \cite[\S~4]{AMS1}, one can extend the maps $\nu^\circ $ and $\nu_\cus^\circ$ to the case of disconnected groups.
Let $j=(\cM,U_\cus,\cE_\cus)\in \bJ^{\cG}$. We set  $W_j:=\Nor_{\cG}(j)/\cM^\circ$. There exists a subgroup $\fR_j$ of $W_j$ such that $W_j=W_j^\circ\rtimes \fR_j$ (see \cite[Lemma~4.2]{AMS1}). 
Suppose that $\sharp_j$ is a $2$-cocycle 
\[\sharp_j\colon \fR_j\times\fR_j\to \overline{{\mathbb Q}}_\ell^\times.\] 
We view $\sharp_j$ as a $2$-cocycle on $W_j$ that is trivial on $W_j^\circ$. Then the $\sharp_j$-twisted group algebra of $W_j$, denoted by $\overline{{\mathbb Q}}_\ell[W_j,\sharp_j]$, is defined to be the $\overline{{\mathbb Q}}_\ell$-vector space $\overline{{\mathbb Q}}_\ell[W_j,\sharp_j]$ with basis $\left\{f_w\,:\,w\in W_j\right\}$ and multiplication rules
\[f_wf_{w'}=\sharp_j(w,w')f_{ww'},\quad w,w'\in W_j.\]

One constructs a $\cG$-equivariant semisimple perverse sheaf $K_j$ supported on $\overline Y_j$ that has a $W_j$-action and a decomposition \cite[Theorem 6.5]{LuIC}
\[K_j=\bigoplus_{\rho\in \Irr(\overline{{\mathbb Q}}_\ell[W_j,\sharp_j])} V_\rho\otimes A_{j,\rho},
\]
where $(\rho,V_{\rho})$ ranges over the (equivalence classes of) simple modules of  $\overline{{\mathbb Q}}_\ell[W_j,\sharp_j]$, and $A_{j,\rho}$ is an irreducible $\cG$-equivariant perverse sheaf.

We set 
\begin{equation}\label{eqn:tildeJ}
\widetilde{\bJ}^\cG:=\{(j,\rho)\,:\, j\in \bJ^\cG,\ \rho\in 
\Irr(\overline{{\mathbb Q}}_\ell[W_j,\sharp_j])\}.
\end{equation} 
The generalized Springer correspondence for $\cG$ is the bijection 
\begin{equation}\label{eqn:gen-Springer_disc}
\nu\colon \bI^\cG \to \widetilde{\bJ}^\cG
\end{equation}
defined in \cite[Theorem~5.5]{AMS1}. 
\begin{defn} \label{defn:nu_cus}
Let $\nu_\cus=\nu_\cus^\cG\colon \bI^\cG\to \bJ^\cG$ denote the composition of $\nu$ with the projection from $\widetilde \bJ^\cG$ to $\bJ^\cG$.
\end{defn}
{Suppose $(U, \cE) \in \bI^\cG$, and suppose the $\cG$-class $U$ splits into $\cG^\circ$-classes $U^\circ_1$, $\ldots$, $U^\circ_\ell$, for some $\ell\ge 1$. If we regard $\cE$ as a $\cG^\circ$-equivariant local system, then it restricts as $\cE |_{U_i^\circ}=\bigoplus_{t=1}^{k_i}\cE_{i,t}^\circ$, $1\le i\le \ell$, where $\nu^\circ(U^\circ_i,\cE_{i,t}^\circ)=(j^\circ,\rho_{i,t}^\circ)$, with $j^\circ=\nu^\circ_\cus(U^\circ_1,\cE^\circ_{1,1})$ and
$\rho|_{W_j^\circ}=\bigoplus_{i,t}\rho^\circ_{i,t}$. }

\begin{ex}
Let $\cG=\mathsf O_{2n}(\CC)$, $\cG^\circ=\mathsf{SO}_{2n}(\CC)$, $\cG/\cG^\circ\cong \bZ/2\bZ$. The unipotent classes in $\cG$ are parametrized by partitions $\lambda=(\lambda_1,\dots,\lambda_m)$ of $2n$ such that each even part appears with even multiplicity. If $U_\lambda$ is the corresponding unipotent class, then $U_\lambda$ is a single $\cG^\circ$-class unless the partition $\lambda$ is ``very even'' \cite{SpSt,CM}, i.e., all parts $\lambda_i$ are even, in which case $U_\lambda$ splits into two $\cG^\circ$-classes, $U_\lambda^+$ and $U_\lambda^-$. 

Let $j=j_0$ correspond to the trivial cuspidal local system on the torus of $\cG^\circ$. Then $W_{j_0}^\circ=W^\circ\cong W(D_n)$ and $W_j=W\cong W(B_n)$, hence $W/W^\circ\cong \cG/\cG^\circ=\bZ/2\bZ$. (Here $W(D_n)$ denotes a Weyl group of type $D_n$ and similarly for $W(B_n)$.) If $\lambda$ is not a very even partition and $u\in U_\lambda$, then $A_u/A_{\cG^\circ}(u)=\bZ/2\bZ$; if $(j_0,\rho^\circ)=\nu^\circ(U_\lambda^\circ,\phi^\circ)$, then there are two nonisomorphic ways $\phi,\phi'$ in which one can extend $\phi^\circ$ to $A_u$, and two nonisomorphic ways $\rho,\rho'$ to extend $\rho^\circ$ to $W$, which can be chosen such that $\rho$ corresponds to $\phi$ and $\rho'$ corresponds to $\phi'$ under the disconnected Springer correspondence.

If, on the other hand, $\lambda$ is a very even partition, $u=u^+$ is a representative of $U^+_\lambda$, and $u^-$ a representative of $U^-_\lambda$, then $A_u=A_{\cG^\circ}(u^+)=A_{\cG^\circ}(u^-)=\{1\}$. In this case, $\nu^\circ(U^\pm,\mathbf 1)=\rho^\pm$ ($W(D_n)$-representations), where $\rho$ is parametrized by a bipartition of $n$ of the form $\lam' \times \lam'$ (necessarily $n$ is even). Then $\nu(U,\mathbf 1)=\rho$ ($W(B_n)$-representation), where $\rho|_{W(D_n)}=\rho^+\oplus \rho^-$.
\end{ex}

\section{ The Langlands parametrization}\label{sec:LLC}

We use the notation of Section \ref{s:inner}. In addition, we write $I_F$ for the inertia subgroup of $W_F$, and we set $W'_F:=W_F\times\SL_2(\CC)$. We have natural projections from $p_1\colon W'_F\twoheadrightarrow W_F$ and $p_2\colon{}^LG\twoheadrightarrow W_F$. 

\subsection{Langlands parameters} A \textit{Langlands parameter} (or \textit{$L$-parameter}) for $G$ is a continuous morphism
$\varphi\colon W'_F\to {}^LG$ such that 
$\varphi(w)$ is semisimple for each $w\in W_F$ (that is, $r (\varphi (w))$
is semisimple for every finite-dimensional representation $r$ of 
${}^L G$), the restriction of $\varphi$ to $\SL_2(\CC)$ is a morphism of complex algebraic groups, and the diagram 
\begin{center}
    $\xymatrix {
W_F' \ar[rr]^{\varphi} \ar[dr]_{p_1} & & {}^L G \ar[dl]^{p_2} \\
& {W_F} & }$
\end{center}
commutes. Write $\Phi(G)$ for the set of $G^\vee$-conjugacy classes of Langlands parameters for $G$.

\smallskip

Let $\rZ_{G^\vee}(\varphi)$ denote the centralizer in $G^\vee$ of $\varphi(W'_F)$. We have
\begin{equation} \label{eqn:inter}
\rZ_{G^\vee}(\varphi)\cap\rZ_{G^\vee}=\rZ_{G^\vee}^{W_F},
\end{equation}
and hence
\[\rZ_{G^\vee}(\varphi)/\rZ_{G^\vee}^{W_F}\simeq \rZ_{G^\vee}(\varphi)\rZ_{G^\vee}/\rZ_{G^\vee}.\]
The group $\rZ_{G^\vee}(\varphi)\rZ_{G^\vee}/\rZ_{G^\vee}$ can be considered as a subgroup of $G^\vee_\ad$ and we define $\rZ^1_{G^\vee_\sc}(\varphi)$ to be its inverse image under the canonical projection $p\colon G^\vee_\sc \to G^\vee_\ad$. The group $\rZ^1_{G^\vee_\sc}(\varphi)$ coincides with the one introduced by Arthur in \cite[(3.2)]{ArNote} (denoted there by $\widetilde{\cS}_\varphi$). As observed in \cite{ArNote}, it is an extension of $\rZ_{G^\vee}(\varphi)/\rZ_{G^\vee}^{W_F}$  by $\rZ_{G^\vee_\sc}$.
Let $A^1_{\varphi}$ denote the component group of $\rZ^1_{G^\vee_\sc}(\varphi)$.

\begin{rem} \label{rem: LLC}{\rm   
When $Z_{\bG}^\circ$ is $F$-split, the group $A^1_{\varphi}$ also coincides  with the group considered by Kaletha in \cite[\S 4.6]{Ka} in the parametrization of the $L$-packet of $\varphi$.}
\end{rem}

An \textit{enhancement} of  $\varphi$ is an irreducible representation $\phi$ of $A^1_\varphi$.  We denote by $\widehat {A}^1_{\varphi}$  the set of irreducible characters of $A^1_{\varphi}$. The pairs $(\varphi,\phi)$ are called \textit{enhanced $L$-parameters}. 
Let $\phi\in\widehat {A}^1_{\varphi}$. Then $\phi$ determines
a character $\zeta_\phi$ of $\rZ_{G^\vee_\sc}$. An enhanced $L$-parameter $(\phi,\varphi)$ is said to be \textit{$G$-relevant} if $\zeta_\phi=\zeta$, where $\zeta$ is as defined in Section \ref{s:Lgrp}.
The set $\Phi_\enh(G)$ of $G^\vee$-conjugacy classes of $G$-relevant enhanced $L$-parameters is expected to parametrize the admissible dual of $G$.

\smallskip

The group $H^1(W_F,\rZ_{G^\vee})$ acts on $\Phi(G)$ by
\begin{equation} \label{eqn:action}
(z\varphi)(w,x):=z'(w)\,\varphi(w,x)\quad\varphi\in\Phi(G), w\in W_F, x\in \SL_2(\CC),
\end{equation}
where $z'\colon W_F\to \rZ_{G^\vee}$ represents $z\in H^1(W_F,\rZ_{G^\vee})$. This extends to an action of $H^1(W_F,\rZ_{G^\vee})$ on $\Phi_\enh(G)$ that does nothing to the enhancements.

A character of $G$ is called \textit{weakly unramified} if it is trivial on the kernel of the Kottwitz homomorphism. Let  $X_\wnr(G)$ denote the group of weakly unramified characters of $G$. There  is  a natural isomorphism
\begin{equation} \label{eqn:dual_unr}
X_\wnr(G)\simeq (\rZ_{G^\vee}^{I_F})_\Frob\,\subset \,H^1(W_F,\rZ_{G^\vee})
\end{equation}
(see \cite[\S 3.3.1]{Hai}). Its identity component is the group $X_\un(G)$ of unramified characters of $G$. Via (\ref{eqn:action}) and (\ref{eqn:dual_unr}), the group $X_\wnr(G)$ acts naturally on $\Phi_\enh(G)$.

Let $\varphi\colon W_F\times\SL_2(\CC) \to {}^L G$ be an $L$-parameter.
We consider the (possibly disconnected) complex reductive group
\begin{equation} \label{eqn:Gvarphi}
\cG_\varphi:=\rZ^1_{G^\vee_\sc}\left(\varphi|_{W_F}\right),
\end{equation}
defined analogously to $\rZ^1_{G^\vee_\sc}(\varphi)$. Denote by $\cG_\varphi^\circ$ its identity component.

We define elements $u_\varphi, s_\varphi \in G^\vee$ by   
\begin{equation} \label{s-u}
(u_\varphi,1) =\varphi (1,\matje{1}{1}{0}{1}) \quad\text{and}\quad
(s_\varphi,\Frob) =\varphi(\Frob,\Id_{\SL_2(\CC)}).
\end{equation} 
Then $u_\varphi\in \cG_\varphi^\circ$.

We recall that by the Jacobson--Morozov Theorem any unipotent element $u$ of 
$\cG_\varphi^\circ$ determines (up to conjugation by
$\rZ_\cG (u)$) a homomorphism of algebraic groups $\SL_2(\CC)\to \cG_\varphi^\circ$ taking the value $u$ at $\matje{1}{1}{0}{1}$. Hence any enhanced $L$-parameter $(\varphi,\phi)$ is completely determined, up to $G^\vee$-conjugacy, by $\varphi\vert_{W_F}$, $u_\varphi$ and $\phi$. More precisely, the map
\begin{equation} \label{eqn:descrip}
(\varphi,\phi)\mapsto (\varphi\vert_{W_F},u_\varphi,\phi)
\end{equation}
provides a bijection between $\Phi_\enh(G)$ and the set of $G^\vee$-conjugacy classes of triples $(\varphi\vert_{W_F},u_\varphi,\phi)$.

We define an action of $G^\vee_\sc$ on $G^\vee$ by setting 
\[h\cdot g:=h'gh^{\prime -1}\quad \text{for $h\in G^\vee_\sc$ and $g\in G^\vee$, where $p(h)=h'\rZ_{G^\vee}$.}\]
It induces an action of $G^\vee_\sc$ on ${}^LG$ and we denote by $\rZ_{G^\vee_\sc}(\varphi)$ the stabilizer in $G^\vee_\sc$ of $\varphi(W_F')$ for this action.

On the other hand, the inclusion $\rZ_{G^\vee_\sc}^1(\varphi)\hookrightarrow \rZ_{G^\vee_\sc}^1(\varphi\vert_{W_F})\cap \rZ_{G^\vee_\sc}(u_\varphi)$ induces a group isomorphism
\begin{equation} \label{eqn:A1}
A^1_\varphi\isom \pi_0\left(\rZ_{G^\vee_\sc}^1(\varphi\vert_{W_F})\cap \rZ_{G^\vee_\sc}(u_\varphi)\right).
\end{equation}
As observed in \cite[(92)]{AMS1}, another way to formulate (\ref{eqn:A1}) is
 \begin{equation} \label{eqn:same}
 A^1_\varphi\simeq A_{\cG_\varphi}(u_\varphi):=\rZ_{\cG_\varphi}(u_\varphi)/\rZ_{\cG_\varphi}(u_\varphi)^\circ.
 \end{equation}

The $L$-parameter $\varphi$  is is called 
\begin{itemize}
\item \textit{discrete} if  there is no proper $W_F$-stable Levi subgroup $L^\vee \subset G^\vee$ such that $\varphi (W_F') \subset  L^\vee \rtimes W_F$.
\item \textit{bounded} if $s_\varphi$ belongs to a bounded subgroup of $G^\vee$.
\end{itemize}
We say that  $(\varphi,\phi)\in\Phi_\enh(G)$ is \textit{cuspidal} if $\varphi$ is discrete and $(u_\varphi,\phi)$ is a {cuspidal pair} for $\cG_\varphi$ (as defined in Section \ref{sec:GSC}). 
The set of $G$-relevant cuspidal (respectively discrete, bounded) enhanced $L$-parameters is expected to correspond to the set of {\it supercuspidal} (respectively, {\it essentially square-integrable, tempered}) irreducible smooth $G$-representations \cite[\S~6]{AMS1}.
 
\subsection{Inertial classes} \label{subsec:inertial} For $L$ a Levi subgroup of $G$ and $g\in G^\vee$, the group $g L^\vee g^{-1}$  is not necessarily $W_F$-stable, so the group $G^\vee$ need not act on pairs of the form $({}^LL,(\varphi_\cus,\phi_\cus))$ with $(\varphi_\cus,\phi_\cus)$ a cuspidal enhanced $L$-parameter for $L$. In order to deal with this, as in \cite[Definition~7.1]{AMS1}, we will have to consider all the pairs $(\rZ_{{}^LG}(\cT),(\varphi_\cus,\phi_\cus))$ of the following form: 
\begin{itemize}
\item $\cT$ is a torus of $G^\vee$ such that the projection $\rZ_{{}^LG}(\cT)\to W_F$ is surjective.
\item $\varphi_\cus\colon W'_F\to \rZ_{{}^L G}(\cT)$ satisfies the requirements in the definition of an $L$-parameter.
\item Let $\cL =G^\vee\cap \rZ_{{}^LG}(\cT)$, and let $\cL_\sc$ be the simply connected cover of the derived group of $\cL$. Then $\phi_\cus$ is an irreducible representation of $\pi_0(\rZ_{\cL_\sc}^1(\varphi))$ such that $(u_{\varphi_\cus},\phi)$ is a cuspidal pair for $\rZ_{\cL_\sc}^1(\varphi_\cus\vert_{W_F}))$ and $\phi_\cus$ is $G$-relevant as defined in \cite[Definition~7.2]{AMS1}
(that is $\zeta_\phi=\zeta$ on $\cL_{\sc}\cap\rZ_{G^\vee_\sc}^{W_F}$ and $\phi=1$ on $\cL_\sc\cap \rZ_{\cL_c}^\circ$, where $\cL_c$ denotes the preimage of $\cL$ under $G^\vee_\sc\to G^\vee$).
\end{itemize}

Fix such a pair $(\rZ_{{}^LG}(\cT),(\varphi_\cus,\phi_\cus))$. The group 
\begin{equation} \label{eqn:gen_unr}
X_\un(\rZ_{{}^LG}(\cT))
:=\left(
\rZ_{(G^\vee\rtimes I_F)\cap \rZ_{{}^LG}(\cT)}
\right)_{\Frob}^\circ
\end{equation} 
plays the role of unramified characters for $\rZ_{{}^LG}(\cT)$. It acts on the enhanced $L$-parameters $(\varphi_\cus,\phi_\cus)$ (see \cite[(110) and (111)]{AMS1}) and we denote by $X_\un(\rZ_{{}^LG}(\cT))\cdot(\varphi_\cus,\phi_\cus)$ the orbit of $(\varphi_\cus,\phi_\cus)$.

We denote by $\fs^\vee$ the $G^\vee$-conjugacy class of  
$(\rZ_{{}^LG}(\cT),X_\un(\rZ_{{}^LG}(\cT))\cdot(\varphi_\cus,\phi_\cus))$.
We write
\[\fs^\vee=\fs^\vee_{G}=[\rZ_{{}^LG}(\cT),(\varphi_\cus,\phi_\cus)]_{G^\vee},\]  and call $\fs^\vee$ an \emph{inertial class} for $\Phi_\enh(G)$. We denote by $\fB^\vee(G)$ the set of all such $\fs^\vee$. 

Note that 
there exists a $W_F$-stable Levi subgroup $L^\vee$ of $G^\vee$ such that $\rZ_{{}^LG}(\cT)$ is $G^\vee$-conjugate to $L^\vee\rtimes W_F$ and $\cL =G^\vee\cap \rZ_{{}^LG}(\cT)$ is $G^\vee$-conjugate to $L^\vee$. Conversely, every $G^\vee$-conjugate of this $L^\vee\rtimes W_F$ is of the form $\rZ_{{}^LG}(\cT)$ for a torus $\cT$ as above (see \cite[Lemma~6.2]{AMS1}).

We write
\begin{equation} \label{eqn:sLdual}
\fs^\vee_L=(\rZ_{{}^LG}(\cT),X_\un(\rZ_{{}^LG}(\cT))\cdot(\varphi_\cus,\phi_\cus)).
\end{equation}
We will consider the groups
\begin{equation}
W_{\fs^\vee}=\rN_{G^\vee} (\fs^\vee_L) / L^\vee\quad\text{and} \quad J_{\varphi_\cus}:=\rZ_{G^\vee}(\varphi_\cus(I_F)).\end{equation}
The group $J_{\varphi_\cus}$ is a complex (possibly disconnected) reductive group.  
Define $R(J_{\varphi_\cus}^\circ,\cT)$ as the set of $\alpha \in X^*(\cT) \setminus \{0\}$ 
that appear in the adjoint action of $\cT$ on the Lie algebra of $J_{\varphi_\cus}^\circ$.
It is a root system (see \cite[Proposition~3.9]{AMS3}).

We set
$W_{\fs^\vee}^\circ:=\rN_{J_{\varphi_\cus}^\circ} (\cT) / \rZ_{J^\circ_{\varphi_\cus}} (\cT)$, where $W_{\fs^\vee}^\circ$ is  the Weyl group of $R(J_{\varphi_\cus}^\circ,\cT)$.
Let $R^+(J^\circ_{\varphi_\cus},\cT)$ be the positive system defined by a parabolic subgroup 
$P_{\varphi_\cus}^\circ \subset J_{\varphi_\cus}^\circ$ with Levi factor $(L_{\varphi_\cus}^\vee)^\circ$. Two such parabolic subgroups $P_{\varphi_\cus}^\circ$ are $J_{\varphi_\cus}^\circ$-conjugate, so the choice is inessential.

Since $W_{\fs^\vee}^\circ$ acts simply transitively on the collection of positive
systems for $R(J_{\varphi_\cus}^\circ,\cT)$, we obtain a semi-direct factorization
\begin{align*}
& W_{\fs^\vee} = W_{\fs^\vee}^\circ \rtimes \fR_{\fs^\vee} , \\
& \text{where }\fR_{\fs^\vee} = 
\{ w \in W_{\fs^\vee} \mid w \cdot R^+ (J_{\varphi_\cus}^\circ,\cT) = R^+ (J_{\varphi_\cus}^\circ,\cT) \}.
\end{align*}

\begin{defn} \label{defn:GSC-Phi}
Let $\bnu_\cus\colon\Phi_\enh(G)\to \fB^\vee(G)$ be the map defined by
\[\bnu_\cus(\varphi,\phi)=[\rZ_{{}^LG}(\rZ_{\cM_\varphi}^\circ),\varphi\vert_{W_F},u_\cus,\phi_\cus]_{G^\vee},\]
where $(\varphi\vert_{W_F},u,\phi)_{G^\vee}$ is the image of $(\varphi,\phi)_{G^\vee}$ via the bijection (\ref{eqn:descrip}),  $(u_\cus,\phi_\cus)$ corresponds to $(U_\cus,\cE_\cus)\in\bI^{\cM_\varphi}_\cus$ and $(\cM_\varphi,U_\cus,\cE_\cus):=\nu_\cus^{\cG_\varphi}(U,\cE)$ is the image under the map $\nu_\cus^{\cG_\varphi}$ from Definition~\ref{defn:nu_cus} of the pair $(U,\cE)\in\bI^{\cG_\varphi}$ associated with $(u,\phi)$.
\end{defn}

We have the following decomposition (see \cite[(115)]{AMS1}): 
\begin{equation} \label{eqn:dec_Phi}
\Phi_\enh(G)=\bigsqcup_{\fs^\vee\in\fB^\vee(G)}\Phi_\enh(G)^{\fs^\vee},\quad\text{where}\quad 
\Phi_\enh(G)^{\fs^\vee}:=\bnu^{-1}_\cus(\fs^\vee).
\end{equation}

\smallskip

Let $\Irr(G)$ be the set of isomorphism classes of irreducible smooth $G$-representations. 
For $L$ a Levi subgroup of $G$, we denote by $\Irr_\cusp (L)$  the set of isomorphism classes of supercuspidal irreducible smooth $L$-representations. 

Let $\sigma\in\Irr_\cusp (L)$.  We call $(L,\sigma)$ a \textit{supercuspidal pair}, and we consider such pairs up to {\em inertial equivalence}: this is
the equivalence relation generated by
\begin{itemize}
\item unramified twists, $(L,\sigma) \sim (L,\sigma \otimes \chi)$ for $\chi \in 
X_\un (L)$,
\item $G$-conjugation, $(L,\sigma) \sim (g L g^{-1},g \cdot \sigma)$ for $g \in G$.
\end{itemize}
We denote the set of all inertial equivalence classes for $G$ by $\fB(G)$ and a typical inertial equivalence class by $\fs:= [L,\sigma]_G$. 

In \cite{Be},  Bernstein attached to every $\fs\in\fB(G)$ a block $\fR(G)^\fs$ in the category $\fR(G)$ of smooth $G$-representations as follows. 
Denote by $I_P^G$ the normalized parabolic induction functor, where $P$ is a parabolic subgroup of $G$ with Levi subgroup $L$. If
$\pi\in \Irr(G)$ is a constituent of $I_P^G (\tau)$ for some $\sigma \in \Irr(L)$ such that $[L,\sigma]_G=\fs$, then $\fs$
is called the \emph{inertial supercuspidal support} of $\pi$.   
We set
\begin{align*}
& \Irr(G)^\fs: = \{ \pi \in \Irr(G) \,:\,  \text{$\pi$ has inertial supercuspidal support $\fs$}\} , \\
& \fR(G)^\fs: = \{ \pi \in \fR(G) \;:\; \text{ every irreducible constituent of } \pi
\text{ belongs to } \Irr(G)^\fs \} .
\end{align*}

\subsection{Unipotent representations}\label{s:unip}

An irreducible smooth representation $(\pi,V)$ of $G$ is called \textit{unipotent} if there exists a parahoric subgroup $K$ of $G$ such that the subspace $V^{K^+}$ of the vectors in $V$ that are fixed by the pro-unipotent radical $K^+$ of $K$ contains an irreducible unipotent representation of the finite reductive group $\overline K:=K/K^+$. 
We denote by $\Irr_\un(G)$ the set of isomorphism classes of irreducible unipotent $G$-representations. 
 
For the rest of the section we will assume that the quasi-split inner form $G^*$ of $G$ is $F$-split. 
 
 \begin{defn} \label{defn:Phi_e}
An $L$-parameter $\varphi\colon W_F\times\SL_2(\CC) \to {}^L G$ is called \textit{unipotent} if $\varphi(w,1)=(1,w)$ for any element $w$ of the inertia subgroup $I_F$ of $W_F$. 
\end{defn}
Denote by $\Phi_{\un}({}^LG)$ the set of {$G^\vee$-conjugacy classes of} unipotent $L$-parameters $\varphi$ and by $\Phi_{\enh,\un}(G)$ the set of unipotent enhanced $G$-relevant parameters, i.e., the  subset of the $(\varphi,\phi)\in\Phi_{\enh}(G)$ such that $\varphi$ is unipotent.  
We then set 
\[\Phi_{\enh,\un}({}^LG) = G^\vee\backslash\{(\varphi,\phi)\mid \varphi\text{ unipotent}, \phi\in\widehat{A_\varphi^1}\}.\]
Given $\varphi \in \Phi_{\un}({}^LG)$, since $G^*$ is $F$-split, every $\phi \in \widehat{A_\varphi^1}$ is $G'$-relevant for some $G' \in \InnT(G)$. We thus have
\[\Phi_{\enh,\un}({}^LG) =
\bigsqcup_{G'\in\InnT(G)}\Phi_{\enh,\un}(G').
\]
Given $\varphi \in \Phi_{\un}({}^LG)$, let $A_\varphi$ be the component group of $Z_{G^\vee}(\varphi)$. We then set 
\[\Phi^p_{\enh,\un}({}^LG)=G^\vee\backslash\{(\varphi,\phi)\mid \varphi\text{ unipotent}, \phi\in\widehat{A_\varphi}\}.
\]

The set  $\Phi_{\enh,\un}(G)$ is known to parametrize $\Irr_\un(G)$: such a parametrization was defined by Lusztig in \cite{LuI,LuII} in the case when $\bG$ is simple adjoint, and extended by Solleveld in \cite{So1,So2} to the case when $G$ is arbitrary. In the case when $G=\GL_n(F)$ or $\SL_n(F)$, it also follows from \cite{HS} and \cite{ABPS}. 
This parametrization induces a bijection:
\begin{equation} \label{eqn:corr}
\mathsf{LLC}\colon \Phi_{\enh,\un}({}^LG)\longleftrightarrow \bigsqcup_{G'\in \InnT(G)}\Irr_\un(G').
\end{equation}
This correspondence sends cuspidal (respectively, discrete, bounded) parameters to supercuspidal (respectively, essentially square-integrable, tempered) irreducible unipotent representations.

\smallskip


Let $x\in G^\vee$ with Jordan decomposition $x=su$. There is an unipotent $L$-parameter $\varphi$ (unique up to $G^\vee$-conjugation) such that $u=u_\varphi$ and $s=s_\varphi$.
We set 
\begin{equation} \label{eqn:defs}
\cG_s=\rZ_{G^\vee_\sc}^1(\varphi|_{W_F}).
\end{equation}
Notice that $\varphi|_{W_F}$ depends only on $s$, which explains the notation. By (\ref{eqn:same}),  \[A^1_\varphi\cong A_{\cG_s}(u).\] We set 
\[\Phi_{\enh,\un}({}^LG,s)=G^\vee\backslash\{(\varphi',\phi)\in \Phi_{\enh,\un}({}^LG)\mid \varphi'(\Frob,1)=(s',\Frob),\ s'\in G^\vee\cdot s\}.
\]
Then
\begin{equation}
\begin{aligned}
\Phi_{\enh,\un}({}^LG,s)&=\rZ_{G^\vee}(\varphi|_{W_F})\text{-orbits in }\{(u',\phi)\mid u'\in \cG_s^\circ\text{ unipotent},\ \phi\in \widehat {A_{\cG_s}(u')}\}\\
&=\cG_s\text{-orbits in }\{(u',\phi)\mid u'\in \cG_s^\circ\text{ unipotent},\ \phi\in \widehat {A_{\cG_s}(u')}\}.
\end{aligned}
\end{equation}
The second equality follows from the fact that conjugation of unipotent elements is insensitive to isogenies. 
This allows us to rephrase the unipotent local Langlands correspondence as follows.  Let $\cC(\cG)$, $\cC(\cG)_{\mathsf{ss}}$, $\cC(\cG)_\un$ denote the set of conjugacy classes, respectively semisimple, unipotent conjugacy classes in a complex group $\cG$. Let $R_\un (G')$ be the $\CC$-span of $ \Irr_\un (G')$. Then (\ref{eqn:corr})  can be written as the bijection 
\[\mathsf{LLC}_{\un}:\ \bigsqcup_{s\in \cC(G^\vee)_{\mathsf{ss}}} \bigsqcup_{u\in \cC(\cG_s)_\un} \widehat{A_{\cG_s}(u)}\longleftrightarrow  \bigsqcup_{G'\in \InnT(G)}\Irr_\un(G'),
\]
which induces a linear isomorphism
\begin{equation}\label{e:LLC-iso}
\begin{aligned}
\mathsf{LLC}_{\un}\colon R(\Phi_{\enh,\un}({}^LG))&:=\bigoplus_{s\in \cC(G^\vee)_{\mathsf{ss}}}\bigoplus_{u\in \cC(\cG_s)_\un}R(A_{\cG_s}(u)) \longrightarrow \bigoplus_{G'\in\InnT(G)} R_\un(G').
\end{aligned}
\end{equation}
If we instead consider pure inner twists, then we need to replace the group $\cG_s$ by the group
\begin{equation} \label{eqn:Gsp} 
\cG_s^p=\rZ_{G^\vee}(\varphi|_{W_F}),
\end{equation}
and the correspondence becomes
\begin{equation}\label{e:LLC-pure}
\mathsf{LLC}^p_{\un}\colon R(\Phi_{\enh,\un}^p({}^LG)):=\bigoplus_{s\in \cC(G^\vee)_{\mathsf{ss}}}\bigoplus_{u\in \cC(\cG_s^p)_\un}R(A_{\cG_s^p}(u)) \longrightarrow \bigoplus_{G'\in\InnT^p(G)} R_\un(G').
\end{equation}
Given $\phi \in R(A_{\cG_s^p}(u))$, write $\pi(s, u, \phi)$ for the image of $\phi$ under $\mathsf{LLC}^p_{\un}$.

\begin{rem}
Note that the existing local Langlands correspondence for unipotent representations is not entirely canonical (see the discussion in \cite[Introduction]{So1}) but for the rest of the paper, we fix a map $\mathsf{LLC}_{\un}^p$ as above satisfying the usual properties (described, for example, in \cite[Theorem 1]{So1}).
\end{rem}

\begin{ex} For $G=\SL_n(F)$, recall that there is a one-to-one correspondence between
$\InnT(\SL_n(F))$ and $\mathbb Z/n\mathbb Z$, where the inner twists are $\SL_m(D_{d,r/m})$, $m=\gcd(r,n)$, $r\in \ZZ/n\ZZ$.

The dual Langlands group is $G^\vee=\PGL_n(\mathbb C)$. The correspondence (\ref{eqn:corr}) takes the form:
\begin{equation}
\bigsqcup_{r\in \mathbb Z/n\mathbb Z} \Irr_\un(\SL_m(D_{d,r/m}))\longleftrightarrow \PGL_n(\mathbb C)\backslash\{(x,\phi)\,|\,x\in \PGL_n(\mathbb C), \phi\in\widehat {A^1_{x}}\},
\end{equation}
in particular
\[\Irr_\un(\SL_n(F))\longleftrightarrow \PGL_n(\mathbb C)\backslash\{(x,\phi)\,|\,x\in \PGL_n(\mathbb C), \phi\in\widehat {A_{x}}\}.
\]
In this case, $G^\vee_\sc=\SL_n(\mathbb C)$ and  $\rZ_{G^\vee_\sc}=C_n$. The irreducible central characters are therefore $\widehat{\rZ_{\SL_n(\mathbb C)}}=\{\zeta_r\mid r\in \mathbb Z/n\mathbb Z\}.$ A Langlands parameter $(x,\phi)$ parametrizes an irreducible unipotent representation of $\SL_m(D_{d,r/m})$ if and only if $\zeta_\phi=\zeta_r$. In particular, the unipotent representations of $\SL_n(F)$ correspond to central characters $\zeta_0=1$.

Moreover, for $x\in \PGL_n(\mathbb C)$, $A^1_x$ is the group of components of $\rZ^1_{\SL_n(\CC)}(x)=\{g\in \SL_n(\mathbb C)\mid g x g^{-1}=x\}$.
\end{ex}

\begin{rem}
Note the Langlands parameters we call unipotent here are also known as \textit{unramified} Langlands parameters (cf. \cite{Vo}, \cite{So1}).  
\end{rem}

\section{ Lusztig's nonabelian Fourier transform for finite groups}\label{s:Lusztig}

We recall the definition of the nonabelian Fourier transform \cite{Lubook}. For the background material, we follow \cite{Lubook,GM,DM}. 

\subsection{Fourier transforms}\label{s:finiteFT} For a finite group $\Gamma$, define $\cM(\Gamma)$ to be the set 
\begin{equation}
\{(x,\sigma)\mid x\in \Gamma,~\sigma\in \widehat{\rZ_\Gamma(x)}\},
\end{equation}
modulo the equivalence relation given by conjugation by $\Gamma$: $g\cdot (x,\sigma)=(gxg^{-1},\sigma^g)$, where $\sigma^g(y)=\sigma(g^{-1}yg)$ for all $g\in \Gamma$, $y\in \rZ_\Gamma(g x g^{-1})$. 
Define also
\begin{equation}
\cY(\Gamma)= \{(y,z)\in \Gamma\times\Gamma\mid yz=zy\}.
\end{equation}
Write $\Gamma \setminus \cY(\Gamma)$ for the set of $\Gamma$-orbits on $\cY(\Gamma)$.
Let $\mathsf{Sh}^\Gamma(\Gamma)$ be the category of $\Gamma$-equivariant coherent sheaves of $\Gamma$ ($\Gamma$ acting on itself by conjugation). The irreducible objects in $\mathsf{Sh}^\Gamma(\Gamma)$ are parametrized by $\cM(\Gamma)$: for every pair $(x,\sigma)\in \cM(\Gamma)$, let $\cV({(x,\sigma)})=\Gamma\times_{\rZ_\Gamma(x)}\sigma$ be the corresponding irreducible $\Gamma$-equivariant sheaf. This means that there is a natural isomorphism $\CC[\cM(\Gamma)]\cong K(\mathsf{Sh}^\Gamma(\Gamma))_\CC$, where $K(\;)_\CC$ is the complexification of the $K$-group. Moreover, there is an isomorphism
\[\kappa\colon \CC[\cM(\Gamma)]\cong K(\mathsf{Sh}^\Gamma(\Gamma))_\CC\to \CC[\Gamma\setminus\cY(\Gamma)],\quad \cV\mapsto ((y,z)\mapsto \tr(z,\cV|_y)).
\]
Lusztig \cite{Lubook} defined a pairing on $\cM(\Gamma)$:
\begin{equation}
\{(x,\sigma),(y,\tau)\}=\frac 1{|\rZ_\Gamma(x)| |\rZ_\Gamma(y)|} \sum_{\substack{g\in \Gamma\\xgyg^{-1}=gyg^{-1}x}} \sigma(gyg^{-1})\tau(g^{-1}x^{-1}g),
\end{equation}
which extends to a Hermitian pairing on $\bC[\cM(\Gamma)]$. He also defined a linear map, the Fourier transform for $\Gamma$,
\begin{equation}\label{e:finite-flipM}
\FT_\Gamma\colon \CC[\cM(\Gamma)]\to \CC[\cM(\Gamma)],\ \FT_\Gamma(f)(x,\sigma)=\sum_{(y,\tau)\in \cM(\Gamma)}\{(x,\sigma),(y,\tau)\} f(y,\tau).
\end{equation}
See Lemma \ref{l:finite-flip} for the interpretation of $\FT_\Gamma$ in terms of $\cY(\Gamma)$.

\medskip

Now consider the following generalization. Suppose $\widetilde \Gamma=\Gamma\rtimes\langle \alpha\rangle$, where $\alpha$ has order $c$. Set $\Gamma'=\Gamma\alpha\subset \widetilde\Gamma$. As in \cite[\S4.16]{Lubook}, define two sets $\cM=\cM(\Gamma\unlhd \widetilde\Gamma)$ and $\overline\cM=\overline\cM(\Gamma\unlhd \widetilde\Gamma)$ as follows:
\begin{equation}
\begin{aligned}
\cM&=\{(x,\sigma)\mid x\in\Gamma\text{ such that }\rZ_{\widetilde\Gamma}(x)\cap \Gamma'\neq\emptyset,\ \sigma\in \widehat{\rZ_{\widetilde\Gamma}(x)} \text{ with }\sigma|_{\rZ_\Gamma(x)}\text{ irreducible}\},\\
\overline \cM&=\{(x,\bar\sigma)\mid x\in \Gamma',\ \bar\sigma\in \widehat{\rZ_\Gamma(x)}\},
\text{ in each case modulo the equivalence relation}\\&\text {given by conjugation by $\widetilde\Gamma$}.
\end{aligned}
\end{equation}
 In addition, the cyclic group $\langle\al\rangle$ acts on $\cM$ by twists in the second entry of the pair $(x,\sigma)$. Denote by $\sim_c$ the corresponding equivalence relation.

The set $\cM$ is a subset of $\cM(\widetilde\Gamma)$. Given $(x, \bar\sigma) \in \overline\cM$, we have that $(x, \sigma) \in \cM(\widetilde\Gamma)$ for any extension $\sigma$ of $\bar\sigma$ to $\rZ_{\widetilde\Gamma}(x)$. Thus the pairing $\{~,~\}$ on $\cM(\widetilde \Gamma)$ induces a pairing
\begin{equation}
\{~,~\}:\overline \cM\times \cM\to \CC,\quad \{(x,\bar\sigma),(y,\tau)\}:=c  \{(x,\sigma),(y,\tau)\},
\end{equation}
for any fixed extension $\sigma$ of $\bar\sigma$ to $\rZ_{\widetilde\Gamma}(x)$. 

Let $\cP=\cP(\Gamma\unlhd \widetilde \Gamma)$ and $\overline\cP=\overline\cP(\Gamma\unlhd \widetilde \Gamma)$ be the spaces of functions on $\cM(\widetilde \Gamma)$ with support in $\cM$ and $\overline\cM$, respectively. The operator \cite[(4.16.1)]{Lubook} (see also \cite[\S4.2.14]{GM})
\begin{equation}
\FT_{\Gamma\unlhd\widetilde\Gamma}: \cP\to \overline\cP,\ \FT_{\Gamma\unlhd\widetilde\Gamma}f(x,\bar\sigma)=\sum_{(y,\tau)\in \cM/\sim_c}\{(x,\bar\sigma),(y,\tau)\} f(y,\tau)
\end{equation} 
is an isomorphism with inverse $\FT_{\Gamma\unlhd\widetilde\Gamma}^{-1}f(y,\tau)=\sum_{(x,\bar\sigma)\in\overline\cM}\{(x,\bar\sigma),(y,\tau)\} f(x,\bar\sigma)$.

\subsection{Families of Weyl group representations}  Let $W$ be a finite Weyl group with the set of simple generators $S$. The partition of $\widehat W$ into families is defined in \cite[\S4.2]{Lubook} as follows. Let $\sgn$ denote the sign character of $W$. If $W=\{1\}$, there is only one family consisting of the trivial representation. Otherwise, assume that the families have been defined for all proper parabolic subgroups of $W$. Then $\mu,\mu'\in \widehat W$ belong to the same family of $W$ if there exists a sequence $\mu=\mu_0,\mu_1,\dots,\mu_m=\mu'$, $\mu_i\in \widehat W$, such that for each $i$ there exists a parabolic subgroup $W_i\subsetneq W$ and $\mu_i',\mu_i''\in \widehat W_i$ in the same family of $W_i$ such that either
\[\langle \mu_i',\mu_{i-1}\rangle_{W_i}\neq 0,\ a_{\mu_i'}=a_{\mu_{i-1}},\quad \langle \mu_i'',\mu_{i}\rangle_{W_i}\neq 0,\ a_{\mu_i''}=a_{\mu_{i}},
\]
or
\[\langle \mu_i',\mu_{i-1}\otimes\sgn\rangle_{W_i}\neq 0,\ a_{\mu_i'}=a_{\mu_{i-1}\otimes\sgn},\quad \langle \mu_i'',\mu_{i}\otimes\sgn\rangle_{W_i}\neq 0,\ a_{\mu_i''}=a_{\mu_{i}\otimes\sgn}.
\]
Here $a_\mu$ is the $a$-invariant of $\mu$ defined in \cite[\S4.1]{Lubook}. It follows from the definition that if $\cF\subset \widehat W$ is a family, then so is $\cF\otimes\sgn$ and the families for $W_1\times W_2$ are $\cF_1\boxtimes\cF_2$, where $\cF_i$ is a family for $W_i$, $i=1,2$. 

\medskip

Suppose in addition that we have a Coxeter group automorphism $\sigma\colon W\to W$, i.e. $\sigma \in \Aut(W)$  such that $\sigma(S)=S$. Such an automorphism is called ordinary if, on each irreducible component of $W$, it is not the nontrivial graph automorphism of type $B_2$, $G_2$, or $F_4$. The automorphism $\sigma$ acts on $\widehat W$ and it permutes the families $\cF$. An important observation \cite[\S4.17]{Lubook} is that if $\sigma$ is ordinary and $\cF$ is $\sigma$-stable, then every element of $\cF$ is $\sigma$-stable.

\subsection{Families of unipotent representations}\label{s:families-connected}  Let $\qG$ be a connected reductive algebraic group over $\overline{\mathbb F}_q$ with a Frobenius map $\Fr \colon\qG\to \qG$ such that there exists a maximal torus $\qT_0$ with the property $\Fr(t)=t^q$, for all $t\in \qT_0$. Let $W=N_{\qG}(\qT_0)/{\qT_0}$ be the Weyl group.  Recall that an irreducible representation $\rho\in \Irr~\qG^\Fr$ is called \textit{unipotent} if $\langle \rho, R_{\qT}^{\qG}(1)\rangle_{\qG^\Fr}\neq 0$ for some $\Fr$-stable maximal torus $\qT$ of $\qG$. Here $R_{\qT}^{\qG}$ is Deligne--Lusztig induction \cite[\S7.8]{DL}. Let $\Irr_\un\qG^\Fr$ denote the set of irreducible unipotent $\qG^\Fr$-representations. By the results of Lusztig, the classification of $\Irr_{\un}\qG^\Fr$ is reduced to the case when $\qG$ is adjoint simple (see for example the exposition in \cite[Remark 4.2.1]{GM}). More precisely, if $\pi\colon\qG\to \qG_{\ad}$ is the surjective homomorphism with central kernel ($\qG_{\ad}$ is the semisimple adjoint group isogeneous to $\qG/\rZ_{\qG}$), there exists a Frobenius map $\Fr_{\ad}$ such that $\Fr_{\ad}\circ \pi=\pi\circ \Fr$ such that the resulting group homomorphism $\pi\colon \qG^\Fr\to \qG_{\ad}^{\Fr_{\ad}}$ induces a bijection
\[\Irr_\un\qG^\Fr\leftrightarrow \Irr_\un\qG^{\Fr_{\ad}}_{\ad}.
\]
Furthermore, write $\qG_{\ad}=\qG_1\times\dots\times\qG_r$ for the decomposition into factors such that each $\qG_i$ is semisimple adjoint, $\Fr_{\ad}$-stable, and a direct product of simple algebraic groups that are cyclically permuted by $\Fr_{\ad}$. Let $\qH_i$ be one of the simple factors in $\qG_i$: if $h_i$ is the number of copies of $\qH_i$, then $\Fr^{h_i}_{\ad}$ preserves $\qH_i$. Denote by $\Fr_i$ the restriction to $\qH_i$. Then 
\[\Irr_\un\qG^{\Fr_{\ad}}_{\ad}\cong \prod_{i=1}^r \Irr_\un \qH_i^{\Fr_i}.
\]
The Frobenius map $\Fr$ induces a Coxeter group automorphism $\sigma$ of $W$. Define a graph with vertices $\Irr_\un\qG^\Fr$ as follows: $\rho_1,\rho_2\in\Irr_\un\qG^\Fr$ are joined by an edge if and only if there is $\sigma$-stable $\mu\in \widehat W$ such that $\langle \rho_i,R_{\widetilde\mu}\rangle_{\qG^\Fr}\neq 0$ for $i=1,2$ where $R_{\widetilde \mu}$ is the almost character associated to a fixed extension $\widetilde\mu$ of $\mu$ to $\widetilde W=W\rtimes\langle\sigma\rangle$ as defined in \cite[(3.7.1)]{Lubook}.  Each connected component of this graph is called a \emph{family} in $\Irr_\un\qG^\Fr$. 
One can define an equivalence relation on the set $\widehat W^\sigma$ of $\sigma$-stable irreducible $W$-representations: $\mu$ and $\mu'$ are equivalent if $R_{\widetilde \mu}$ and $R_{\widetilde\mu'}$ have unipotent constituents in the same family. By \cite{Lubook} (see also \cite[Proposition 4.2.3]{GM}), the equivalence classes are the same as the $\sigma$-stable families in $\widehat W$, when $\sigma$ is ordinary.

To each family $\cU\subset \Irr_\un\qG^\Fr$ corresponding to the $\sigma$-stable family $\cF\subset \widehat W^\sigma$, Lusztig \cite[\S4]{Lubook} attached  finite groups $\Gamma_\cU\unlhd\widetilde\Gamma_\cU$ such that $\widetilde\Gamma_\cU=\Gamma_\cU\langle\sigma\rangle$, a bijection
\begin{equation}\label{e:L-class}
\cU\longleftrightarrow \overline \cM(\Gamma_\cU\unlhd\widetilde\Gamma_\cU),\  \rho\mapsto \bar x_\rho,
\end{equation}
scalars $\Delta(\bar x_\rho)\in \{\pm 1\}$ \cite[\S6.7]{Lubook}, 
and an injection
\begin{equation}
\cF\longrightarrow \cM(\Gamma_\cU\unlhd\widetilde\Gamma_\cU),\ \mu\mapsto x_\mu,
\end{equation}
such that, when $\sigma$ is ordinary, \cite[Theorem 4.23]{Lubook} says that
\begin{equation}
\langle \rho,R_{\wti \mu}\rangle_{\qG^\Fr}=\Delta(\bar x_\rho)\{\bar x_\rho,x_\mu\}.
\end{equation}
Define the \emph{unipotent almost characters} of $\qG^\Fr$ to be the set of orthonormal class functions
\begin{equation}
R_x=\sum_{\rho\in \cU}\Delta(\bar x_\rho)\{\bar x_\rho,x\}\rho,\quad x\in \cM(\Gamma_\cU\unlhd\widetilde\Gamma_\cU).
\end{equation}
Hence the unipotent nonabelian Fourier transform of $\qG^\Fr$
\begin{equation}
\FT_{\qG^\Fr}:=\bigoplus_{\cU\subset \Irr_\un\qG^\Fr}\FT_{\Gamma_\cU\unlhd\widetilde\Gamma_\cU}
\end{equation}
gives the change of bases matrix, up to the signs $\Delta(\bar x_\rho)$, between irreducible unipotent characters and almost characters.

\smallskip

Assume $\mathbb G$ is $\Fr$-split, so that $\sigma$ is trivial and a family $\cU$ is parametrized by $\cM(\Gamma_\cU)$. Let $x\in \Gamma=\Gamma_\cU$, and define the virtual combinations of unipotent characters 
\begin{equation}\label{e:pi-U}
\Pi_\cU(x,y)=\sum_{\sigma\in \widehat{\rZ_\Gamma(x)}}\sigma(y^{-1}) \rho_{(x,\sigma)},\text{ where }y\in \rZ_\Gamma(x),
\end{equation}
where $ \rho_{(x,\sigma)}$ is the representation in $\cU$ parametrized by $(x,\sigma)\in\cM(\Gamma_\cU)$.
If $\Gamma$ is abelian (which is often the case), we may also define
\begin{equation}\label{e:pi-U-2}
\Pi_\cU(\sigma,\tau)=\sum_{y\in \Gamma} \tau(y)\rho_{(y,\sigma)},\text{ if } \sigma,\tau\in \widehat{\Gamma}.
\end{equation}

\begin{lem}[(cf. {\cite{DM}})]\label{l:finite-flip} With the notation of (\ref{e:pi-U})-- (\ref{e:pi-U-2}) and $\Gamma=\Gamma_\cU$,
\[\FT_{\Gamma}(\Pi_\cU(x,y))=\Pi_\cU(y,x),\quad \FT_{\Gamma}(\Pi_\cU(\sigma,\tau))=\Pi_\cU(\tau,\sigma),
\]
the latter when $\Gamma$ is abelian.
\end{lem}

\begin{proof}
We verify the first formula. The second one is analogous (or it follows by change of bases.) Denote by $C_y$ the conjugacy class of $y$ in $\rZ_\Gamma(x)$.
\begin{align*}
\FT_{\Gamma_\cU}(\Pi_\cU(x,y))&=\sum_\sigma \sigma(y^{-1})\sum_{(z,\tau)}\{(x,\sigma),(z,\tau)\} \rho_{(z,\tau)}\\
&=\sum_\sigma \sigma(y^{-1}) \sum_{(z,\tau)}\frac 1{|\rZ_\Gamma(x)||\rZ_\Gamma(z)|}\sum_{g\in\Gamma,~gzg^{-1}\in \rZ_\Gamma(x)} \sigma(gzg^{-1})\tau(g^{-1}x^{-1}g)\rho_{(z,\tau)}\\
&=\sum_{(z,\tau)}\frac 1{|\rZ_\Gamma(z)|} \left(\sum_{g\in\Gamma,~gzg^{-1}\in \rZ_\Gamma(x)}\frac 1{|\rZ_\Gamma(x)|} \sum_\sigma \sigma(y^{-1})\sigma(gzg^{-1})\right) \tau(g^{-1}x^{-1}g) \rho_{(z,\tau)}\\
&=\sum_{g,z\in \Gamma,~gzg^{-1}\in C_y} \frac 1{|C_y|} \frac 1{|\rZ_\Gamma(z)|} \sum_{\tau\in \widehat{\rZ_\Gamma(z)}} \tau(g^{-1}x^{-1} g)\rho_{(z,\tau)}\quad \text{(by character orthogonality)}\\
&=\frac 1{|C_y|}\sum_{y'\in C_y}\sum_{\tau\in \widehat{\rZ_\Gamma(g^{-1}y' g)}} \tau(g^{-1}x^{-1} g)\rho_{(g^{-1}y'g,\tau)}\quad\text{(setting $y'=gzg^{-1}$)}\\
&=\frac 1{|C_y|}\sum_{y'\in C_y} \Pi_\cU(g^{-1}y'g,g^{-1}xg)=\Pi_\cU(y,x),
\end{align*}
where we used column orthogonality of characters and that $(y,x)$ is $\Gamma$-conjugate to $(g^{-1}y'g,g^{-1}xg)$ when $y'\in C_y$.

\end{proof}

\section{ Disconnected groups over finite fields}\label{s:disconn}
Suppose that $\qG$ is a disconnected reductive group over $\overline {\mathbb F}_p$ with Frobenius map $\Fr\colon\qG\to\qG$ and identity component $\qG^\circ$ such that $A=\qG/\qG^\circ$ is abelian. (In our applications, $\qG/\qG^\circ$ will almost always be a cyclic group.) 
By definition, the irreducible unipotent $\qG^\Fr$-representations $\Irr_\un\qG^\Fr$ are the constituents of all induced representations $\Ind_{{\qG^\circ}^\Fr}^{\qG^\Fr}\rho$, where $\rho\in \Irr_\un{\qG^\circ}^\Fr$. See \cite[Proposition 4.8.19]{GM} for the compatibility with the definition in terms of the appropriate version of $R_{\qT}^{\qG}(1)$. The parametrization of $\Irr_\un\qG^\Fr$ follows from that of $ \Irr_\un{\qG^\circ}^\Fr$ via Mackey induction using the explicit results for simple groups, e.g. \cite[Theorem 4.5.11 and 4.5.12]{GM}.

We are interested in studying the irreducible unipotent representations for groups $\qG^\Fr$ that are related via the structure theory of $p$-adic groups.

Let $\qG$ be a reductive algebraic group  over $\overline{\mathbb{F}}_p$ with identity component $\qG^\circ$ and such that $\qG/\qG^\circ=A$ is a finite abelian group.  Let $\Fr_0$ be a Frobenius map on $\qG$ and assume that ${\qG}^{\Fr_0}$ is split. Given $a\in A$, conjugation by $a$ defines an outer automorphism of $\qG^\circ$, which induces an isomorphism, call it $\sigma_a$, of the based root datum of $\qG^\circ$. For every $a$, define the Frobenius automorphism 
$\Fr_a=\Fr_0\circ \sigma_a$.
By analogy with Section \ref{s:inner}, we write
\[\InnT^p\qG=\{\qG^{\Fr_a}\mid a\in A\}.\]
and call this the set of pure inner twists of $\qG^{\Fr_0}$. Just as in the $p$-adic case,
this set is in one-to-one correspondence with the first Galois cohomology group
\[
\InnT^p\qG\leftrightarrow H^1(\mathbb F_q,\qG)\cong H^1(\mathbb F_q,\qG/\qG^\circ)=H^1(\mathbb F_q,A)\cong A,
\]
using the fact that $H^1(\mathbb F_q,\qG^\circ)=0$ by Lang's Theorem (see \cite[III.\S2, Corollary 3]{Se}, for example), and the assumption that $\Fr_0$ acts trivially on $A$.

By (\ref{e:L-class}), every unipotent family $\cU\subset \Irr_\un({\qG^\circ}^{\Fr_0})$ has an associated finite group $\Gamma_\cU=\widetilde\Gamma_\cU$ (since ${\qG^\circ}^{\Fr_0}$ is split). The group $A$ acts on the set of families $\cU$. For every orbit $\cO_A=A\cdot \cU$ with representative $\cU$, let $\rZ_A(\cU)$ be the corresponding isotropy group. Then $\rZ_A(\cU)$ permutes the elements of $\cU$, hence the corresponding parameters $\cM(\Gamma_\cU)$. If $\Gamma_\cU$ is abelian, which turns out to be the case in all of the examples of interest to us when $A\neq \{1\}$, this automatically defines an action of $\rZ_A(\cU)$ on $\Gamma_\cU$, hence a group
\begin{equation}
\widetilde\Gamma_\cU^A=\Gamma_\cU\rtimes \rZ_A(\cU).
\end{equation}
See \cite[\S5]{DM} and \cite[\S17]{Lu86} for more details.

\begin{prop}[(cf. {\cite[Proposition 5.2]{DM}})]\label{p:disconnected}
As above, assume $\qG$ is $\Fr_0$-split.  
The parametrization (\ref{e:L-class}) induces a bijection
\[\bigsqcup_{a\in A}\Irr_\un(\qG^{F_a})\longleftrightarrow \bigsqcup_{\cU\subset  A\backslash\Irr_\un({\qG^\circ}^{F_0})} \cM(\widetilde\Gamma_\cU^A),
\]
where $\cU$ in the right-hand side ranges over a set of representatives of the $A$-orbits of families in $\Irr_\un({\qG^\circ}^{F_0})$.
\end{prop}

\begin{proof}
This can be viewed as a particular case of \cite[Proposition 5.2]{DM}. In {\it loc. cit.}, the right-hand side of the bijection involves the groups $\overline\cM(\widetilde\Gamma_\cU^A\subset \widetilde\Gamma_\cU^A\rtimes \langle\Fr_0\rangle)$, but since we are assuming $\qG$ is $\Fr_0$-split, $\overline\cM(\widetilde\Gamma_\cU^A\subset \widetilde\Gamma_\cU^A\rtimes \langle\Fr_0\rangle)\cong\cM(\widetilde\Gamma_\cU^A).$
\end{proof}

\begin{rem}
The bijection in Proposition \ref{p:disconnected} is not unique. To account for this (and in order to be able to carry out computations later on), we will work out explicit parametrizations in Examples \ref{ex:2-flip} -- \ref{ex:D4}. These cases are relevant for the branching computations for the unipotent representations of reductive $p$-adic groups in the inner class of the split group.
\end{rem}

Let $R_{\un}(\qG^{F_a})$ be the $\mathbb{C}$-span of $\Irr_\un(\qG^{F_a})$. The bijection of Proposition \ref{p:disconnected} induces a linear isomorphism 
\begin{equation}\label{e:vs-iso}
    \bigoplus_{a \in A} R_{\un}(\qG^{\Fr_a}) \to \bigoplus_{\cU\subset  A\backslash\Irr_\un{\qG^\circ}^{\Fr_0}} \CC[\cM(\widetilde\Gamma_\cU^A)].
\end{equation}
The right-hand side of (\ref{e:vs-iso}) has the involution given by (\ref{e:finite-flipM}). Define 
\begin{equation}\label{e:FT-disconnected}
\FT_{\qG}\colon \bigoplus_{a\in A}R_{\un}(\qG^{\Fr_a})\to \bigoplus_{a\in A}R_{\un}(\qG^{\Fr_a}),
\end{equation}
to be the corresponding involution on the left-hand side.

In the examples below, when $A$ is clear from the context, we may write $\widetilde \Gamma_\cU$ in place of $\widetilde \Gamma_\cU^A$ for simplicity of notation.
\begin{ex}\label{ex:2-flip}
Let $\qH$ be a connected almost simple $\mathbb F_q$-split group and $\qG=(\qH\times \qH)\rtimes \bZ/2\bZ$, where the nontrivial element $\delta$ of $A=\bZ/2\bZ$ acts by flipping the two copies of $\qH$. There are two pure inner twists:
\[\InnT^p \qG=\{\qH(\mathbb F_q)^2\rtimes \bZ/2\bZ, \qH(\mathbb F_{q^2})\rtimes \bZ/2\bZ\},
\]
the second one for the Frobenius map $\Fr_1(h_1,h_2)=(\Fr_0(h_2),\Fr_0(h_1))$, $h_1,h_2\in \qH$. 
A family of $\qG^\circ(\mathbb F_q)=\qH(\mathbb F_q)^2$ is $\cU_1\boxtimes \cU_2$, where $\cU_1$, $\cU_2$ are unipotent families of $\qH$. The $A$-orbits are either $\{\cU_1\boxtimes\cU_2, \cU_2\boxtimes\cU_1\}$ for $\cU_1\neq \cU_2$ or $\{\cU\boxtimes\cU\}$. Assume that all $\Gamma_\cU$ are abelian. Set
\[\widetilde \Gamma_{\cU_1\boxtimes\cU_2}^{\bZ/2\bZ}=\Gamma_{\cU_1}\times \Gamma_{\cU_2},\ \cU_1\neq \cU_2,\quad \widetilde \Gamma_{\cU\boxtimes\cU}^{\bZ/2\bZ}=\Gamma_\cU^2\rtimes \bZ/2\bZ,
\]
with the flip action of $\delta$. There are $\frac {\ell(\ell+3)}2$ conjugacy classes in $\widetilde \Gamma_{\cU\boxtimes\cU}$, $\ell=|\Gamma_\cU|$, and they are represented by 
\begin{itemize}
\item $(x,x')\sim (x',x)$ if $x\neq x'\in \Gamma_\cU$, $\rZ_{\widetilde\Gamma_{\cU\boxtimes\cU}}((x,x'))=\Gamma_\cU^2$;
\item $(x,x)$, $x\in\Gamma_\cU$, $\rZ_{\widetilde\Gamma_{\cU\boxtimes\cU}}((x,x))=\Gamma_\cU^2\rtimes\bZ/2\bZ$;
\item $(x,1)\delta$, $x\in\Gamma_\cU$, $\rZ_{\widetilde\Gamma_{\cU\boxtimes\cU}}((x,1)\delta)=\langle\Gamma_\cU^\Delta,(x,1)\delta\rangle$, where $\Gamma_{\cU}^\Delta$ is the diagonal copy of $\Gamma_\cU$.
\end{itemize}

When $\cU_1\neq \cU_2$, if $\rho_1\in \cU_1$, $\rho_2\in \cU_2$, then $\rho_1\times\rho_2:=\Ind_{\qG^\circ(\mathbb F_q)}^{\qG(\mathbb F_q)}(\rho_1\boxtimes\rho_2)$ is parametrized by $(\bar x_{\rho_1},\bar x_{\rho_2})\in \cM(\widetilde \Gamma_{\cU_1\boxtimes\cU_2})$.

In the second case, let $\rho,\rho'\in \cU$. If $\rho\neq \rho'$, then $\rho\times\rho'\cong \rho'\times\rho$ is an irreducible representation of $\qG(\mathbb F_q)$. If $\bar x_\rho=(x,\sigma)$, $\bar x_{\rho'}=(x',\sigma')$ are the corresponding parameters of $\rho$, $\rho'$ in $\cM(\Gamma_u)$, then the parameter for $\rho\times\rho'$ in $\cM(\widetilde\Gamma_{\cU\boxtimes\cU})$ is $((x,x'), \sigma\boxtimes\sigma')$, if $x\neq x'$, or $((x,x), \sigma\times \sigma')$, where $\sigma\times\sigma'=\Ind_{\Gamma_\cU}^{\Gamma_\cU^2\rtimes\bZ/2\bZ}(\sigma\boxtimes\sigma')$, if $\sigma\neq\sigma'$.

If $\rho=\rho'$, then we can extend $\rho\boxtimes \rho$ in two different ways to $\qG(\mathbb F_q)$, denoted by $(\rho \times \rho)^\pm$ relative to the character of $\bZ/2\bZ$. The corresponding parameters in $\cM(\widetilde \Gamma_{\cU\boxtimes\cU})$ are $((x,x), (\sigma\times\sigma)^\pm)$, with the obvious notation. 

For the second pure inner twist, the irreducible unipotent representations of $\qH(\mathbb F_{q^2})$ are given by the same families $\cU$ as for $\qH(\mathbb F_q)$ and $\delta$ fixes each unipotent representation $\rho$ of $\qH(\mathbb F_{q^2})$. Let $\rho$ be an irreducible $\qH(\mathbb F_{q^2})$-representation in $\cU$ parametrized by $\bar x_\rho=(x,\sigma)$, $x\in \Gamma_\cU$, $\sigma\in \widehat \Gamma_\cU$. Then it can be extended in two different ways $\rho^\pm$ to $\qH(\mathbb F_{q^2})\rtimes\bZ/2\bZ$. The centralizer $\rZ_{\widetilde\Gamma_{\cU\boxtimes\cU}}((x,1)\delta)=\langle\Gamma_\cU^\Delta,(x,1)\delta\rangle$ is isomorphic to the direct product $C_x:=\langle (y,y)\mid y\neq x\in \Gamma_\cU\rangle \times \langle (x,1)\delta\rangle$, since $\left((x,1)\delta\right)^2=(x,x)$. Regard $\sigma$ as a representation of the subgroup $\Gamma_\cU^\Delta$. There are two ways $\sigma^\pm$ to extend it to $C_x$, coming from the short exact sequence $1\to \langle (x,x)\rangle \to\langle (x,1)\delta\rangle\to \bZ/2\bZ\to 1$. We attach $((x,1)\delta,\sigma^\pm)\in \cM(\widetilde\Gamma_{\cU\boxtimes\cU})$ to $\rho^\pm$. To fix a choice of $\pm$, we fix a choice of primitive $\ell$-th root of unity $\zeta_\ell$ for each $\ell$. Then, if $\sigma((x,x))=\zeta_k^j$, for some $j$, where $k$ is the order of $x$, then $\sigma^+((x,1)\delta)=\zeta_{2k}^j$. For our applications, $\qH$ will be a classical group and therefore, $\Gamma_\cU$ a $2$-group, hence $x$ will have order $k\le 2$.

\end{ex}

\begin{ex}\label{ex:GL}
Let $\qG^\circ=\GL_k^m$, ${\qG^\circ}^{\Fr_0}=\GL_k(\mathbb F_q)^m$, and $A=\bZ/m\bZ$ acting by cyclic permutations on the factors of $\qG^\circ$.  Then
\[\InnT^p \qG=\{{\qG^\circ}^{\Fr_r}\rtimes\bZ/m\bZ= \GL_k(\mathbb F_{q^{m/d}})^{d}\rtimes \bZ/m\bZ \mid r\in \bZ/m\bZ,\ d=\gcd(r,m)\}.
\]
Each unipotent family of $ \GL_k(\mathbb F_{q^{m/d}})^{d}$ is a singleton $\{\rho_1\boxtimes\dots\boxtimes\rho_{d}\}$ where $\rho_i\in \widehat S_k$, $1\le i\le d$. Hence, we can ignore the difference between unipotent families and irreducible representations of symmetric groups. The irreducible representations of $S_k^{d}\rtimes\bZ/m\bZ$ are constructed by Mackey theory.

Start with a unipotent representation  $\rho=\rho_1\boxtimes\dots\boxtimes\rho_m$ of $\GL_k(\mathbb F_{q})^{m}\rtimes \bZ/m\bZ$ with stabilizer $\bZ/c\bZ$, $c|m$. This means that $\rho_i=\rho_{i+m/c}$ for all $i$ (viewed mod $m$) and that $\bZ/{(m/c)}\bZ$ has no fixed points under the cyclic action on $\rho_1\boxtimes\dots\boxtimes\rho_{m/c}$. The corresponding unipotent family $\widetilde\cU$ that we construct for $\InnT^p\qG$ has \[\widetilde\Gamma_\cU^{\bZ/m\bZ}=\bZ/c\bZ.\] The irreducible representations $\widetilde \rho$ of ${\qG}^{\Fr_0}$ whose restriction to ${\qG^\circ}^{\Fr_0}$ contain $\rho$ are in one-to-one correspondence to the characters of $\bZ/c\bZ$, hence they are parametrized in $\cM(\widetilde\Gamma_\cU)$ by the pairs $(0,\sigma)$, $\sigma\in \widehat {\bZ/c\bZ}$.

For every $r\in \bZ/m\bZ$ such that $m/c$ divides $d=\gcd(r,m)$, consider the representation of ${\qG^\circ}^{\Fr_r}$ given by $\rho^r=\rho_1\boxtimes\dots\boxtimes\rho_{d}$. The stabilizer of this representation in $\bZ/m\bZ$ is also $\bZ/c\bZ$. The irreducible representations $\widetilde \rho^r$ of ${\qG}^{\Fr_r}$ whose restriction to ${\qG^\circ}^{\Fr_r}$ contain $\rho^r$ are again in one-to-one correspondence to the characters of $\bZ/c\bZ$, and we parametrize them in $\cM(\widetilde\Gamma_\cU)$ by the pairs $(\frac{rc}m,\sigma)$, $\sigma\in \widehat {\bZ/c\bZ}$. This completes the parametrization via $\cM(\widetilde \Gamma_\cU^{\bZ/m\bZ})$ of the unipotent representations for $\InnT^p \qG$ corresponding to the family $\cU=\{\rho\}$ in ${\qG^\circ}^{\Fr_0}$. 
\end{ex}

\begin{ex}\label{ex:O}
Let $\qG = \mathrm O_{2n}, \qG^\circ =\SO_{2n}$, $n\ge 2$, $A=\bZ/2\bZ =\langle\delta\rangle$. There are two pure inner twists $\InnT^p\qG=\{\mathrm O^+_{2n}(\mathbb F_q),\mathrm O^-_{2n}(\mathbb F_q)\}$. In this case, we use the parametrizations of \cite[\S4.6,\S4.18]{Lubook}.
Recall that a symbol for type $D_n$ is an array $\Lambda=\left(\begin{matrix}\lambda_1&\lambda_2&\dots&\lambda_b\\\mu_1&\mu_2&\dots&\mu_{b'}\end{matrix}\right)$, $b+b'=2m$, $0\le \lambda_1<\dots<\lambda_{b'}$, $0\le\mu_1<\dots<\mu_{b'}$, which is considered the same as the array where the rows are flipped. A symbol where $b=b'=m$ and $\lambda_1\le\mu_1\le\lambda_2\le\mu_2\le\dots$, $\sum \lambda_i^2+\lambda \mu_i^2=n+m^2-m$ is called special. 

Let $Z$ be a special symbol. 
In the case when $\lambda_i=\mu_i$ for all $1\le i\le m$, one attaches to $Z$ two unipotent ${\qG^\circ}^{\Fr_0}$-families, $\cU'=\{\rho\}$ and $\cU''=\{\rho'\}$, each consisting of a single unipotent representation and with $\Gamma_{\cU'}=\Gamma_{\cU''}=\{1\}$. In this case, the action of $\delta$ flips the two families. Hence they give rise to a single family $\{\widetilde\rho\}$ for $\qG^{F_0}$, $\widetilde\rho|_{{\qG^\circ}^{\Fr_0}}=\rho\oplus\rho'$, and $\widetilde\Gamma_{\cU'}=\{1\}$.

Assume now that the two rows of the symbol $Z$ are not identical, i.e., $Z$ is nondegenerate in the sense of {\it loc. cit.}. Then $Z$ defines one unipotent family for ${\qG^\circ}^{\Fr_0}$ and one for ${\qG^\circ}^{\Fr_1}$. Each element of these families is stable under $\delta$ so it can be lifted to two different $\qG^{\Fr_0}$, respectively $\qG^{\Fr_1}$, representations.

The unipotent representations in the ${\qG^\circ}^{\Fr_0}$-family $\cU_Z$ corresponding to $Z$ are indexed by the set $\cM_Z$ of symbols $\Lambda$ such that $b-b'\equiv 0$ mod $4$, $b+b'=2m$. The unipotent representations in the ${\qG^\circ}^{\Fr_1}$-family $\cU^-_Z$ corresponding to $Z$ are indexed by the set $\cM_Z^-$ of symbols $\Lambda$ such that $b-b'\equiv 2$ mod $4$, $b+b'=2m$. Let $Z_1$ be the set of elements that appear as entries of $Z$ only once. Let $2d=|Z_1|$. Define
\begin{itemize}
\item $V_{Z_1}$: the set of subsets $X\subseteq Z_1$ of even cardinality, with the structure of an $\mathbb F_2$-vector space with the sum given by the symmetric difference;
\item $V_{Z_1}'$: the set of subsets $X\subseteq Z_1$ with the structure of an $\mathbb F_2$-vector space with the sum given by the symmetric difference, modulo the line spanned by $Z_1$ itself;
\item $(V_{Z_1}')^+$: the subspace of $V_{Z_1}'$ where the elements are the subsets $X$ of even cardinality;
\item $(V_{Z_1}')^-$: the subspace of $V_{Z_1}'$ where the elements are the subsets $X$ of odd cardinality.
\end{itemize}
Notice that $(V_{Z_1}')^+$ is also the image of the projection of $V_{Z_1}$ to $V_{Z_1}'$. The dimensions of $V_{Z_1}'$ and $V_{Z_1}$ over $\mathbb F_2$ are $2d-1$, while the dimension of $(V_{Z_1}')$ is $2d-2$. There is a nonsingular pairing 
\begin{equation}
(~,~)\colon V_{Z_1}'\times V_{Z_1}\to \mathbb F_2,\quad (X_1,X_2)\mapsto |X_1\cap X_2| \text{ mod }2.
\end{equation}
This pairing restricts to a nonsingular symplectic $\mathbb F_2$-form of $(V_{Z_1}')^+$. If $V_{Z_1}$ has the basis $e_1,e_2,\dots,e_{2d-1}$ as in \cite{Lubook}, then $(V_{Z_1}')^+$ is spanned by the images $\bar e_1,\bar e_2,\dots,\bar e_{2d-1}$ modulo the relation $\bar e_1+\bar e_3+\dots+\bar e_{2d-1}=0$. Let
\begin{align*}
\bar I'&=\text{subspace of }(V_{Z_1}')^+\text{ spanned by } \bar e_1,\bar e_3,\dots,\bar e_{2d-1},\\
\bar I''&=\text{subspace of }(V_{Z_1}')^+\text{ spanned by } \bar e_2,\bar e_4,\dots,\bar e_{2d-2};
\end{align*}
they are maximal isotropic subspaces of $(V_{Z_1}')^+$ and $(V_{Z_1}')^+=\bar I'\oplus \bar I''$. Then
\[\Gamma_{\cU_Z}=\bar I''\cong (\bZ/2\bZ)^{d-1}.
\]
As shown in \cite[\S4.6]{Lubook}, there is a natural bijection \[\cM_Z\leftrightarrow \cM(\Gamma_{\cU_Z})\cong (V_{Z_1}')^+=\bar I'\oplus\bar I''\] (where $\bar I'$ is identified with the group of characters of $\Gamma_{\cU_Z}$).
Denote
\begin{equation}
\widetilde\Gamma_{\cU_Z}=\{v\in V_{Z_1}'\mid (v,e_{2i})=0,\ 1\le i\le d-1\}\cong (\bZ/2\bZ)^d.
\end{equation}
Clearly, $\Gamma_{\cU_Z}\unlhd\widetilde\Gamma_{\cU_Z}$. 
As shown in \cite[\S4.18]{Lubook}, there is a natural bijection \[\cM_Z^-\leftrightarrow \overline\cM(\Gamma_{\cU_Z}\unlhd\widetilde\Gamma_{\cU_Z})\cong (V_{Z_1}')^-\cong (\widetilde \Gamma_{\cU_Z}\setminus \Gamma_{\cU_Z})\times \bar I'.\]
Let $\widetilde \cU_Z$ be the set (family) of irreducible representations in $\Irr_\un \qG^{\Fr_0}\sqcup \Irr_\un \qG^{\Fr_1}$ whose restrictions to ${\qG^\circ}^{\Fr_0}$ (resp. ${\qG^\circ}^{\Fr_1}$) are in $\cU_Z$ (resp. $\cU_Z^-$). Since each unipotent representation in $\cU_Z$ and $\cU_Z^-$ extends in two different ways to the corresponding disconnected group, the parametrization above implies easily that there is natural bijection
\begin{equation}
\widetilde \cU_Z\longleftrightarrow \cM(\widetilde\Gamma_{\cU_Z}).
\end{equation}
Explicitly, let $\{\bar f_1,\bar f_2,\dots,\bar f_d\}$ be the spanning set of $V_{Z_1}'$ subject to $\sum_{i=1}^{2d} \bar f_i=0$, such that $\bar e_i=\bar f_i+\bar f_{i+1}$, $1\le i \le 2d-1$. Then an $\mathbb F_2$-basis of $\widetilde \Gamma_{\cU_Z}$ is given by $\{\bar f_1,\bar e_2,\bar e_4,\dots, \bar e_{2d}\}$. An irreducible character of $\Gamma_{\cU_Z}=\langle \bar e_2,\bar e_4,\dots, \bar e_{2d}\rangle$ can be extended in two different ways to $\widetilde \Gamma_{\cU_Z}$ by setting the character value on $\bar f_1$ to $1$ or $-1$. The value $1$ corresponds to the representations of the identity components of $\qG^{\Fr_0}$, $\qG^{\Fr_1}$ extended by letting $\delta$ act trivially, while the $-1$ value to the ones where $\delta$ acts by $-1$.
\end{ex}

\begin{ex}\label{ex:A-E}
Let $\qG^\circ$ be of type $A_{k-1}$, $k\ge 3$, or $E_6$, and $A=\bZ/2\bZ=\langle\delta\rangle$ acting by the nontrivial automorphism of the Dynkin diagram. The nonsplit pure inner twist has ${\qG^\circ}^{\Fr_1}$ of type $^2\!A_{k-1}$ or $^2\!E_6$, respectively. By \cite[\S4.19]{Lubook}, every unipotent family $\cU$ of $\qG^\circ$ is fixed pointwise by $A$. Hence 
\[\widetilde \Gamma_\cU^{\bZ/2\bZ}=\Gamma_U\times \bZ/2\bZ,\text{ for all } \cU.
\]
Each irreducible representation $\rho$ of the split form ${\qG^\circ}^{\Fr_0}$ can be extended to $\qG^{\Fr_0}$ in two different ways $\rho^\pm$ corresponding to the two characters of $\mathbb Z/2\bZ$. If the parameter for $\rho$ is $\bar x_\rho=(x,\sigma)\in \cM(\Gamma_\cU)$, then the parameters for $\rho^\pm$ are $((x,1), \sigma^\pm)$ with the obvious notation.

Similarly, an irreducible representation $\rho'$ of the nonsplit pure inner twist ${\qG^\circ}^{\Fr_1}$ can be extended to $\qG^{\Fr_1}$ in two different ways ${\rho'}^\pm$. If the parameter for $\rho'$ is $\bar x_\rho=(x',\sigma')\in \cM(\Gamma_\cU)$, then the parameters for ${\rho'}^\pm$ are $((x',\delta), {\sigma'}^\pm)$.
\end{ex}

\begin{ex}\label{ex:spin}
Let $\qH$ be of type $D_k$, $k\ge 2$ and $\qG^\circ=\qH\times \qH$. Let $A=\langle \delta_1\rangle\times \langle\delta_2\rangle\cong \bZ/2\bZ\times \bZ/2\bZ$, where $\delta_1$ acts by the nontrivial outer automorphism of the Dynkin diagram of type $D_k$, and $\delta_2$ flips the $\qH$-factors. This case is therefore a combination of Example \ref{ex:O} and Example \ref{ex:2-flip} and the parametrization of families for the pure inner twists of $\qG$ follows from these examples, i.e., the same parametrizations as in Example \ref{ex:2-flip} but constructed from the orthogonal families $\widetilde U_Z$ from Example \ref{ex:O}.

\

Now for the same $\qH$ and $\qG^\circ$, suppose $A=\langle \delta\rangle=\bZ/4\bZ$. If $s'_1,s''_1$ are the two commuting extremal reflections of the first $\qH=D_k$ and $s'_2,s''_2$ are the similar reflections for the second $\qH=D_k$, then $\delta$ acts by the cyclic permutation:
\[\delta\colon\quad s'_1\mapsto s'_2\mapsto s''_1\mapsto s''_2\mapsto s'_1.
\]
On all the other simple reflections of the two components of type $D_k$, $\delta$ acts by the obvious diagram flip (of order $2$). To describe the pure inner twists, let $\Fr$ denote the Frobenius map of $\qH$ whose fixed points is the nonsplit group of type $^2\!D_k$. Then
\[\Fr_1\colon \qH\times \qH\to \qH\times \qH, \quad \Fr_1(h_1,h_2)=(\Fr(h_2),h_1)
\]
is a Frobenius automorphism and $\Fr_r=\Fr_1^r$, $r\in \bZ/4\bZ$, e.g., \cite[Example 1.4.23]{GM}. The identity components of the pure inner twists are the finite reductive groups of types:
\[{\qG^\circ}^{\Fr_0}: ~D_k\times D_k,\quad {\qG^\circ}^{\Fr_1}: ~^2\!D_k,\quad {\qG^\circ}^{\Fr_2}: ~^2\!D_k\times ~^2\!D_k,\quad {\qG^\circ}^{\Fr_3}: ~^2\!D_k.
\]
If $\rho_1,\rho_2$ are two unipotent representations of $D_k$, the action of $\delta$ is
\[\delta(\rho_1,\rho_2)=(\rho_2',\rho_1),\text{ where }\rho_2'=\begin{cases}\rho_2,&\text{ if the symbol of }\rho_2\text{ is nondegenerate},\\\rho_2^-,&\text{ otherwise},\end{cases}
\]
where $\rho^-_2$ is unipotent representation parametrized by the other degenerate symbol with the same rows. See Example \ref{ex:O}.

We start with a family $\cU_1\times\cU_2$ of ${\qG^\circ}^{\Fr_0}=D_k\times D_k$. If $\cU_2$ consists of a degenerate symbol, then the stabilizer in $A$ is always $1$, regardless of what $\cU_1$ is. (Similarly if $\cU_1$ is degenerate.) In this case, 
\[\widetilde \Gamma_{\cU_1\times\cU_2}^{\bZ/4\bZ}=\Gamma_{\cU_1}.
\]
(Recall that $\Gamma_{\cU_2}=1$ necessarily.) Since the stabilizer in $\bZ/4\bZ$ of each representation $\rho_1\boxtimes\rho_2$, $\rho_1\in \cU_1$, $\rho_2\in\cU_2$ is also trivial in this case, it follows that there is a one-to-one correspondence between the representations $\Ind_{{\qG^\circ}^{\Fr_0}}^{\qG^{\Fr_0}}(\rho_1\boxtimes\rho_2)$ and $\rho_1\in\cU_1$, hence a parametrization by $\cM(\Gamma_{\cU_1})$ as expected.
\smallskip

For the rest of the example, assume that all families correspond to nondegenerate symbols. Let $Z_1$, $Z_2$ be two nondegenerate symbols of type $D_k$. Let $\cU_1,\cU_2$ be the corresponding families for $D_k$ and $\cU_1^-$, $\cU_2^-$ the families for $^2\!D_k$. Suppose first that $Z_1\neq Z_2$. Then the stabilizer in $A$ is $\bZ/2\bZ=\langle\delta^2\rangle$, hence the group for the pure inner twists is
\[\widetilde \Gamma_{\cU_1\boxtimes\cU_2}^{\bZ/4\bZ}=\Gamma_{\cU_1}\times\Gamma_{\cU_2}\times \bZ/2\bZ.
\]
 If $\rho_1\in\cU_1$ and $\rho_2\in\cU_2$, the stabilizer in $A$ of $\rho_1\boxtimes\rho_2$ is also $\bZ/2\bZ=\langle\delta^2\rangle$. By Mackey theory, we get two irreducible representations of $\qG^{\Fr_0}$ by inducing $\rho_1\boxtimes\rho_2$ twisted by the trivial or the sign character of $\bZ/2\bZ$. If $\bar x_{\rho_i}=(x_i,\sigma_i)\in \cM(\Gamma_{\cU_i})$, $i=1,2$, then the two induced representations are parametrized by $((x_1,x_2,1),\sigma_1\boxtimes\sigma_2\boxtimes\tau)\in \cM(\Gamma_{\cU_1}\times\Gamma_{\cU_2}\times \bZ/2\bZ)$, where $\tau$ is the trivial or the sign character of $\bZ/2\bZ$.

If  $\rho_1'\in\cU_1^-$ and $\rho_2'\in\cU_2^-$, the analysis is analogous. The difference is that $\bar x_{\rho_i'}=(x_i',\sigma_i')\in \overline \cM(\Gamma_{\cU_i}\unlhd\widetilde \Gamma_{\cU_i})$, $i=1,2$, where $x_i'\in \widetilde \Gamma_{\cU_i}\setminus\Gamma_{\cU_i}$, $\sigma_i\in \widehat \Gamma_{\cU_i}$. Write $x_i=y_i\alpha$, $i=1,2$, $y_i\in \Gamma_{\cU_i}$, where $\alpha$ is the nontrivial automorphism of the $D_k$ diagram. Then the two unipotent representations of $\qG^{\Fr_2}$ whose restriction to ${\qG^\circ}^{\Fr_2}$ contain $\rho_1'\boxtimes\rho_2'$ are parametrized by $((y_1,y_2,\delta^2),\sigma_1\boxtimes\sigma_2\boxtimes\tau)\in \cM(\Gamma_{\cU_1}\times\Gamma_{\cU_2}\times \bZ/2\bZ)$, where $\tau$ is the trivial or the sign character of $\bZ/2\bZ$.

\medskip

Finally, if $Z_1=Z_2=Z$ with the families $\cU$ of $D_k$ and $\cU^-$ of $^2\!D_k$,  then the stabilizer  of $\cU\boxtimes\cU$ is $A=\bZ/4\bZ$. In this case, set
\begin{equation}\widetilde \Gamma_{\cU\boxtimes\cU}^{\bZ/4\bZ}=(\Gamma_{\cU}\times\Gamma_{\cU})\rtimes \bZ/4\bZ.
\end{equation}
All four pure inner twists $\qG^{\Fr_r}$ contribute in this case. Each conjugacy class in $\widetilde \Gamma_{\cU\boxtimes\cU}^{\bZ/4\bZ}$ is represented by an element $(x,y,r)$ with $x,y\in \Gamma_\cU$ and $r\in \bZ/4\bZ$. It will correspond to a unipotent representation of $\qG^{\Fr_r}$, for the same $r$. 

If $r=0$, then the conjugacy classes are given by $(x,x',0)\sim (x',x,0)$ and its stabilizer in $\widetilde \Gamma_{\cU\boxtimes\cU}^{\bZ/4\bZ}$ is $\Gamma_\cU^2 \times\langle\delta^2\rangle$ if $x\neq x'$, or all of $\widetilde \Gamma_{\cU\boxtimes\cU}^{\bZ/4\bZ}$ if $x=x'$. If $\rho,\rho'\in\cU$ with parameters $\bar x_\rho=(x,\sigma)$, $\bar x_{\rho'}=(x',\sigma')$, it is clear that there is a perfect matching between the induced representation coming from the Mackey construction and the parameters $((x,x',\bar 0),\widetilde\sigma)\in \cM(\widetilde \Gamma_{\cU\boxtimes\cU}^{\bZ/4\bZ})$, where $\widetilde\sigma\in \widehat{\rZ_{\widetilde \Gamma_{\cU\boxtimes\cU}^{\bZ/4\bZ}}((x,x',\bar 0))}$.

If $r=2$, the conjugacy classes are given by $(x,x',\delta^2)\sim (x',x,\delta^2)$ and its stabilizer in $\widetilde \Gamma_{\cU\boxtimes\cU}^{\bZ/4\bZ}$ is $\Gamma_\cU^2 \times\langle\delta^2\rangle$  if $x\neq x'$ or all $\widetilde \Gamma_{\cU\boxtimes\cU}^{\bZ/4\bZ}$ if $x=x'$. If $\rho,\rho'\in\cU^-$ with parameters $\bar x_\rho^-=(x\alpha,\sigma)$, $\bar x_{\rho'}^-=(x'\alpha,\sigma')$, again there is a perfect matching between the induced representations coming from the Mackey construction and the parameters $((x,x',\delta^2),\widetilde\sigma)\in \cM(\widetilde \Gamma_{\cU\boxtimes\cU}^{\bZ/4\bZ})$, where $\widetilde\sigma\in \widehat{\rZ_{\widetilde \Gamma_{\cU\boxtimes\cU}^{\bZ/4\bZ}}((x,x',\bar 2))}$.

If $r=1$ or $3$, the conjugacy classes are given by $(x,1,\delta^r)$, cf. Example \ref{ex:2-flip}. The stabilizer in this case is $\langle\Gamma_\cU^\Delta,\delta^2,(x,1,\delta)\rangle.$ Let $\rho$ be a representation in the unipotent family $\cU^-$ with parameter $\bar x_\rho=(x\alpha,\sigma)$, $\sigma\in \widehat \Gamma_\cU$. It can be extended in four different ways to $^2\!D_k\rtimes \bZ/4\bZ$ corresponding to the characters of $\bZ/4\bZ$. Since $x$ has order $2$, the cyclic group
\[\langle (x,1,\delta)\rangle=\langle (1,1,1), (x,1,\delta), (x,x,\delta^2), (1,x,\delta^3)\rangle\cong \bZ/4\bZ.
\]
Notice that there is a short exact sequence (which does not split)
\[1\longrightarrow \Gamma_\cU^\Delta\longrightarrow \rZ_{\widetilde\Gamma_{\cU\boxtimes\cU}^{\bZ/4\bZ}}((x,1,\delta^r))\longrightarrow \bZ/4\bZ\longrightarrow 1,
\]
where the quotient $\bZ/4\bZ$ is generated by the image of $(x,1,\delta)$. This means that $\sigma\in \widehat \Gamma_\cU$, viewed as a representation of $\Gamma_\cU^\Delta$ can be lifted in four different ways to $\rZ_{\widetilde\Gamma_{\cU\boxtimes\cU}^{\bZ/4\bZ}}((x,1,\delta^r))$: first one lifts $\sigma$ in two different ways to $\sigma^\pm$, representations of $\Gamma_\cU^\Delta\times\langle\delta^2\rangle$ corresponding to the trivial and the sign character of $\langle\delta^2\rangle$. Then, fixing a square roots $\zeta^\pm$ of $\sigma^\pm(x,x,\delta^2)$, one constructs lifts $\widetilde\sigma^i$ of $\sigma$, $0\le i\le 3$, by setting
\[\widetilde\sigma^0((x,1,\delta))=\zeta^+,\ \widetilde\sigma^1((x,1,\delta))=-\zeta^+,\ \widetilde\sigma^2((x,1,\delta))=\zeta^-,\ \widetilde\sigma^3((x,1,\delta))=-\zeta^-. 
\]
Notice that $\{\pm\zeta^\pm\}$ is the set of $4$-th roots of $1$, and this gives the desired parametrization.
\end{ex}

\begin{ex}\label{ex:D4}
Let $\qG^\circ$ be of type $D_4$ and $A=\bZ/3\bZ=\langle \delta\rangle$ acting on the Dynkin diagram by cyclically permuting the extremal nodes. The Weyl group $W(D_4)$ has $13$ irreducible representations which we denote by bipartitions of $4$, $\alpha\times\beta$ up to swapping $\alpha$ and $\beta$, except where $\alpha=\beta$, there are two non-isomorphic representations $\alpha\times\alpha^\pm$. There is one cuspidal unipotent representation $\rho_c$, and in total $14$ unipotent representations of ${\qG^\circ}^{\Fr_0}$.

All families are singletons with associated finite group $\Gamma_\cU=\{1\}$, except the family
\[\{(12)\times(1),  (22)\times\emptyset, (11)\times (2),\rho_c\}
\]
for which the finite group is $\Gamma_\cU=\bZ/2\bZ$. This family and the following four singleton families:
\[\{(4)\times\emptyset\},\quad \{(1111)\times\emptyset\},\quad \{(3)\times (1)\},\quad \{(111)\times (1)\}
\]
are $A$-stable and in fact each element in the family is $A$-stable. The remaining $6$ unipotent (singleton) families form two $A$-orbits:
\[\{(13)\times\emptyset, (2)\times (2)^+,(2)\times (2)^-\}\text{ and } \{(112)\times\emptyset, (11)\times (11)^+,(11)\times (11)^-\}.
\]
According to our recipe, the groups $\widetilde\Gamma_{\cU}^{\bZ/3\bZ}$ are:
\begin{itemize}
\item $\bZ/3\bZ$ corresponding to each of the four $A$-stable singleton families $\cU$;
\item $\bZ/2\bZ\times \bZ/3\bZ$ for the unique family with $4$ elements;
\item $\{1\}$ for each one of the two nontrivial $A$-orbits. 
\end{itemize}
Hence the right hand side of Proposition \ref{p:disconnected} is $\cM(\bZ/3\bZ)^4\sqcup \cM(\bZ/2\bZ\times \bZ/3\bZ)\sqcup \cM(\{1\})^2$ which has $3^2\times 4+6^2+1^2\times 2=74$ elements.

The irreducible unipotent representations of the disconnected group $D_4\rtimes \bZ/3\bZ$ are parametrized, via Mackey theory, by the elements $(x,\sigma)\in \cM(\widetilde\Gamma_{\cU}^{\bZ/3\bZ})$, where $x\in \Gamma_{\cU}$ and $\cU$ ranges over a set of representatives of the $A$-orbits of families of $D_4$. There are $26$ such irreducible representations.

The other two $A$-forms corresponding to $\delta$ and $\delta^{-1}$ are both isomorphic to the finite group of type $^3\!D_4$. There are $8$ irreducible unipotent representations of $^3\!D_4$ each coming from one of $\gamma$-stable irreducible unipotent representations of $D_4$. By induction, there are $8\times 3=24$ irreducible unipotent representations of $^3\!D_4\rtimes \bZ/3\bZ$. The irreducible representations of the $^3\!D_4$ corresponding to $\delta$ are parametrized by $(x\delta,\sigma)\in \cM(\widetilde\Gamma_{\cU}^{\bZ/3\bZ})$, where $x\in \Gamma_{\cU}$ and $\cU$ ranges over the set of $A$-stable families. Similarly for $\delta^{-1}$.

\end{ex}

\section{ Maximal compact subgroups}\label{s:maximal}

We return to the setting of Section \ref{s:inner}, so $\bG$ is a connected reductive group over $F$ and $G = \bG(F)$. In this section, we assume in addition that $\bG$ is simple and $F$-split with maximal $F$-split torus $\bS$. Let $\Pi_G$ be a set of simple roots for $\bG$ with respect to $\bS$, and extend $\Pi_G$ to a set of simple affine roots $\Pi^a_G = \Pi_G \cup \{\al_0\}$. Let $I \subset \bG(\fo_F)$ be the corresponding Iwahori subgroup of $G$, with $\bS(\mathfrak o_F)=\bS(F)\cap I$. Let $\widetilde W_G=\rN_{G}(\bS(F))/\bS(\mathfrak o_F)$ be the Iwahori--Weyl group. We have
\[G=\bigsqcup_{w\in \widetilde W_G} I \dot wI,
\]
where $\dot w$ denotes a choice of a lift in $N_{G}(\bS(F))$ of $w\in \widetilde W_G$. The finite Weyl group is $W_G = \rN_{G}(\bS(F))/\bS(F)$. Let $W^a_G$ be the affine Weyl group generated by the simple reflections $\{s_i\mid i\in \Pi^a_G\}$. Then
\[\widetilde W_G=W^a_G\rtimes\Omega_G,
\]
where $\Omega_G$ is a finite abelian group, the stabilizer in $\widetilde W_G$ of $I$.

\subsection{} Let $\max(G)$ denote the set of conjugacy classes of maximal compact open subgroups in $G$. To parametrize $\max(G)$, we define $S_{\max}(G)$ to be the set of $\Omega_G$-orbits of pairs $(A, \cO)$, where $A$ is a subgroup of $\Omega_G$ and $\cO$ is an $A$-orbit in $\Pi^a_G$ satisfying
\[\mathrm{Stab}_{\Omega_G}(\cO)=A.
\] 
By \cite{IM} and \cite{BT1}, $\max(G)$ is parametrized by $S_{\max}(G)$. 
Explicitly, given $(A, \cO) \in S_{\max}(G)$, we construct an element $K_\cO \in \max(G)$ as follows:
let $\widetilde W_{\cO}$ be the finite subgroup of $\widetilde W_G$  generated by $A$ and  $\{s_i,\ i\in \Pi^a_G\setminus \cO\}$. 
Set \[K_{\cO}=\bigsqcup_{w\in \widetilde W_{\cO}} I\dot w I.\] 
The map $(A, \cO) \mapsto K_\cO$ defines a bijection
between $S_{\max}(G)$ and $\max(G)$. (Note that a pair $(A, \cO) \in S_{\max}(G)$ is completely determined by $\cO$, so the notation $K_\cO$ is unambiguous.) In this notation, the maximal hyperspecial subgroup $\bG(\mathfrak{o}_F)$ is $K_{\{\al_0\}}$, where $\al_0$, as defined above, is the unique simple affine root in $\Pi^a_G \setminus \Pi_G$. 

Given $(A, \cO) \in S_{\max}(G)$, let $\widetilde W_{\cO}^\circ$ be the normal subgroup of $\widetilde W_{\cO}$ generated by $\{s_i,\ i\in \Pi_G^a\setminus \cO\}$. Then $K_{\cO}^\circ:=\bigsqcup_{w\in \widetilde W_{\cO}^\circ} I\dot w I$ is a parahoric subgroup of $G$, and we denote by $K_{\cO}^+$ its pro-unipotent radical. There is a short exact sequence
\[1\longrightarrow K_{\cO}^\circ\longrightarrow K_{\cO}\longrightarrow A\longrightarrow 1.
\]
Set $\overline K_{\cO}^\circ=K_{\cO}^\circ/K_{\cO}^+$ and $\overline K_{\cO}=K_{\cO}/K_{\cO}^+$. Then $\overline K_{\cO}=M_{\cO}(k_F)$, $\overline K_{\cO}^\circ=M_{\cO}^\circ(k_F)$, for a reductive $k_F$-split group $M_{\cO}$ with identity component $M_{\cO}^\circ$ and $M_{\cO}/M_{\cO}^\circ\cong A$. Let $\InnT^p \overline K_{\cO}\leftrightarrow A$ be the collection of pure inner twists of $M_\cO$. 

\medskip

Now we consider pure inner twists of $G$. By Section \ref{s:inner}, $\InnT^p G\cong \rH^1(F,\bG)$. If $\bG_\sc$ is the simply connected cover of $\bG$ and identifying $\Omega_G$ with the kernel of the surjection $\bG_\sc\to \bG$, then by \cite[Satz 2]{Kn}, 
$\rH^1(F,\bG)\cong \rH^2(F,\Omega_G)\cong  \Omega_G$, with the last equivalence because $\bG$ is $F$-split. 
Given $x \in \Omega_G$, let $G_x \in \InnT^p G$ be the corresponding pure inner twist.
If we denote the set of conjugacy classes of maximal compact open subgroups of $G_x$ by $\max(G_x)$, then there is a one-to-one correspondence
\begin{equation}\label{e:pairs-maxcpt}
\{(A,\cO) \in S_{\max}(G) \mid  x \in A\} \longleftrightarrow \max(G_x),
  \end{equation}
  which we write as $(A, \cO) \mapsto K_{x, \cO}$. More precisely, we can realize $K_{x, \cO}$ in the following way.
  We realize $G_x$ as the subgroup of $\bG(F_{\un})$ fixed under the $\Gal(F_{\un}/F)$-action corresponding to $x$. Then since $x \in A$, the parahoric of  $\bG(F_{\un})$ corresponding to $(A, \cO)$ as above is $\Gal(F_{\un}/F)$-stable. Let $K_{x, \cO}^\circ \subset G_x$ be the Galois-fixed subgroup. By \cite[3.3.4 Proposition]{BT1} (applied using \cite[5.2.12 Proposition]{BT2}), the normalizer $K_{x, \cO} := N_{G_x}(K_{x, \cO}^\circ)$ is a maximal compact subgroup of $G_x$.  
  
  Note that 
$\overline K_{x,\cO} \in \InnT^p\overline K_{\cO}$ is the pure inner twist given by $x\in A\cong \rH^1(k_F,M_\cO)\leftrightarrow\InnT^p\overline K_{\cO}$ (see Section \ref{s:disconn}). For fixed $(A, \cO) \in S_{\max}(G)$, we have
\begin{equation}\label{e:partition}
    \InnT^p \overline{K}_\cO = \{\overline{K}_{x, \cO} \mid x \in A\}.
\end{equation}
 Given $G' \in \InnT^p(G)$, for every maximal compact open subgroup $K' \in \max (G')$, 
    define $R_\un(\overline K')$ to be the $\bC$-span of $\Irr_\un \overline K'$, and
let
\begin{equation} \label{eqn:C(G)cpt_un}
  \cC(G)_{\cpt,\un} = \bigoplus_{G'\in \InnT^p G}~~\bigoplus_{K'\in \max(G')} R_\un (\overline K').
  \end{equation}
 By the discussion above, we have
 \begin{align}    
\cC(G)_{\cpt, \un} &= \bigoplus_{x \in \Omega}  ~~~\bigoplus_{\substack{(A, \cO) \in S_{\max}(G)\\ \text{ with } x \in A}} R_{\un}(\overline K_{x, \cO})\\
&=\bigoplus_{(A, \cO) \in S_{\max}(G)} ~~ \bigoplus_{x \in A} R_{\un}(\overline K_{x, \cO}) \label{e:AO-decomp}.
\end{align}
Note that for each $(A, \cO) \in S_{\max}(G)$, by (\ref{e:partition}), $\oplus_{x \in A} R_{\un}(\overline K_{x, \cO})$ has the involution $\FT_{\overline{K}_\cO}$ given by (\ref{e:FT-disconnected}). Putting together these involutions for all choices of $(A, \cO)$ gives the following definition.

\begin{defn}\label{d:FT-compact}
  Let $\FT_{\cpt,\un}=\bigoplus_{(A, \cO) \in S_{\max}(G)}\FT_{\overline K_\cO}$ be the involution on  $\cC(G)_{\cpt,\un}$
  defined by using (\ref{e:FT-disconnected}) and (\ref{e:partition}).
  \end{defn}
Notice that $\FT_{\cpt,\un}$ always preserves the space $R_\un (\overline{\bG(\fo_F)})$, since $\bG(\fo_F)$ corresponds to the pair $(A, \{\al_0\})$ with $A$ trivial. In the case when $\bG$ is simply connected, we have $\Omega_G = 1$ and $K_{\cO}=K^\circ_{\cO}$ for all $\cO$, so $\FT_{\cpt,\un}$ preserves the space $R_\un (\overline K_\cO)$ for all $K_\cO \in \max(G)$.
But in general it does not preserve $R_\un (\overline K_{x,\cO})$ for every maximal compact open subgroup $K_{x, \cO}$, which can be seen even in the case when $\bG = \PGL_2$ (see Example \ref{ex:pgl2}). 

    \begin{ex} We list the type of groups $\overline K_\cO$ in the case when $\bG$ is adjoint. Since we are only interested in unipotent representations, only the Lie type of $\overline K_\cO^\circ$ is important.
\begin{enumerate}
\item      If $\bG$ is also simply connected, which is the case for types $G_2$, $F_4$, $E_8$, the set $\max (G)$ is in one-to-one correspondence with the maximal subsets of $\Pi_a$ (equivalently, $\cO$ is a single vertex in $\Pi^a_G$).
\item If $\bG=\PGL_n$, $\Omega_G=\bZ/n\bZ$ acting by cyclically permuting $\Pi^a_G$, then for every divisor $m$ of $n$, we have an orbit $\cO_m$ in $\Pi^a_G$ with stabilizer $A=\bZ/m\bZ$ and
  \[ \overline K_{\cO_m}=P(\GL_{n/m}^m)\rtimes \bZ/m\bZ,
  \]
  where the semidirect product is given by the permutation action. This case corresponds to Example \ref{ex:GL}.
\item If $\bG=\SO_{2n+1}$, $\Omega_G=\bZ/2\bZ$, then $\overline K_\cO$ is either $\SO_{2n+1}$ ($A=1$) or $\SO_{2m+1}\times \rO_{2(n-m)}$ ($A=\bZ/2\bZ$), for $0\le m<n$. This case corresponds to Example \ref{ex:O}.
\item If $\bG=\mathrm{PSp}_{2n}$, $\Omega=\bZ/2\bZ$, then $\overline K_\cO$ is type
\begin{itemize}
\item  $C_k\times C_\ell$ ($A=1$) for $0\le k<\ell$, $k+\ell=n$;
\item $(C_k\times C_k)\rtimes \bZ/2\bZ$ ($A=\bZ/2\bZ$), if $n=2k$, or
\item $(C_i\times C_i\times A_{n-1-2i})\rtimes \bZ/2$ ($A=\bZ/2\bZ$), for $0\le i<\frac n2$.
\end{itemize}
Here $\bZ/2\bZ$ acts by flipping the two type $C$ factors and by the nontrivial diagram automorphism on the type $A$ factor. These cases are covered by Examples \ref{ex:2-flip}, \ref{ex:A-E}.

\item If $\bG=\mathrm{PSO}_{4m}$, $m\ge 2$, $\Omega_G=\langle\delta_1\rangle\times \langle\delta_2\rangle\cong \bZ/2\bZ\times \bZ/2\bZ$, then $\delta_1$ acts by flipping $\Pi_a$ horizontally and $\delta_2$ by the vertical flip. For each subgroup $A\le \Omega_G$, we give the possible $\overline K_\cO^\circ$.
  \begin{enumerate}
  \item If $A=\Omega$, $\overline K_\cO^\circ$ is of type: $D_k\times D_k\times A_{2m-2k-1}$, $2\le k< m$, $D_{m}\times D_{m}$, or $A_{2m-3}$, where $A$ acts on $D_\ell\times D_\ell$ as in Example \ref{ex:spin}, while on $A_{2m-2k-1}$ $\delta_1$ acts trivially, and $\delta_2$ by the nontrivial diagram automorphism.
\item If $A=\langle\delta_1\rangle$, $\overline K_\cO^\circ$ is of type: $D_k\times D_{2m-k}$, $2\le k< m$, or $D_{2m-1}$, where $\delta_1$ acts by the nontrivial automorphism of type $D_\ell$. These cases are covered by Example \ref{ex:O}.   
\item If $A=\langle\delta_2\rangle$ or $A=\langle\delta_1\delta_2\rangle$, $\overline K_\cO^\circ$ is of type $A_{2m-1}$ with the diagram automorphism action as in Example \ref{ex:A-E}.
  \item If $A=1$, then $\overline K_\cO^\circ$ is the hyperspecial subgroup of type $D_{2m-1}$.
  \end{enumerate}
\item If $\bG=\mathrm{PSO}_{4m+2}$, $m\ge 2$, $\Omega_G=\langle\delta\rangle\cong \bZ/4\bZ$:
   \begin{enumerate}
   \item If $A=\Omega_G$, $\overline K_\cO^\circ$ is of type: $D_k\times D_k\times A_{2m-2k}$, $2\le k\le m$, or $A_{2m-2}$,  where $\delta$ acts on $D_\ell\times D_\ell$ as in Example \ref{ex:spin}, while on $A_{2m-2k}$, $\delta$ acts by the nontrivial diagram automorphism.
   \item If $A=\langle\delta^2\rangle$, $\overline K_\cO^\circ$ is of type: $D_{2m}$ or $D_k\times D_{2m-k+1}$, $2\le k\le m$, where $\delta^2$ acts on each factor by the nontrivial automorphism of type $D_\ell$, as in Example \ref{ex:O}.
   \item If $A=1$, then $\overline K_\cO^\circ$ is the hyperspecial subgroup of type $D_{2m}$.
   \end{enumerate}
 \item Let $\bG$ be of type $E_6$ with $\Omega_G=\langle\delta\rangle\cong\bZ/3\bZ$.
   \begin{enumerate}
   \item If $A=1$, then $\overline K_\cO^\circ$ is either of type $E_6$ or $A_5\times A_1$.
   \item If $A=\Omega_G$, then  $\overline K_\cO^\circ$ is of type $A_2^3$ with $\delta$ acting by permutation, as in Example \ref{ex:GL};  $A_1^3\times A_1$, where $\delta$ permutes the first $3$ factors and it fixes the last one; or $D_4$, where $\delta$ acts as in Example \ref{ex:D4}.
   \end{enumerate}
 \item Let $\bG$ be of type $E_7$ with $\Omega_G=\langle\delta\rangle\cong\bZ/2\bZ$.
   \begin{enumerate}
   \item If $A=1$, then $\overline K_\cO^\circ$ is either of type $E_7$, $D_6\times A_1$, or $A_5\times A_2$.
      \item If $A=\Omega_G$, then  $\overline K_\cO^\circ$ is of type $E_6$ with $\delta$ acting by the nontrivial diagram automorphism as in Example \ref{ex:A-E}; $D_4\times A_1^2$, with $\delta$ acting by an order-$2$ diagram automorphism of $D_4$ (Example \ref{ex:O}) and by flipping the two $A_1$'s; $A_2^2\times A_2$, flipping the first $A_2$'s and acting trivially on the third $A_2$; $A_3^2\times A_1$, with the flip on $A_3^2$ and trivial action on $A_1$; or $A_7$ with the nontrivial diagram automorphism (see Examples \ref{ex:2-flip}, \ref{ex:A-E}).  
     \end{enumerate}
\end{enumerate}
      \end{ex}

\section{ Elliptic pairs}\label{s:ellipticpairs}

\subsection{Finite groups}
Suppose $H$ is a finite group. Given a finite-dimensional $H$-representation $(\delta,V_\delta)$ over $\bC$ and functions $f,f'\colon H\to\bC$, define 
\begin{equation}
(f,f')^\delta_{\ellip}=\frac 1{|H|} \sum_{h\in H}{\det}_{V_\delta}(1-\delta(h)) f(h^{-1}) f'(h).
\end{equation}
For $\rho,\rho'\in R(H)$, set
\[(\rho,\rho')^\delta_{\ellip}=(\chi_\rho,\chi_{\rho'})^\delta_{\ellip},
\]
where $\chi_\rho$, $\chi_{\rho'}$ denote the corresponding characters. The basic facts about $(~,~)^\delta_H$ can be found in \cite[\S2]{Re3}. An element $h\in H$ is called ($\delta$-)elliptic if $V_{\delta}^{\delta(h)}=0$. The set $H_\ellip$ of elliptic elements of $H$ is obviously closed under conjugation by $H$. Let $H\backslash H_\ellip$ denote the set of elliptic conjugacy classes. Fix $(\delta, V_\delta)$, and let $\overline R(H)$ be the quotient of $R(H)$ by the radical of the form $(~,~)^\delta_\ellip$. As in {\it loc. cit.}, there is a natural identification of $\overline R(H)$ with the space of class functions of $H$ supported on $H_{\ellip}$. 

For every $h\in H$, let $\mathbf 1_h$ denote the characteristic function of the conjugacy class of $h$. Clearly, $\{\mathbf 1_h\mid h\in H\backslash H_\ellip\}$ is an orthogonal basis of $\overline R(H)$ with respect to the elliptic pairing $(~,~)^\delta_\ellip$.

\smallskip

Suppose that, in addition, we are given an automorphism $\theta:H\to H$ . Let $\langle\theta\rangle$ denote the cyclic group generated by $\theta$ and $H'=H\rtimes\langle \theta\rangle$.  Given a finite-dimensional complex $H'$-representation $(\delta,V_\delta)$ and functions $f,f': H'\to\CC$, define 
\begin{equation}
(f,f')^\delta_{\theta-\ellip}=\frac 1{|H|} \sum_{h\in H}{\det}_{V_\delta}(1-\delta(\theta h)) f((\theta h)^{-1}) f'(\theta h).
\end{equation}
  For $\rho,\rho'\in R(H')$, set
\[(\rho,\rho')^\delta_{\theta-\ellip}=(\chi_\rho,\chi_{\rho'})^\delta_{\theta-\ellip},
\]
where $\chi_\rho$, $\chi_{\rho'}$ denote the corresponding characters. Note that if the representation $\delta$ is understood, we may write $(~, ~)_\ellip^H$ for $(~, ~)_\ellip^\delta$, and similarly for $(~, ~)_{\theta-\ellip}^H$.

\subsection{Complex reductive groups}\label{s:complex}

Let $\cG$ be a possibly disconnected complex reductive group with identity component $\cG^\circ$. If $x\in \cG$ is given, 
fix a Borel subgroup $B_x$ of $\rZ_\cG(x)^\circ$ and a maximal torus $T_x$ in $B_x$. Let $\mathfrak t_x$ be the Lie algebra of $T_x$.
As in \cite[\S2]{Wa} (see also  \cite[\S3.2]{Re3}), we define a complex representation $(\delta_x,\mathfrak t_x)$ of $A_{\cG}(x)$ as follows. 
Every $z\in \rZ_{\cG}(x)$ acts on $\rZ_\cG(x)^\circ$ by the adjoint action, denote by $\alpha_z$ the resulting automorphism of $\rZ_\cG(x)^\circ$. There exists $y\in \rZ_\cG(x)^\circ$ such that $\alpha_z\circ \Ad(y)$ preserves $B_x$ and $T_x$. This means that $\alpha_z\circ \Ad(y)$ defines an automorphism of the cocharacter lattice $X_*(T_x)$ in $\rZ_\cG(x)^\circ$, and therefore a linear isomorphism of $\mathfrak t_x$ denoted $\delta_x(z)$. If $\bar z\in A_{\cG}(x)$, let $z\in Z_{\cG}(x)$ be a representative and set $\delta_x(\bar z):=\delta_x(z)$. This construction gives a representation of $A_{\cG}(x)$. We consider the elliptic theory of the finite group $A_{\cG}(x)$ with respect to the representation $\delta_x$.

\smallskip

An element $g\in \cG$ is called {\it elliptic} if the centralizer $\rZ_\cG(g)$ contains no nontrivial torus.

\subsection{Definitions}\label{s:ell-pairs} Suppose $\Gamma$ is a (possibly disconnected) complex reductive group with identity component $\Gamma^\circ$. Extending the definition in Section \ref{s:finiteFT}, we define the sets  (cf. \cite[Def.~1.1]{Ciu})
\begin{equation}\label{e:Y-gamma}
\begin{aligned}
\cY(\Gamma)&=\{(s,h)\in \Gamma\times\Gamma\mid s,h \text{ semisimple}, \ sh=hs\},\\
\cY(\Gamma)_\ellip&=\{(s,h)\in \Gamma\times\Gamma\mid s,h \text{ semisimple}, \ sh=hs, \ \rZ_{\Gamma}(s,h)\text{ is finite}\}.
\end{aligned}
\end{equation}
Here $\rZ_{\Gamma}(s,h)=\rZ_\Gamma(s)\cap \rZ_\Gamma(h)$ and the finiteness condition is equivalent to saying that no nontrivial torus in $\Gamma$ centralizes both $s$ and $h$. We refer to elements of $\cY(\Gamma)_\ellip$ as {\it elliptic pairs.} Notice that the condition in $\cY(\Gamma)_{\ellip}$ is equivalent to saying that $h$ is elliptic in $\rZ_\Gamma(s)$ or equivalently $s$ is elliptic in $\rZ_\Gamma(h)$.

The sets $\cY(\Gamma),  \cY(\Gamma)_\ellip$ have  $\Gamma$-actions via conjugation: $g\cdot (s,h)=(gsg^{-1},ghg^{-1})$. They also have a natural $\Gamma$-equivariant involution given by the flip \[(s,h)\mapsto (h,s).\] 
Let $\Gamma\backslash \cY(\Gamma)$, $\Gamma\backslash \cY(\Gamma)_\ellip$ be the sets of $\Gamma$-orbits, and given $(s, h) \in \cY(\Gamma)$, write $[(s, h)]$ for the corresponding orbit in $\Gamma \backslash \cY(\Gamma)$. Then we get an involution
\[\mathsf{flip}\colon \Gamma\backslash \cY(\Gamma)\to \Gamma\backslash \cY(\Gamma),\ \mathsf{flip}([(s,h)])=[(h,s)],
\]
which preserves $\Gamma\backslash\cY(\Gamma)_\ellip$.

\begin{lem}\label{l:Y-gamma-finite}
 $\Gamma\backslash \cY(\Gamma)_\ellip$ is a finite set.
\end{lem}

\begin{proof}
Suppose $(s,h)\in  \cY(\Gamma)_\ellip$. The cyclic group $\langle s\rangle$ is in $\rZ_{\Gamma}(s,h)$, hence $s$ has finite order. Moreover, $s$ must be isolated in $\Gamma$ in the sense that $\rZ_\Gamma(s)$ does not contain a nontrivial central torus. The classification of  isolated semisimple automorphisms of $\Gamma$ is well known \cite{St,DM2}, in particular, there are finitely many automorphisms up to inner conjugation. 
\end{proof}

In the next lemma, we relate elliptic pairs in $\rZ_\Gamma(s)$ to elements of $A_\Gamma(s)$ that are elliptic with respect to the action described in Section \ref{s:complex}.

\begin{lem}\label{l:ell-comp} Fix $s\in \Gamma$ semisimple. The projection map $\rZ_\Gamma(s)\to A_\Gamma(s)$, $h\mapsto \bar h$, induces
a bijection between $\rZ_\Gamma(s)$-orbits of elliptic pairs $(s,h)$ and the elliptic conjugacy classes in $A_\Gamma(s)$.
\end{lem}
\begin{proof}
We need a result from the theory of semisimple automorphisms of reductive groups, e.g. \cite[Proposition 9]{Som}: if $x,y$ are semisimple elements in a reductive group $\cG$ such that their images in the group of components $\cG/\cG^\circ$ are in the same conjugacy class, and $S$ is a maximal torus in $\rZ_\cG(x)$, then there exist $g\in \cG$ and $s\in S$ such that $g y g^{-1}=xs$.

We apply this to $\cG=\rZ_\Gamma(s)$ (a reductive group).
Suppose $h,h'$ are semisimple elements such that $(s,h)$ and $(s,h')$ are elliptic pairs. The elliptic condition implies that the maximal torus in $\rZ_\cG(h)$ is trivial, hence $s=1$ in the relation above, and $h$ and $h'$ are $\cG$-conjugate. This implies that if $\overline h=\overline{h'}$ then $[(s,h)]=[(s,h')]$. 

It remains to show that $(s,h)$ is an elliptic pair if and only if $\bar h$ is elliptic in $A_\Gamma(s)$. This is just a matter of checking the definitions in the case $\cG=\rZ_\Gamma(s)$. Given the semisimple element $h\in \cG$, choose a maximal torus $T_s$ in $\cG$ that is normalized by $h$. Then $h$ is not elliptic if and only if there exists a nontrivial torus $S\subset T_s$ that centralizes $h$, equivalently if and only if $\delta_s(\bar h)$ fixes a nonzero element of $\mathfrak t_s$, i.e., if $\overline h$ is not elliptic in $A_\Gamma(s)$. 

\end{proof}

For every $(s,h)\in  \cY(\Gamma)$, define
\begin{equation}\label{e:stable-gamma}
\Pi(s,h)=\sum_{\phi\in \widehat {A_\Gamma(s)}} \phi(h)\phi\in R(A_{\Gamma}(s)),
\end{equation}
and let $\overline \Pi(s,h)$ denote the image in $ \overline R(A_{\Gamma}(s))$. Here $\phi(h)$ is interpreted as $\phi(\bar h)$ where $\bar h$ is the image of $h$ in $A_\Gamma(s)$. Let $ \CC[\cY(\Gamma)_\ellip]^\Gamma$ denote the $\Gamma$-invariant functions on $\cY(\Gamma)_\ellip$; this space can be identified with $\CC[\Gamma\backslash \cY(\Gamma)_\ellip]$. Let $\mathbf 1_{[(s,h)]}$ denote the characteristic function of the $\Gamma$-orbit of $(s,h)$.

\begin{prop}\label{p:ell-ident}
The correspondence $\mathbf 1_{[(s,h)]}\mapsto  \overline \Pi(s,h)$ induce an isomorphism
\[ \CC[\cY(\Gamma)_\ellip]^\Gamma\cong \bigoplus_{s\in \cC(\Gamma)_{\mathsf{ss}}} \overline R(A_\Gamma(s)).
\]
\end{prop}

\begin{proof} In light of Lemma \ref{l:ell-comp}, the only thing left is to remark that the elements $\overline \Pi(s,h)$ forms a basis of $\overline R(A_\Gamma(s))$ as $h$ ranges over a set of representatives of $\rZ_\Gamma(s)$-conjugacy classes such that $(s,h)$ is an elliptic pair. It is elementary that in $R(A_\Gamma(s))$,
\[\overline \Pi(s,h)= |\rZ_{A_\Gamma(s)}(\overline h)|~\mathbf 1_{\overline h^{-1}},
\]
and the claim follows.
\end{proof}

We say that $\Gamma_M\subset \Gamma$ is a {\it Levi subgroup} if there exists a torus $S\subset \Gamma^\circ$ such that $\Gamma_M=\rZ_\Gamma(S)$. If a pair $(s,h)$ is in $\cY(\Gamma_M)$, denote by $\Pi^{\Gamma_M}(s,h)$ the combination defined analogously to (\ref{e:stable-gamma}).

\begin{lem}\label{l:kappa} Suppose $s\in \Gamma_M$ is semisimple.
\begin{enumerate}
\item The inclusion $\rZ_{\Gamma_M}(s)\to \rZ_\Gamma(s)$ induces an inclusion $A_{\Gamma_M}(s)\to A_\Gamma(s)$.
\item For every $(s,h)\in \cY(\Gamma_M)$, $\Ind_{A_{\Gamma_M}(s)}^{A_\Gamma(s)}\Pi^{\Gamma_M}(s,h)=\Pi(s,h)$.
\end{enumerate}
\end{lem}

\begin{proof}
(1) This is a well-known argument. We need to show that $\rZ_{\Gamma_M}(s)\cap \rZ_{\Gamma}(s)^\circ$ is connected, and hence in $\rZ_{\Gamma_M}(s)^\circ$. But $\rZ_{\Gamma_M}(s)\cap \rZ_{\Gamma}(s)^\circ=\Gamma_M\cap \rZ_{\Gamma}(s)^\circ=\rZ_\Gamma(S)\cap  \rZ_{\Gamma}(s)^\circ=\rZ_{\rZ_{\Gamma}(s)^\circ}(S)$ which is connected since the centralizer of any torus in a connected reductive group is connected.

(2) This is elementary using that $\phi(\overline h)=\sum_{\psi\in \widehat{A_{\Gamma_M}(s)}} \langle \phi,\psi\rangle_{A_{\Gamma_M}(s)} \psi(\overline h)$ for every $\phi\in  \widehat{A_{\Gamma}(s)}$, by restriction of characters.
\end{proof}

\begin{lem} \label{l:kappa2}
Let $(s,h)\in \cY(\Gamma)$ be given and suppose $S$ is a maximal torus in $\rZ_\Gamma(s,h)^\circ$. Set $\Gamma_M=\rZ_\Gamma(S)$. Then $\rZ_{\Gamma_M}(s,h)^\circ=\rZ_{\Gamma_M}^\circ$, i.e., $(s,h)$ is an elliptic pair in $\Gamma_M/\rZ_{\Gamma_M}^\circ$.
\end{lem}

\begin{proof}
Let $S_1$ be a torus in $\rZ_{\Gamma_M}(s,h)^\circ$. Then $S_1\subset \rZ_\Gamma(s,h)^\circ$ and since it commutes with $S$ which is maximal in $ \rZ_\Gamma(s,h)^\circ$, it follows that $S_1\subset S\subset \rZ_{\Gamma_M}^\circ$.
\end{proof}

\begin{rem}
Our main application will be to consider $\Gamma=\Gamma_u$, the reductive part of the centralizer of a unipotent element $u$ in the Langlands dual group $G^\vee$, while $\Gamma_M$ will be the centralizer of $u$ in a Levi subgroup $M^\vee$.
\end{rem}

\subsection{Elliptic pairs in $\Gamma^\circ$} In applications, we will often encounter the situation where the group $\Gamma$ is connected. For this reason, it is useful to have a precise description of the elliptic pairs in $\Gamma^\circ$. Suppose $s\in \Gamma^\circ$ a semisimple element. Let $T$ be a maximal torus of {$\Gamma$} containing $s$ and let $\Phi$ be the system of roots of $T$ in $\Gamma^\circ$ and {$W(\Gamma^\circ)$} the Weyl group of $T$ in $\Gamma^\circ$. If $\alpha\in\Phi$, let $X_\alpha$ be the {corresponding} one-parameter unipotent subgroup in $\Gamma^\circ$. {For each $w\in W(\Gamma^\circ)$, we fix a representative $\dot w$} of $w$ in $\Nor_{\Gamma^\circ}(T)$. Recall \cite[Theorem 3.5.3]{Car}
\begin{equation}\label{e:cent-ss}
\begin{aligned}
Z_{\Gamma^\circ}(s)^\circ&=\langle T, X_\alpha\mid \alpha(s)=1,\ \alpha\in\Phi\rangle\\
Z_{\Gamma^\circ}(s)&=\langle T, X_\alpha, \dot w\mid \alpha(s)=1,\ \alpha\in\Phi,\ wsw^{-1}=s,\ w\in W(\Gamma^\circ)\rangle.
\end{aligned}
\end{equation}

We say that $w\in W(\Gamma^\circ)$ is \emph{elliptic} if $T^w$ is finite, equivalently if $\mathfrak t^w=0$, where $\mathfrak t$ is the Lie algebra of $T$.

\begin{prop}\label{p:conn}
With the notation as above, 
\[
\Gamma^\circ\backslash\cY(\Gamma^\circ)_\ellip\leftrightarrow W(\Gamma^\circ)\backslash\{(s,w)\mid s\in T\text{ is regular},\ w\in W(\Gamma^\circ) \text{ is elliptic},\ s\in T^w\}.
\] 
\end{prop}

\begin{proof} Since we are considering $\Gamma^\circ$-orbits of pairs $(s,h)\in \cY(\Gamma^\circ)_\ellip$, we may assume that $s\in T$ (in a fixed $W(\Gamma^\circ)$-orbit in $T$) and $h$ is in a semisimple conjugacy class of $Z_{\Gamma^\circ}(s)$. If $h\in Z_{\Gamma^\circ}(s)^\circ$, since $Z_{\Gamma^\circ}(s)^\circ$ is reductive (\cite[Theorem 3.5.4]{Car}), $h$ is contained in a maximal torus of $Z_{\Gamma^\circ}(s)^\circ$, hence $(s,h)$ is not an elliptic pair. This means that $h$ must be in $Z_{\Gamma^\circ}(s)\setminus Z_{\Gamma^\circ}(s)^\circ$. By (\ref{e:cent-ss}), we can assume that $h=\dot w$ 
for some $w\in W(\Gamma^\circ)$ such that $s\in T^w$. It is clear that $Z_{\Gamma^\circ}(s,\dot w)\supseteq T^w$, which means that $w$ is necessarily elliptic if $(s,\dot w)$ is an elliptic pair. Suppose $s$ is not regular. Then there exists $\alpha \in \Phi$ such that $\alpha(s)=1$. Let $O_w=\{\alpha, w(\alpha), w^2(\alpha),\dots, w^{n-1}(\alpha)\}$, where $n$ is the order of $w$. Then in the Lie algebra of $\Gamma$, there exists an appropriate sum of root vectors $e=\sum_{\beta\in O_w}e_\beta$ that is invariant under $\Ad(\dot w)$, and therefore $Z_{\Gamma^\circ}(s,\dot w)$ contains the one-parameter subgroup for $e$ and it is infinite. 

Conversely, suppose $(s,\dot w)$ is such that $s$ is regular and $w$ is elliptic. By (\ref{e:cent-ss}), $Z_{\Gamma^\circ}(s)= W({\Gamma^\circ})_s T$, where $W({\Gamma^\circ})_s=\{w_1\in W(\Gamma^\circ)\mid w_1 s w_1^{-1}=s\}.$ Then $Z_{\Gamma^\circ}(s,\dot w)$ is finite if and only if $\Ad(\dot w)$ has no nonzero fixed points on the Lie algebra $\mathfrak z_{\Gamma^\circ}(s)$. But $\mathfrak z_{\Gamma^\circ}(s)=\mathfrak t$, so this is equivalent with $w$ being elliptic.
\end{proof}

\begin{rem}
If $\Gamma$ is connected and simply connected, then $ \cY(\Gamma)_\ellip=\emptyset$. This is because in that case, for every regular semisimple $s \in T$, $Z_\Gamma(s)=Z_\Gamma(s)^\circ=T$, a maximal torus. 
\end{rem}

\section{ The dual nonabelian Fourier transform}\label{s:dualFT}

Let $\mathbf G$ be a connected semisimple algebraic $F$-group and $G=\mathbf G(F)$. Let $\fR_\un(G)$ denote the category of smooth unipotent representations of $G$. If $V,V'\in \Irr_\un (G)$, let 
\begin{equation}
\EP_G(V,V')=\sum_{i\ge 0}(-1)^i\dim \Ext^i(V,V'),
\end{equation}
where $\Ext^i(V,V')$ are calculated in the category $\mathfrak{R}(G)$ of all smooth $G$-representations (\cite{SS}), or equivalently, since $\fR_\un(G)$ is a direct summand of $\mathfrak{R}(G)$, in the category $\fR_\un(G)$.  We remark that this is a finite sum by Bernstein's result on the finiteness of the cohomological dimension of $G$.
 Extend, as we may, $\EP_G(~,~)$ as a Hermitian pairing on $R_\un (G)$ (as defined in Section \ref{s:unip}). Let $\overline R_\un(G)$ denote the quotient of $R_\un(G)$ by the radical of $\EP_G$. 

Let $R_\un^\temp (G)$ be the subspace spanned by the irreducible unipotent tempered representations and let $\overline R_\un^\temp(G)$ be the image of $R_\un^\temp (G)$ in $\overline R_\un(G)$. As it is well-known (\cite{SS,Re}), as a consequence of the (parabolic induction) Langlands classification:
\begin{equation}
\overline R_\un^\temp (G)=\overline R_\un(G).
\end{equation}

Let $\mathfrak B_\un(G)$ denote the unipotent Bernstein center so that $R_\un (G)=\bigoplus_{\mathfrak s\in \mathfrak B_\un(G)} R(G)^{\mathfrak s}$, where $R(G)^{\mathfrak s}$ is the $\CC$-span of irreducible objects in the subcategory $\mathfrak R(G)^{\mathfrak s}$ (defined in Section \ref{subsec:inertial}). Since there are no nontrivial extensions between objects in different Bernstein components, we have an $\EP$-orthogonal decomposition:
\[
\overline R_\un (G)=\bigoplus_{\mathfrak s\in \mathfrak B_\un(G)} \overline R(G)^{\mathfrak s}.
\]
With the same notation for a pure inner twist $G'$ of $G$, we get
\begin{equation}\label{e:ell-decomp-1}
\bigoplus_{G'\in\InnT^p(G)} \overline R_\un(G')=\bigoplus_{G'\in{{\InnT^p}}(G)}\bigoplus_{\mathfrak s\in \mathfrak B_\un(G')} \overline R(G')^{\mathfrak s}.
\end{equation}


\medskip

Recall the unipotent Langlands correspondence in the form (\ref{e:LLC-pure}). Given a semisimple element $s \in G^\vee$ and a unipotent element $u \in \cG_s^p$, apply the definitions of Section \ref{s:complex} to $u\in\cG_s^p$ to obtain a representation $\delta_u^s$ of $A_{\cG^p_s}(u)$ on the Cartan subalgebra $\mathfrak t^s_u$ in the Lie algebra of  $\rZ_{\cG^p_s}(u)$. Let $(~,~)^{\delta_u^s}_\ellip$ be the elliptic inner product on $R(A_{\cG^p_s}(u))$ and let $\overline R(A_{\cG^p_s}(u))$ be the elliptic quotient by the radical of the form. One expects the following correspondence to hold.

\begin{conj}\label{conj:main-ell}
The unipotent Langlands correspondence (\ref{e:LLC-pure}) induces an isometric isomorphism
\begin{equation}\label{e:ell-pure-LLC}
\overline{\mathsf{LLC}^p}_{\un}:\bigoplus_{s\in \cC(G^\vee)_{\mathsf{ss}}}\bigoplus_{u\in \cC(\cG_s^p)_\un}\overline R(A_{\cG_s^p}(u)) \longrightarrow \bigoplus_{G'\in\InnT^p(G)} \overline R_\un(G'),
\end{equation}
where the spaces on the left are endowed with the elliptic inner products $(~,~)^{\delta_u^s}_\ellip$, while the spaces on the right have the Euler--Poincar\'e pairings $\EP_{G'}$.
\end{conj}
Here $\cC(G^\vee)_{\mathsf{ss}}$ and $\cC(\cG_s)_\un$ refer to conjugacy classes of semisimple and unipotent elements, as defined in Section \ref{s:unip}.

\begin{rem}
In \cite{Re}, Reeder proves that this elliptic correspondence holds in the case of irreducible representations with Iwahori-fixed vectors of a split adjoint group. In Section \ref{s:elliptic}, Theorem \ref{t:main-elliptic}, we prove Conjecture \ref{conj:main-ell} in the form (\ref{e:ell-pure-LLC}) for a  semisimple adjoint $F$-split group $\bG$. In Section \ref{s:extend}, we explain how this result could be extended to arbitrary isogenies and, in particular, in Corollary \ref{c:Iwahori}  we prove it in the Iwahori-fixed case for an arbitrary $F$-split group $\bG$.
As a concrete example, in Proposition \ref{p:ell-SLn}, we also illustrate the result with a direct proof for $\bG=\SL_n$.
\end{rem}

\begin{rem}
Given $s \in G^\vee$ as above and $u \in \cG_s$, it also makes sense to define a representation of $A_{\cG_s}(u)$ similarly to $\delta_u^s$. It would be natural to expect that the local Langlands correspondence for inner twists as described in (\ref{e:LLC-iso}) also induces an isometric isometry
\begin{equation}
\overline{\mathsf{LLC}}_{\un}\colon \bigoplus_{s\in \cC(G^\vee)_{\mathsf{ss}}}\bigoplus_{u\in \cC(\cG_s)_\un}\overline R(A_{\cG_s}(u)) \longrightarrow \bigoplus_{G'\in\InnT(G)} \overline R_\un(G'),
\end{equation}
though we will not consider a conjecture of this form in this paper.
\end{rem}

\begin{rem}
    One can formulate Conjecture \ref{conj:main-ell} without restricting to the unipotent case. In general, the expected Langlands correspondence should induce an isometric isomorphism
    \begin{equation}\label{pure-ell-general}
    \overline{\mathsf{LLC}^p}\colon \bigoplus_{\varphi}\overline R(A_\varphi)\longrightarrow \bigoplus_{G'\in\InnT^p(G)} \overline R(G'),
    \end{equation}
    where $\varphi$ ranges over the $G^\vee$-conjugacy classes of L-parameters $\varphi: W_F'\to {}^LG$ (equivalently, tempered L-parameters), and $A_\varphi = \pi_0(\rZ_{G^\vee}(\varphi))$ (cf. Section \ref{s:unip}). The elliptic theory of the finite group $A_\varphi$ is taken with respect to the action on a Cartan subalgebra of $\rZ_{G^\vee}(\varphi)$ as before. This formulation is of course related to Arthur's ideas on elliptic representations: in \cite[Corollary 6.3]{ArEll}, Arthur proved the equality of Kazhdan's elliptic pairing between irreducible tempered representations with the elliptic pairing of the corresponding irreducible characters of the Knapp--Stein R-groups. Later, Opdam and Solleveld \cite[Theorems 6.5 and 7.3]{OS2} extended this work to all admissible representations using the homological Euler--Poincar\' e pairing. Moreover, Arthur \cite[\S 7]{ArAutom} expected an identification between the R-groups and the geometric A-groups, and in fact, Reeder \cite[\S 8]{Re}, as part of his proof of the elliptic correspondence, proved this matching in the Iwahori case. 
\end{rem}

\subsection{The elliptic Fourier transform: the split case}\label{s:ell-FT} Suppose $G$ is the split $F$-form. In order to apply the ideas in Section \ref{s:ell-pairs}, we rephrase the left hand side of (\ref{e:ell-pure-LLC}). Since $\Frob$ acts trivially on $G^\vee$, in this situation we have $\cG_s^p=\rZ_{G^\vee}(s)$, and hence
\[
\bigoplus_{s\in \cC(G^\vee)_{\mathsf{ss}}}\bigoplus_{u\in \cC(\cG_s^p)_\un}\overline R(A_{\cG_s^p}(u)) =\bigoplus_{s\in \cC(G^\vee)_{\mathsf{ss}}}\bigoplus_{u\in \cC(\rZ_{G^\vee}(s))_\un}\overline R(A_{G^\vee}(su)),
\]
which can be written as
\[\bigoplus_{u\in \cC(G^\vee)_\un}\bigoplus_{s\in \CC(\Gamma_u)_{\mathsf{ss}}}\overline R(A_{\Gamma_u}(s))=\bigoplus_{u\in \cC(G^\vee)_\un}\CC[\cY(\Gamma_u)_\ellip]^{\Gamma_u},
\]
via Proposition \ref{p:ell-ident}, where $\Gamma_u$ is the reductive part of $\rZ_{G^\vee}(u)$, as before. For simplicity we define
\begin{equation} \label{eqn:Ruep}
 \cR_{\un,\ellip}^p(G)=\bigoplus_{G'\in\InnT^p(G)} \overline R_\un(G'),
\end{equation}
endowed with the Euler--Poincar\'e pairing $\EP=\bigoplus_{G'} \EP_{G'}$. Hence the elliptic unipotent local Langlands correspondence for pure inner twists of a split group can be viewed as the isomorphism
\begin{equation}\label{e:ell-LLC}
\overline{\mathsf{LLC}^p}_{\un}:\bigoplus_{u\in \cC(G^\vee)_\un}\CC[\cY(\Gamma_u)_\ellip]^{\Gamma_u}\longrightarrow \cR_{\un,\ellip}^p(G).
\end{equation}
For every class of elliptic pairs $[(s,h)]\in\Gamma_u\backslash \cY(\Gamma_u)_\ellip$, define the virtual combination (cf. \cite{Wa2,Ciu}):
\begin{equation} \label{eqn:pi_ush} \
\Pi(u,s,h)=\sum_{\phi\in \widehat {A_{\Gamma_u}(s)}}{\phi(h)}\,\pi(s,u,\phi).
\end{equation} 
Regard $\Pi(u,s,h)$ (or rather its image) as an element in $ \cR_{\un,\ellip}^p(G)$. As before $\phi(h)=\phi(\bar h)$, where $\bar h$ is the image of $h$ in $A_{\Gamma_u}(s)$.

\begin{lem}\label{l:EP-stable}
With notation as above and setting $A_x=A_{G^\vee}(x)$, we have:
\begin{enumerate}
\item $\EP(\Pi(u,s,h),\Pi(u',s',h'))=0$ if $x=su$ and $x'=s'u'$ are not $G^\vee$-conjugate;
\item \begin{align*}\EP(\Pi(u,s,h),\Pi(u,s,h'))&=(|\rZ_{A_x}(\bar h)| \mathbf 1_{\bar h^{-1}},|\rZ_{A_x}(\bar h')| \mathbf 1_{\bar {h'}^{-1}})_\ellip^{A_x}\\&=\begin{cases}|Z_{A_x}(\bar h)| {\det}_{\mathfrak t_x}(1-\bar h^{-1}),&\text{ if }h,h'\text{ are conjugate.}\\0,&\text{otherwise},\end{cases}\end{align*}
\end{enumerate}
Hence, the combinations $\{\Pi(u,s,h)\}$ define an orthogonal basis of $\cR_{\un,\ellip}^p(G)$.
\end{lem}

\begin{proof}
This is a straightforward consequence of Theorem \ref{t:main-elliptic} (and an elementary calculation for the last equality in (ii)).
\end{proof}

\begin{defn} \label{defn:FTell} {\rm (cf. \cite{Ciu}, \cite{Wa2}).} The (dual) \emph{elliptic nonabelian Fourier transform} is the involutive linear map 
$\FT^\vee_\ellip\colon \cR_{\un,\ellip}^p(G) \to\cR_{\un,\ellip}^p(G)$, defined by
\[\FT^\vee_\ellip(\Pi(u,s,h))=\Pi(u,h,s), \quad\text{$(s,h)\in \Gamma_u\backslash\cY(\Gamma_u)_\ellip$, $u\in G^\vee$ unipotent.}\]
Notice that $\FT^\vee_\ellip$ is the just the image under $\overline{\mathsf{LLC}^p}_\un$ of the canonical involution of $\bigoplus_{s,u} \overline R(A_{su})$ 
\[
[u,s,h]\mapsto [u,h,s],\quad \text{ where }\ \  [u,s,h]:=
\sum_{\psi\in \widehat A_{su}} \psi(h)\psi.\]
\end{defn}

For every $G'\in \InnT^p(G)$, and $K'_\cO\in\max(G')$ consider the restriction map
    \begin{equation}
    \res_{K'_\cO}\colon \Irr_\un G'\to R_\un(\overline K_\cO),\ V\mapsto V^{K'^+_\cO}.
    \end{equation}
    We define a linear map $ \res_{\cpt,\un}\colon \oplus_{G' \in \InnT^p(G)} R_\un(G') \to \cC(G)_{\cpt, \un}$ by setting 
    \begin{equation*}
        \res_{\cpt, \un}(V) = \sum_{K'_{\cO} \in \max(G')} \res_{K'_\cO}(V)
    \end{equation*}
    for all $G' \in \InnT^p(G)$ and $V \in \Irr_\un(G')$.
  With notation as in Section \ref{s:maximal}, for each $(A,\cO) \in S_{\max}(G)$, we let $\text{proj}_\cO$ be the projection map $\cC(G)_{\cpt, \un} \to \oplus_{x \in A} R_\un(\overline K_{x, \cO})$ with respect to the decomposition (\ref{e:AO-decomp}), and let $\res_\cO = \text{ proj}_\cO \circ \res_{\cpt, \un}$. We have
    \begin{equation}
    \res_{\cpt,\un}=\bigoplus_{(A, \cO) \in S_{\max}(G)} \res_\cO.
    \end{equation}
  
 
 We can now formulate the conjecture for elliptic representations.
 
 \begin{conj}\label{c:elliptic} Let $G$ be a simple $F$-split group. Consider the following diagram:
  \begin{displaymath}\xymatrix@+1pc{
    {\cR_{\un,\ellip}^p(G)} \ar[r]^{\mathrm{FT}^\vee_\ellip} \ar[d]_{\res_{\cpt,\un}}
    & {\cR_{\un,\ellip}^p(G)}\ar[d]^{\res_{\cpt,\un}}\\
    {  \cC(G)_{\cpt,\un}} \ar[r]_{\mathrm{FT_{\cpt,\un}}}
    & {  \cC(G)_{\cpt,\un}} }
\end{displaymath} 
For every  unipotent element $u\in G^\vee$, elliptic pair $(s,h) \in \cY(\Gamma_u)_\ellip$, and maximal compact open subgroup $K_\cO$ of $G$, there exists a root of unity $\zeta=\zeta(u,s,h,\cO)$ such that
\[\res_\cO (\Pi(u,h,s))=\zeta \cdot (\FT_{\cpt,\un}\circ \res_\cO)(\Pi(u,s,h)). \]
\end{conj}

 \begin{rem}\label{r:roots}
 If $K_\cO$ is the maximal hyperspecial compact subgroup of $G$, so that in particular $\res_\cO=\res_{K_\cO}$, we expect that the only roots of unity $\zeta$ that appear are the well-known $\Delta(\bar x_\rho)\in \{\pm 1\}$ (see \cite[\S6.7]{Lubook}) for certain families of unipotent representations of the finite groups of types $E_7$ and $E_8$. But for other maximal compact subgroups, Proposition \ref{p:sl} shows that in $\SL_n$ for example, new roots of unity can appear. 
 
We remark that extra roots of unity already appear in relation with the nonabelian Fourier transform for finite reductive groups, although we do not know if this is a related issue. In that setting there are three bases of the Grothendieck group of unipotent characters: 
\begin{enumerate}
    \item[(1)] the irreducible characters, 
    \item[(2)] the ``almost characters" which by definition are the image of the basis of irreducible characters under Lusztig's nonabelian Fourier transform, and
    \item[(3)] the traces of the Frobenius on unipotent character sheaves.
\end{enumerate}
Lusztig's conjecture states that each element of the basis (2) equals an element of the basis (3) times a root of unity. Determining these roots of unity is a difficult question, see Shoji \cite{shoji-unipotent} for classical groups, also Hetz's recent thesis \cite{hetz-2023} for progress on the exceptional groups. 
 \end{rem}

\begin{rem}
Note that the definition of the linear combinations $\Pi(u, s, h)$, and thus the definition of $\FT_{\ellip}^\vee$, depends on the (in general, non-canonical) map $\mathsf{LLC}_{\un}^p$. It is an interesting question to understand how Conjecture \ref{c:elliptic} depends on $\mathsf{LLC}_{\un}^p$, and more specifically, what choices one must make in constructing $\mathsf{LLC}_{\un}^p$ for the conjecture to be true. In the cases of the conjecture proved below, these subtleties do not arise. 
\end{rem}

\subsection{Regular unipotent elements} In Section \ref{s:type-A}, we will verify this conjecture completely when $G=\SL_n$ and $\PGL_n$, but here we illustrate it in the case when $u$ is a regular unipotent element.

\begin{prop}\label{p:regular}
Let $u_r\in G^\vee$ be a regular unipotent element. Then 
\[\res_{\cpt, \un} (\Pi(u_r,h,s))= \FT_{\cpt,\un}\circ \res_{\cpt, \un}(\Pi(u_r,s,h)).
 \]
 for all $(s, h) \in \cY(\Gamma_{u_r})$. In particular, Conjecture \ref{c:elliptic} holds with trivial roots of unity.
\end{prop}

\begin{proof}
In this case $\Gamma_{u_r}=\rZ_{G^\vee}$ and every pair $(s,h)$ in $\cY(\rZ_{G^\vee})$ is elliptic. 
Write the natural identification
\begin{equation*}
    \Omega_G \overset{\sim}\longrightarrow \widehat{\rZ_{G^\vee}}
\end{equation*}
as $x \mapsto \phi_x$. Then for $(s, h) \in \cY(\Gamma_{u_r})$, we have
$\Pi(u_r,s,h)=\sum_{x \in \Omega} \phi_x(h) \pi(s,u_r,\phi_x)$. 
Note that $\pi(1, u_r, \phi_x)$ is the Steinberg representation $\St_{G_x}$ of $G_x$, so $\pi(s, u_r, \phi_x) \simeq \St_{G_x} \otimes \chi_s$, where $\chi_s$ is the weakly unramified character corresponding to $s$ under (\ref{eqn:dual_unr}). This follows from the fact that $\mathsf{LLC}_{\un}$ is equivariant for the action of weakly unramified characters, cf. \cite[Theorem 1(b)]{So1}.

For the rest of the proof, we fix $(A, \cO) \in S_{\max}(G)$. Then given $s \in \rZ_{G^\vee}$, the character $\chi_s$ is trivial on the parahoric $K_{x, \cO}^\circ$ so defines a character, call it $\sigma_s$, of $K_{x, \cO}/K_{x, \cO}^\circ = \overline K_{x, \cO}/\overline K_{x, \cO}^\circ$.
We have  
\[\res_{\cO}\pi(s,u_r,\phi_x)=\begin{cases}0 &\text{ if } x \notin A,\\ \St_{\overline{K}_{x, \cO}^\circ}\rtimes \sigma_s &\text{ if } x \in A,\end{cases}
\]
where $\St_{\overline{K}_{x, \cO}^\circ}$ is the Steinberg character of the finite group $\overline K_{\cO,x}^\circ$. 
Note that that for every $x \in A$
\[\phi_x(h)=\sigma_h(x)\text{ for all }h\in \rZ_{G^\vee}.
\]
Thus
\begin{equation}\label{e:resO-ur}
\res_\cO\Pi(u_r,s,h)
= \sum_{x\in A} \sigma_h(x) 
\St_{\overline K_{x, \cO}^\circ}\rtimes \sigma_s.
\end{equation}
With notation as in Section \ref{s:Lusztig}, let $\cU_{\cO,\St}=\{\St_{K_\cO^\circ}\}$ be the Steinberg family in $\Irr_\un(\overline K_\cO)$, and let 
$\widetilde\cU_{\cO,\St} \subset \cup_{x \in A} \Irr_\un(\overline K_{x, \cO})$ be the family 
parametrized by $\widetilde{\Gamma}_{\cU_{\cO, \St}}^A = A$ under the bijection of Proposition \ref{p:disconnected}. 
Then by (\ref{e:resO-ur}), $\res_\cO \Pi(u_r, s, h)$ corresponds to $\Pi_{\widetilde\cU_{\cO,\St}}(\sigma_s,\sigma_h)$ defined as in (\ref{e:pi-U}).
The claim then follows from Lemma \ref{l:finite-flip}.
\end{proof}

\section{A definition of the elliptic Fourier transform \`a la Lusztig}\label{sec:Lusztig}

We present an alternative definition of the elliptic nonabelian Fourier transform (cf. Definition \ref{defn:FTell}) along the lines of Lusztig's pairing defined in \cite[\S1]{LualmostII}. Retain the notation from Section \ref{s:ellipticpairs}. In particular, $\Gamma$ is a complex reductive group, not necessarily connected, and $\cY(\Gamma)_\ellip$ is the set of elliptic semisimple pairs in $\Gamma$. Let $\Sigma$ be the set of semisimple elements of $\Gamma$. Extending the definition in Section \ref{s:finiteFT}, we let
\[\cM(\Gamma)= \{(x,\sigma)\mid x\in \Sigma,~\sigma\in \widehat{A_\Gamma(x)}\},
\]
modulo the equivalence relation given by conjugation by $\Gamma$.

For any two semisimple elements $x,y\in\Gamma$, define the set
\begin{equation}
\cA_{x,y}=\{z\in \Gamma\mid zxz^{-1}\in Z_\Gamma(y)\}
\end{equation}
with an action of $Z_\Gamma(x)^\circ\times Z_\Gamma(y)^\circ$ by $(\gamma,\gamma')\cdot z=\gamma' z \gamma^{-1}$. Let $^0\!\cA_{x,y}$ denote the set of orbits. By \cite[Lemma 1.2]{LualmostII}, this is a finite set. It is clear that $\cA_{y,x}=\cA_{x,y}^{-1}$.

Consider the subset
\begin{equation}
\cA_{x,y}^\ellip=\{z\in \Gamma\mid (zxz^{-1},y)\in \cY(\Gamma)_\ellip\},
\end{equation}
and let $^0\!\cA_{x,y}^\ellip$ be the corresponding finite set of $Z_\Gamma(x)^\circ\times Z_\Gamma(y)^\circ$-orbits. Following \cite[\S1.3]{LualmostII}, we suppose $\kappa : \Sigma \times \Sigma \to \bR$ is a nonnegative function 
satisfying
\[\kappa(x',y')=\kappa(y',x'),\quad \kappa(\gamma x'\gamma^{-1},\gamma' y'(\gamma')^{-1})=\kappa(x',y'),\quad \kappa(\zeta x',\zeta' y')=\kappa(x',y'),
\]
for all $x', y' \in \Sigma, \gamma, \gamma' \in\Gamma$, $\zeta,\zeta'\in \rZ_\Gamma$. The following definition is an elliptic analogue of \cite[\S1.3 (a)]{LualmostII}.

\begin{defn}\label{d:L-elliptic} For $(x,\sigma),(y,\tau)\in \cM(\Gamma)$, set
\[\{(x,\sigma),(y,\tau)\}_\ellip={\kappa(x,y)}\sum_{z\in ^0\!\cA_{x,y}^\ellip} \tau(zx^{-1} z^{-1}) \sigma(z^{-1} y z).
\]
It is immediate that $\{(x,\sigma),(y,\tau)\}_\ellip = \overline{\{ (y, \tau), (x, \sigma)\}}_\ellip$ for all $(x, \sigma), (y, \tau) \in \cM(\Gamma)$. Moreover, if $x\in \Sigma$ is such that $x$ does not belong to any elliptic pair, then it is clear that $(x,\sigma)$ is in the radical of $\{~,~\}_\ellip$ for all $\sigma$.
\end{defn}
Let $\mathbf V = \CC[\cM(\Gamma)]$ denote the $\CC$-span of $\cM(\Gamma)$, and given $(x, \sigma) \in \cM(\Gamma)$, write $[(x, \sigma)]$ for its image in $\mathbf{V}$. Then $\{~, ~\}_\ellip$ extends to a Hermitian pairing on $\mathbf{V}$. Similarly to (\ref{e:pi-U}), for every pair of commuting elements $x,y\in\Sigma$, denote
\begin{equation}
\Pi(x,y)=\sum_{\sigma\in \widehat {A_\Gamma(x)}} \sigma(y^{-1}) [(x,\sigma)]\in \mathbf V.
\end{equation}
Let $\bar y$ denote the image of $y$ in $A_\Gamma(x)$. Let $C_{\rZ(x)}(y)$ denote the conjugacy class of $y$ in $\rZ_\Gamma(x)$, and $C_{A(x)}(\bar y)$ the conjugacy class of $\bar y$ in $A_\Gamma(x)$. Since $x$ and $y$ commute, we have that $\cA_{x, y}$ contains the set $\rZ_\Gamma(y)\rZ_\Gamma(x)$, and the action of $\rZ_\Gamma(x)^\circ \times \rZ_\Gamma(y)^\circ$ on $\cA_{x, y}$ restricts to an action on $\rZ_\Gamma(y)\rZ_\Gamma(x)$. Let 
$^0\!\cB_{x,y}$ denote a set of orbit representatives for this restricted action.

\begin{lem}\label{l:first-rel}
For every $(t,\tau)\in \cM(\Gamma)$,
\[\{\Pi(x,y),(t,\tau)\}_\ellip=\begin{cases} {\kappa(x,y)|}{|\rZ_{A(x)}(\bar y)|} \sum_{z\in ^0\!\cB_{x,y}} \tau(zx^{-1} z^{-1}) & \text{ if } t=y\text{ and } (x,y)\in \cY(\Gamma)_\ellip,\\
0 &\text{ otherwise}.
\end{cases}
\]
\end{lem}

\begin{proof}
This is similar to the calculation for the nonabelian Fourier transform in the case of finite reductive groups. We compute:
\begin{align*}
\{\Pi(x,y),(t,\tau)\}_\ellip&= \sum_{\sigma\in \widehat {A_\Gamma(x)}} \sigma(y^{-1}) \{(x,\sigma),(t,\tau)\}_\ellip\\
&={\kappa(x,t)}\sum_{z\in ^0\!\cA_{x,t}^\ellip} \tau(zx^{-1} z^{-1})\left(\sum_{\sigma\in \widehat {A_\Gamma(x)}}\sigma(y^{-1}) \sigma(z^{-1} t z)\right).
\end{align*}
Fix $z \in {}^0\!\cA_{x, t}^\ellip$. By the orthogonality relations for characters, $\sum_{\sigma\in \widehat {A_\Gamma(x)}}\sigma(y^{-1}) \sigma(z^{-1} t z)=0$ unless the image of $z^{-1}tz$ in $A_\Gamma(x)$ is conjugate to $\bar y$, in which case, it equals ${|\rZ_{A(x)}(\bar y)|}$. In addition, $(x,z^{-1}tz)\in\cY(\Gamma)_\ellip$, which implies by Lemma \ref{l:ell-comp} that the image of $z^{-1}tz$ in $A_\Gamma(x)$ is in an elliptic class, so $\bar y$ is elliptic in $A_\Gamma(x)$, and so again by Lemma \ref{l:ell-comp}, $(x,y)\in \cY(\Gamma)_\ellip$. Moreover, by the same result, the image of $z^{-1}tz$ in $A_\Gamma(x)$ is conjugate to $\bar y$ if and only if $z^{-1}tz$ is conjugate to $y$ in $\rZ_\Gamma(x)$. In particular, this means that $t$ is conjugate to $y$ in $\Gamma$. This proves that $\{\Pi(x,y),(t,\tau)\}_\ellip = 0$ unless $t = y$ and $(x,y)\in \cY(\Gamma)_\ellip$. (Since $\cM(\Gamma)$ consists of $\Gamma$-orbits, we may identify $t=y$.)

To complete the proof, assume $(x, y) \in \cY(\Gamma)_\ellip$. We have
\[\{\Pi(x,y),(y,\tau)\}_\ellip={\kappa(x,y)}{|\rZ_{A(x)}(\bar y)|} \sum_{\substack{z\in ^0\!\cA_{x,y}^\ellip\\z^{-1}yz\in C_{Z(x)}(y)}} \tau(zx^{-1} z^{-1}).
\] 
We analyze the index of summation. Notice that $z^{-1}yz\in C_{\rZ(x)}(y)$ is equivalent to $z\in \rZ_\Gamma(y)\rZ_\Gamma(x)$. Since $(x,y)$ is an elliptic pair, it is also automatic that $(zxz^{-1}, y)$ is for any $z \in \rZ_\Gamma(y)\rZ_\Gamma(x)$, hence $z$ ranges over representatives of $\rZ_\Gamma(x)^\circ \times \rZ_\Gamma(y)^\circ$-orbits in $ \rZ_\Gamma(y)\rZ_\Gamma(x)$. 
\end{proof}

\begin{prop}\label{p:second-rel}
For $(x,y)\in \cY(\Gamma)_\ellip$:
\begin{enumerate}
\item[(a)] \[\sum_{\tau\in \widehat {A_\Gamma(y)}} \{\Pi(x,y),(y,\tau)\}_\ellip [(y,\tau)]=\kappa(x,y) |\rZ_{A(x)}(\bar y)| |^0\!\cB_{x,y}|\cdot \Pi(y,x)\text{ in }\mathbf V.
\]
\item[(b)] \[\sum_{h\in A_\Gamma(y)} \{\Pi(x,y),\Pi(y,h)\}_\ellip \Pi(y,h) = \kappa(x,y) |A_\Gamma(y)| |\rZ_{A(x)}(\bar y)| |^0\!\cB_{x,y}|\cdot \Pi(y,x).\]
\end{enumerate}
\end{prop}

\begin{proof}
(a) Denote by $\tau^{z^{-1}}\in \widehat {A_\Gamma(y)}$ the twist of $\tau$ by $z^{-1}$, so $\tau^{z^{-1}}(a)=\tau(z a z^{-1})$.  Applying the previous lemma, we get
\begin{align*}
\frac 1{|\kappa(x,y)||\rZ_{A(x)}(\bar y)|}&\sum_{\tau\in \widehat {A_\Gamma(y)}} \{\Pi(x,y),(y,\tau)\}_\ellip [(y,\tau)]=\sum_{z\in ^0\!\cB_{x,y}} \sum_{\tau\in\widehat{A_\Gamma(y)}} \tau^{z^{-1}}(x^{-1})[(y,\tau)]\\
&=\sum_{z\in ^0\!\cB_{x,y}} \sum_{\tau\in\widehat{A_\Gamma(y)}} \tau^{z^{-1}}(x^{-1})[(z^{-1}yz,\tau^{z^{-1}})]=\sum_{z\in ^0\!\cB_{x,y}}\sum_{\tau'\in \widehat{A_\Gamma(z^{-1}yz)}} \tau'(x^{-1}) [(z^{-1}yz,\tau')]\\
&=\sum_{z\in ^0\!\cB_{x,y}}\Pi(z^{-1}yz,x)=|^0\!\cB_{x,y}|~\Pi(y,x),
\end{align*}
since $z^{-1}yz$ is conjugate to $y$ in $\rZ_\Gamma(x)$.

(b) This is immediate from (a) using $\sum_{h\in A_\Gamma(y)} \tau(h^{-1}) \tau'(h)=|A_\Gamma(y)|\delta_{\tau,\tau'}$, $\tau,\tau'\in\widehat{A_\Gamma(y)}$.
\end{proof}

Consider the set
\begin{equation}\label{e:ell-basis}
B_{\mathbf V}=\{|C_{A(y)}(\bar h)|^{1/2} \Pi(y,h)\mid (y,h)\in \Gamma\backslash \cY(\Gamma)_\ellip\}.
\end{equation}
By Lemma \ref{l:first-rel}, $B_{\mathbf V}$ spans $\bar{\mathbf V}_\ellip:=\mathbf V/\ker\{~,~\}_\ellip$, where $\ker\{~,~\}_\ellip$ denotes the radical of the pairing. Define
\begin{equation}
\cF'_\ellip: \mathbf V\to \mathbf V,\quad \cF'_\ellip(v)=\sum_{b\in B_{\mathbf V}} \{v,b\} b.
\end{equation}
Clearly, $\cF'_\ellip$ descends to a linear map on $\bar{\mathbf V}_\ellip$. 

\begin{cor}
For every $(x,y)\in \cY(\Gamma)_\ellip$,
\[\cF'_\ellip(\Pi(x,y))=\kappa(x,y) |A_\Gamma(y)| |\rZ_{A(x)}(\bar y)| |^0\!\cB_{x,y}|\cdot \Pi(y,x)
\]
\end{cor}

\begin{proof}
From Lemma \ref{l:first-rel}, we have $\cF'_\ellip(\Pi(x,y))=\sum_h |C_{A(y)}(\bar h)| \{\Pi(x,y),\Pi(y,h)\}_\ellip [\Pi(y,h)]$, where the sum is over a set of representatives $h$ of the conjugacy classes of elliptic semisimple elements in $\rZ_\Gamma(y)$, equivalently, elliptic conjugacy classes in $A_\Gamma(y)$. So we may rewrite
\[\cF'_\ellip(\Pi(x,y))=\sum_{h\in A_\Gamma(y)_\ellip}  \{\Pi(x,y),\Pi(y,h)\}_\ellip \Pi(y,h)=\sum_{h\in A_\Gamma(y)}  \{\Pi(x,y),\Pi(y,h)\}_\ellip \Pi(y,h),
\]
using Lemma \ref{l:first-rel} again. Then the claim follows from Proposition \ref{p:second-rel}.
\end{proof}

In other words, $\cF'_\ellip$ acts, up to a scalar multiple, by the flip on each $\Pi(x,y)$. If we set 
\begin{equation}
\kappa(x,y)=\frac{|C_{A(x)}(\bar y)|^{1/2} |C_{A(y)}(\bar x)|^{1/2}}{|A_\Gamma(x)||A_\Gamma(y)| |^0\!\cB_{x,y}|},
\end{equation}
then 
\begin{equation}
{\cF'_\ellip}^2=\Id.
\end{equation}
\begin{rem}
When we specialize to $\Gamma=\Gamma_u$ for a unipotent element $u \in \Gamma$, we see from Lemma \ref{l:EP-stable} that $B_{\mathbf V}$ is in fact an orthogonal basis of $\bar{\mathbf V}_\ellip$ with respect to the Euler--Poincar\'e pairing.
\end{rem}

\section{Elliptic unipotent representations}\label{s:elliptic}

The main result of this section is:

\begin{thm}\label{t:main-elliptic}
Suppose $G$ is a semisimple split $F$-group of adjoint type. Then Conjecture \ref{conj:main-ell} holds for all pure inner twists of $G$.
\end{thm}

\begin{rem}
In Section \ref{s:extend}, we explain how Theorem \ref{t:main-elliptic} can be extended to other isogenies. See in particular Corollary \ref{c:Iwahori}.
\end{rem}
The strategy of the proof is as follows.
As explained in Section \ref{sec:LLC}, the set of unipotent enhanced Langlands parameters $\Phi_{\enh,\un}(G')$, where $G'$ is an inner twist of $G$, decomposes into a disjoint union $\Phi_{\enh,\un}(G')=\bigsqcup_{{\mathfrak s}^\vee\in \mathfrak B_\un^\vee(G')} \Phi_{\enh}(G')^{\mathfrak s^\vee}$. Consequently, there is a decomposition
\[R(\Phi_{{{\enh}},\un}(G'))=\bigoplus_{{\mathfrak s}^\vee\in \mathfrak B_\un^\vee(G')} R(\Phi_{{\enh}}(G')^{\mathfrak s^\vee}),
\]
where $R(\Phi_{{{\enh}},\un}(G'))$ and $\Phi_{\enh}(G')^{\mathfrak s^\vee}$ are defined analogously to $R(\Phi_{{{\enh}},\un}({}^LG))$ (see (\ref{e:LLC-iso})). 
In \cite{LuI} (for adjoint simple groups) and later in \cite{AMS3} (for arbitrary groups), an affine Hecke algebra with possibly unequal parameters $\cH(\mathfrak s^\vee)$ is constructed such that there is a  bijection
\begin{equation}
\Irr ~\cH(\mathfrak s^\vee)\longleftrightarrow  \Phi_{{\enh}}(G')^{\mathfrak s^\vee}
\end{equation}
which induces a  linear isomorphism
\[R(\cH(\mathfrak s^\vee))\cong R(\Phi_{{\enh}}(G')^{\mathfrak s^\vee}).
\]
We need to study the elliptic space $\overline R(\cH(\mathfrak s^\vee))$. 
The important fact for the elliptic theory is that $\cH(\mathfrak s^\vee)$ is a deformation of an extended affine Weyl group $\wti W_{\mathfrak s^\vee}=W_{\mathfrak s^\vee}\ltimes X^*(T_{\mathfrak s^\vee})$, where $T_{\mathfrak s^\vee}=\Phi_{\enh}(L')^{\fs_{L'}^\vee}$  for $L'$ a Levi subgroup of $G'$ that corresponds to $\fs^\vee$. This allows us to use the results of \cite{OS} to further reduce to $\overline R(\cH(\mathfrak s^\vee))\cong \overline R(\wti W_{\mathfrak s^\vee})$. Moreover, the latter space is equivalent to a direct sum of elliptic spaces for certain finite groups
\[\overline R(\wti W_{\mathfrak s^\vee})\cong \bigoplus_{s\in W_{\mathfrak s^\vee}\backslash T_{\mathfrak s^\vee}} \overline R(\rZ_{W_{\mathfrak s^\vee}}(s)).
\]
We then use results of \cite{Wa} and the generalized Springer correspondence to relate the spaces $\overline R(\rZ_{W_{\mathfrak s^\vee}})(s)$ to the relevant spaces of Langlands parameters (for the various unipotent elements) in $\Phi_{\enh,\un}(G')^{\mathfrak s^\vee}$.

Finally, by \cite{LuI,So1} for each $\fs^\vee\in\fB_\un^\vee(G')$, there exists  $\mathfrak s\in\fB_\un(G')$ that is sent to $\fs^\vee$ by the LLC, and then the Hecke algebra $\cH(\fs)$ for $\mathfrak s$ is isomorphic to $\cH(\mathfrak s^\vee)$. The fact that the elliptic space for the representations in the block $\fR(G')^\fs$ is naturally isomorphic to $\overline R(\cH(\mathfrak s))$ is immediate by the exactness of the equivalence of categories between $\fR(G')^\fs$ and $\cH(\mathfrak s)$-modules.

\subsection{Euler--Poincar\' e pairings for affine Hecke algebras}\label{s:ell-Hecke}

We begin by recalling several known facts about elliptic theory for affine Weyl groups and affine Hecke algebras. The main reference is \cite{OS} (see also \cite{CO}). The notation in this section is self contained and independent of the previous sections. For applications, the root datum in this section will be specialized to the root datum of the Langlands dual group $G^\vee$, as well as to the root data for the affine Hecke algebras $\cH(\mathfrak s^\vee)$ that occur on the dual side of the local Langlands correspondence.

Let $\cR=(X,R,X^\vee, R^\vee,\Pi)$ be a based root datum. Here $X,X^\vee$ are lattices in perfect duality
$\langle~,~\rangle\colon X\times X^\vee\to\bZ$, $R\subset X\setminus\{0\}$
and $R^\vee\subset X^\vee\setminus\{0\}$ are the finite sets of
roots and coroots respectively, and $\Pi\subset R$ is a basis of
simple roots. Let $W$ be the finite Weyl group with set of generators
$S=\{s_\al:\al\in \Pi\}.$ Set $\wti W=W\ltimes
X$, the (dual) extended affine Weyl group, and $W^a=W\ltimes Q$, the (dual) affine
Weyl group, where $Q$ is the root lattice of $R$. Then $W^a$ is
normal in $\wti W$ and $\Omega:=\wti W/W^a\cong X/Q$ is an abelian group. We assume that $\cR$ is semisimple, which means that $\Omega$ is a finite group. 

The set $R^a=R^\vee\times \bZ\subset X^\vee\times\bZ$ is the set of
affine roots. (Notice that $\wti W$ is the extended affine Weyl group of a split $p$-adic group $G$ with root datum dual to $\cR$, and $R^a$ is the set of affine roots for $G$.) A basis of simple affine roots is given by
$\Pi^a=(\Pi^\vee\times\{0\})\cup\{(\gamma^\vee,1): \gamma^\vee\in R^\vee
\text{ minimal}\}.$ For every affine root $\mathbf a=(\al^\vee,n)$, let
$s_{\mathbf a}\colon X\to X$ denote the reflection
$s_{\mathbf a}(x)=x-((x,\al^\vee)+n)\al.$ The affine Weyl group $W^a$
has a set of generators $S^a=\{s_{\mathbf a} \mid \mathbf a\in \Pi^a\}$. Given $J \subset S^a$, let $W_J$ be the subgroup of $W^a$ generated by $\{s_{\mathbf a} \mid \mathbf a \in J\}$. 
Let $l\colon \wti W\to\bZ$ be the length function.

Set $E=X\otimes_\bZ \CC$, so the discussion regarding the elliptic theory of $W$ and $E$ from the previous sections applies. We denote a typical element of $\wti W$ by $w t_x$, where $w\in W$ and $x\in X.$ The extended affine Weyl group $\wti W$ acts on $E$ via $(w t_x)\cdot v=w\cdot v+ x,$ $v\in E.$

An element $w t_x\in \wti W$ is called elliptic if $w\in W$ is elliptic (with respect to the action on $E$). For basic facts about elliptic theory for $\wti W$, see \cite[sections 3.1, 3.2]{OS}. There are finitely many elliptic conjugacy classes in $\wti W$ (and in $W^a$). The following fact is well known (see for example \cite[Lemma 5.4]{CO}).

\begin{lem}\label{l:elliptic-affine}
Suppose $C$ is an elliptic conjugacy class in $W^a$. Then there exists one and only one maximal $J\subsetneq S^a$ such that $C\cap W_J\neq\emptyset,$ and in this case $C\cap W_J$ forms a single elliptic $W_J$-conjugacy class.
\end{lem}

Let $R(\wti W)$ be  the Grothendieck group of $\wti W$-mod (the category of finite-dimensional modules). Define the Euler--Poincar\'e pairing of $\wti W$ by
\begin{equation}
\langle V_1,V_2\rangle_\EP^{\wti W}=\sum_{i\ge 0}(-1)^i\dim \Ext^i_{\wti W}(V_1,V_2),\quad V_1,V_2 \text{ finite-dimensional } \wti W\text{-modules}.
\end{equation}
Set $\overline R(\wti W)=R(\wti W)/\text{rad}\langle~,~\rangle_\EP^{\wti W}.$ 
By \cite[Theorem 3.3]{OS}, the Euler--Poincar\'e pairing for $\wti W$ can also be expressed as an elliptic integral. More precisely, define the conjugation-invariant elliptic measure $\mu_\ellip$ on $\wti W$ by setting $\mu_\ellip=0$ on nonelliptic conjugacy classes, and for an elliptic conjugacy class $C$ such that $v\in E$ is an isolated fixed point for some element of $C$, set
\begin{equation*}
\mu_\ellip(C)=\frac{|Z_{\wti W}(v)\cap C|}{|Z_{\wti W}(v)|};
\end{equation*}
here $Z_{\wti W}(v)$ is the isotropy group of $v$ in $\wti W.$ Then
\begin{equation}\label{affine-ell-pair}
\langle V_1, V_2\rangle_\EP^{\wti W}= (\chi_{V_1},\chi_{V_2})_\ellip^{\wti W}:=\int_{\wti W}\chi_{V_1}\overline{\chi_{V_2}}~d\mu_\ellip,\ V_1,V_2\in \wti W\text{-mod},
\end{equation}
where $\chi_{V_1},\chi_{V_2}$ are the characters of $V_1$ and $V_2$.

Set $T=\Hom_\bZ(X,\CC^\times)$. Then $W$ acts on $T$. For every $s\in T,$ set
$W_s=\{w\in W: w\cdot s=s\}$.
One considers the elliptic theory of the finite group $W_s$ acting on the cotangent space $\mathfrak t_s^*$ of $T$ at $s$.
By Clifford theory, the induction map
\begin{equation*}
\Ind_s\colon  W_s\text{-mod}\to \wti W\text{-mod},\quad \Ind_s(V_1):=\Ind_{W_s\ltimes X}^{\wti W}(V_1\otimes s)
\end{equation*}
maps irreducible modules to irreducible modules.
By \cite[Theorem 3.2]{OS}, the map 
\begin{equation}\label{ind-iso}
\bigoplus_{s\in W\backslash T}\Ind_s\colon  \bigoplus_{s\in W\backslash T}\overline R(W_s)_\CC\to \overline R(\wti W)_\CC
\end{equation}
is an isomorphism of metric spaces, in particular,
\begin{equation}\label{ind-pair}
\langle \Ind_s V_1,\Ind_s V_2\rangle_\EP^{\wti W}=(V_1,V_2)_\ellip^{W_s},\ V_1,V_2\in W_s\text{-mod}.
\end{equation}
A space $\overline R(W_s)_\CC$ in the left-hand side of (\ref{ind-iso}) is nonzero if and only if $s$ is an isolated element of $T$, more precisely $s\in T_\iso$, where
\begin{equation*}
T_\iso=\{s\in T^\vee: w\cdot s=s\text{ for some elliptic }w\in W\}.
\end{equation*}

\begin{ex}\label{e:W-PGL} Let $\cR$ be the root datum of $\PGL_n(\CC)$. We may take $T$ to be the maximal diagonal torus of $\PGL_n(\CC)$, with $X=X^*(T)$ the group of characters and $X^\vee=X_*(T)$ the group of cocharacters. In this case, $Q=X$ and \[\wti W=W^a=\langle s_i,~0\le i\le n-1\mid  (s_is_j)^{m(i,j)}=1,\ 0\le i,j\le n-1\rangle,\] where $m(i,i)=1$, $m(i,j)=2$ if $1<|i-j|<n-1$, and $m(i,j)=3$ if $|i-j|=1$ or $|i-j|=n-1$, when $n\ge 3$. This is the extended affine Weyl group associated to the $p$-adic group $\SL_n$. If $n=2$, then $\wti W=W^a=\langle s_0,s_1\mid s_0^2=s_1^2=1\rangle$. 

\smallskip

With this notation, the finite Weyl group is $W=\langle s_1,\dots,s_{n-1}\rangle\subset W^a$.
For every $0\le i\le n-1$, let $W_i=\langle s_0,s_1,\dots, s_{i-1},s_{i+1},\dots, s_{n-1}\rangle\subset W^a$. These are the maximal (finite) parabolic subgroups of $W^a$. In particular, $W_0=W$ and $W_i\cong S_n$ for all $i$. The space $E=\mathfrak t^*\cong \CC^{n-1}$ and each $W_i$ acts on $E$ by the reflection representation. Therefore, there exists a unique elliptic $W_i$-conjugacy class represented by the Coxeter element $w_i=s_0s_1\dotsb s_{i-1} s_{i+1}\dotsb s_{n-1}$. Thus by Lemma \ref{l:elliptic-affine}, there are exactly $n$ elliptic conjugacy classes in $W^a$ each determined by the condition that it meets $W_i$ in the conjugacy class of $w_i$, $0\le i\le n-1$. In particular, $\dim \overline R(\wti W)_\CC=n$ in this case.

On the other hand, by (\ref{ind-iso}), we need to consider $W$-orbits in $T_\iso$. Since there is only one elliptic conjugacy class in $W$, every $W$-orbit in $T_\iso$ is represented by an element of 
\[T^{s_1s_2\dotsb s_{n-1}}=\{\Delta_n(z)\mid z\in\mu_n\},\quad \Delta_n(z)=\text{diag}(1,z,z^2,\dots,z^{n-1}) \in T,\ \mu_n=\{z\mid z^n=1\},
\]
as noticed before. Two elements  $\Delta_n(z)$ and $\Delta_n(z')$ of $T^{s_1s_2\dotsb s_{n-1}}$ are $W$-conjugate if and only if $z$ and $z'$ have the same order. Fix a primitive $n$th root $\zeta$ of $1$. This means that  (\ref{ind-iso}) becomes in this case:
\[\bigoplus_{d | n} \overline R(W_{\Delta_n(\zeta^{n/d})})_\CC\cong \overline R(\wti W)_\CC.
\]
If $z=\zeta^m$, where $n=dm$, then $\Delta_n(z)$ is $W$-conjugate to $\text{diag}(\underbrace{1,\dots,1}_m,\underbrace{z,\dots,z}_m,\dots,\underbrace{z^{d-1},\dots,z^{d-1}}_m)$. Hence, one calculates
\begin{equation}
W_{\Delta_n(z)}\cong S_m^d\rtimes C_d\text{ and } \mathfrak t^*_{\Delta_n(z)}=\{(a_1,\dots,a_n)\mid \sum_{i=1}^n a_i=0\}. 
\end{equation}
The action is the natural permutation action. There are $\varphi(d)$ elliptic conjugacy classes,  represented by $(w_m,1,\dots,1)\cdot x_d$, where $w_m$ is a fixed $m$-cycle (Coxeter element) of $S_m$ and $x_d$ is one of the $\varphi(d)$ generators of $C_d$ (see Lemma \ref{l:Ws-SL}).  Again $\sum_{d|n}\varphi(d)=n$.
\end{ex}

Let $\vec{\mathbf q}=\{\mathbf q(s)\mid s\in S^a\}$ be a set of invertible, commuting
indeterminates such that $\mathbf q(s)=\mathbf q(s')$ whenever $s,s'$ are
$W^a$-conjugate. Let $\Lambda=\CC[\mathbf q(s),\mathbf q(s)^{-1}\mid s\in S^a]$.

The generic affine Hecke algebra $\cH(\cR,\vec{\mathbf q})$ associated to the root datum
  $\cR$ and the set of indeterminates $\vec{\mathbf q}$ is the unique associative,
  unital $\Lambda$-algebra with basis $\{T_w: w\in \wti W\}$ and
  relations
\begin{enumerate}
\item[(i)] $T_w T_{w'}=T_{ww'},$ for all $w,w'\in W$ such that
  $l(ww')=l(w)+l(w')$;
\item[(ii)] $(T_s-\mathbf q(s)^2)(T_s+1)=0$ for all $s\in S^a.$
\end{enumerate}

Fix a real number $q>1$. Given a $W^a$-invariant function $m\colon S^a\to \bR$, we may
define a homomorphism $\lambda_{m}\colon\Lambda\to \CC$, $\mathbf
q(s)= q^{m(s)}$.  Let $\CC_{\lambda_{m}}$ be the one-dimensional complex module on which $\Lambda$ acts by $\lambda_m$. Consider the specialized affine Hecke $\CC$-algebra
\begin{equation}
\cH(\cR, q, m)=\cH(\cR,\vec{\mathbf q})\otimes_\Lambda \CC_{\lambda_{m}}.
\end{equation}

\begin{ex}\label{e: H-PGL}
Let $\cR$ be the root datum of $\PGL_n(\CC)$. If $n=2$, the generic affine Hecke algebra has two indeterminates $\mathbf q(s_0)$ and $\mathbf q(s_1)$ and it is generated by $T_0=T_{s_0}$, $T_1=T_{s_1}$ subject only to the quadratic relations
\[(T_{i}-\mathbf q(s_i)^2)(T_{i}+1)=0, \ i=0,1.
\]
If $n\ge 3$, all the simple reflections are $W^a$-conjugate. There is only one indeterminate $\mathbf q$ such that the affine Hecke algebra is generated by $\{T_i=T_{s_i}, 0\le i\le n-1\}$ subject to the relations:
\begin{enumerate}
\item[(i)] $T_{i}T_{j}=T_{j}T_{i}$, $1<|i-j|<n-1$;
\item[(ii)] $T_{i}T_{i+1} T_{i}=T_{i+1}T_i T_{i+1}$, $0\le i\le n-1$;\quad $T_0 T_{n-1} T_0=T_{n-1}T_0 T_{n-1}$;
\item[(iii)] $(T_i -\mathbf q^2)(T_i+1)=0.$
\end{enumerate}
\end{ex}

Let $\cH=\cH(\cR, q, m)$ for simplicity of notation. If $V_1,V_2$ are two finite-dimensional $\cH$-modules, define the Euler--Poincar\'e pairing \cite[\S3.4]{OS}:
\begin{equation}
\EP_\cH(V_1,V_2)=\sum_{i\ge 0} (-1)^i\dim \Ext^i_\cH(V_1,V_2).
\end{equation}
This is a finite sum since $\cH$ has finite cohomological dimension \cite[Proposition 2.4]{OS}. The pairing $\EP_\cH$ is symmetric and positive semidefinite. It extends to a Hermitian positive-semidefinite pairing on the complexified Grothendieck group $R(\cH)_\CC$ of finite-dimensional $\cH$-modules. We wish to compare the Euler--Poincar\' e pairings for $\cH(\cR, q, m)$ and $\cH(\cR, q^\epsilon, m)$, where $\epsilon\in[0,1]$. Suppose we have a family of maps
\begin{equation}
\sigma_\epsilon\colon  \cH(\cR,q,m)\text{-mod}\to \cH(\cR,q^\epsilon,m)\text{-mod},\quad \sigma_\epsilon(\pi,V)=(\pi_\epsilon,V),
\end{equation}
such that
\begin{enumerate}
\item[(a)] for every $w\in\wti W$ and every $(\pi,V)$, the assignment $\epsilon\mapsto \pi(\epsilon)(T_w)$ is a continuous map $[0,1]\to \End(V)$.
\end{enumerate}
Then \cite[Theorem 3.5]{OS} shows that
\[
\EP_{\cH(\cR,q,m)}(V_1,V_2)=\EP_{\cH(\cR,q^\epsilon,m)}(\sigma_\epsilon(V_1),\sigma_\epsilon(V_2)),\quad \text{ for all }\epsilon\in[0,1].
\]
In particular, notice that $\cH(\cR,q^0,m)=\CC[\wti W]$, meaning that
\begin{equation}\label{e:deform}
\EP_{\cH(\cR,q,m)}(V_1,V_2)=\langle\sigma_0(V_1),\sigma_0(V_2)\rangle_\EP^{\wti W}.
\end{equation}
Using \cite[Theorem 1.7]{OS} or alternatively, for the affine Hecke algebras that occur for unipotent representations of $p$-adic groups, via the geometric constructions of \cite{KL,LuI,LuII}, we know that scaling maps $\sigma_\epsilon$ as above exist and in addition, they also behave well with respect to harmonic analysis:
\begin{enumerate}
\item[(b)] for every $\epsilon\in [0,1]$, $V$ is unitary (resp., tempered) if and only if $\sigma_\epsilon(V)$ is unitary (resp., tempered);
\item[(c)] for every $\epsilon\in (0,1]$, $V$ is discrete series if and only if $\sigma_\epsilon(V)$ is discrete series.
\end{enumerate}
Denoting by $\overline R(\cH)_\CC$ the quotient of $R(\cH)_\CC$ by the radical of $\EP_\cH$, it follows \cite[Proposition 3.9]{OS} that the scaling map $\sigma_0$ induces an injective isometric map 
\begin{equation}\label{affine-ell}
\sigma_0\colon  \overline R(\cH)_\CC\to \overline R(\wti W)_\CC\cong \bigoplus_{s\in T/W}\overline R(W_s)_\CC.
\end{equation}
 In fact this map is also an isomorphism, for example via \cite[Theorem 8.1]{CH}.

\subsection{Elliptic inner products for Weyl groups (after Waldspurger \cite{Wa})}\label{s:wald}

\medskip

Let $\cG=\cG^\circ$ be a complex connected reductive group and $\theta\colon\cG\to \cG$ a quasi-semisimple automorphism of $\cG$ of finite order. 

As in Section \ref{sec:GSC}, we let $\mathbf I$ be the set of pairs $(U, \cE)$ where $U$ is a unipotent conjugacy class in $\cG$ and $\cE$ is an irreducible $\cG$-equivariant local system on $U$.
The automorphism $\theta$ acts on $\mathbf I$ via $(U,\cE)\mapsto (\theta(U),(\theta^{-1})^*(\cE))$. Let $\mathbf I^\theta$ denote the fixed points of this action and suppose $(U,\cE)\in\mathbf I^\theta$. If we fix $u\in U$, there exists $x\in\cG$ such that $\Ad(x)\circ \theta(u)=u$, hence $\Ad(x)\circ \theta$ preserves $\rZ_\cG(u)$ and hence it defines an automorphism of $A_u$, denoted $\theta_u$. As explained in \cite[p. 612]{Wa}, if $\phi\in \widehat A_u$ corresponds to the local system $\cE$, the fact that $(\theta^{-1})^*(\cE)\cong \cE$ is equivalent to the condition that $\phi$ extends to a representation $\widetilde \phi$ of $A_u\rtimes\langle\theta_u\rangle$.

Fix a Borel subgroup $B_u$ of $\rZ_\cG(u)^\circ$ and a maximal torus $T_u$ in $B_u$. Let $\mathfrak t_u$ be the complex Lie algebra of $T_u$.
Define a complex representation $(\delta_u,{\mathfrak t}_u)$ of $A_u\rtimes \langle \theta_u\rangle$, extending the previous definition for the action of $A_u$. Similarly to Section \ref{s:complex}, since $\rZ_\cG(u) $ 
acts on $\rZ_\cG(u)^\circ$ by conjugation and $\theta_u$ acts on $\rZ_\cG(u)^\circ$ as above, every element $z\in \rZ_\cG(u)\rtimes \langle \theta_u\rangle$ acts on $\rZ_\cG(u)^\circ$ via an automorphism $\alpha_z$. There exists $y\in \rZ_\cG(u)^\circ$ such that $\alpha_z\circ \Ad(y)$ preserves $B_u$ and $T_u$. This means that $\alpha_z\circ \Ad(y)$ defines an automorphism of the cocharacter lattice $X_*(T_u)$ that also preserves the sublattice $X_*(Z_\cG^\circ)$, and therefore a linear isomorphism $\delta_u(z)$ of $ {\mathfrak t}_u$. If $\bar z\in A_u\rtimes \langle \theta_u\rangle$, set $\delta_u(\bar z):=\delta_u(z)$, where $z$ is a lift of $\bar z$ in $\rZ_\cG(u)\rtimes \langle \theta_u\rangle$. This defines a representation of $A_u\rtimes \langle \theta_u\rangle$.

Suppose $(U,\cE)$ and $(U',\cE')$ are two elements of $\mathbf I^\theta$ represented by $(u,\phi)$ and $(u',\phi')$, respectively. Define
\begin{equation}
(\widetilde\phi,\widetilde\phi')_{\theta-\ellip}=\begin{cases} (\widetilde \phi, \widetilde\phi')^{\delta_u}_{\theta-\ellip},&\text{ if }U=U',\\0,&\text{ if }U\neq U'.
\end{cases}
\end{equation}
This is the $\theta$-elliptic pairing on $\bigoplus_{U} R(A_u\rtimes\langle\theta_u\rangle)$. 

\smallskip

The relation between this elliptic pairing and the generalized Springer correspondence \cite{LuIC} is explained in \cite[\S 3]{Wa}. The automorphism $\theta$ acts naturally on all of the objects involved in the definition of the Springer correspondence. As discussed in \cite[\S 3]{Wa}, this leads to an action of $W_j\rtimes \langle \theta\rangle$ on $\mathfrak z_{\cM}$, the Lie algebra of $\rZ_{\cM}$ (notation as in Section \ref{sec:GSC}), and to a $\theta$-generalized Springer correspondence $\nu\colon  \mathbf I^\theta\to \mathbf {\widetilde J}^\theta$. For every $(j,\rho)\in \mathbf {\widetilde J}^\theta$, let $\widetilde\rho$ denote the extension of $\rho$ to a representation of $W_j\rtimes\langle \theta\rangle$ as in {\it loc. cit.}.

Let $i=(U,\cE)$, $i'=(U',\cE')$ be two elements of $\mathbf I^\theta$, and $\nu(i)=(j,\rho)$, $\nu(i')=(j',\rho')$. For every $m\in\ZZ$, the constructible sheaf $\cH^{2m+a_{U'}}(A_{j',\rho'})|_U$ decomposes as a direct sum of $G$-equivariant local systems on $U$. As in \cite{Wa}, setting 
\[H^m_{i,i'}=\Hom(\cE,\cH^{2m+a_{U'}}(A_{j',\rho'})|_U),
\]
the automorphism $\theta$ defines a linear map $\theta^m_{i,i'}: H^m_{i,i'}\to H^m_{i,i'}$. In particular, 
\[H^m_{i,i}=0\text{ if }m\neq 0\text{ and }\dim H^\circ_{i,i}=1.
\]
We may arrange the construction so that $\theta^\circ_{i,i}$ is the identity map.
Moreover, it is clear that $H^m_{i,i'}\neq 0$ for some $m$ only if $U\subset \overline{U'}$. (Recall that the restriction of $A_{j',\rho'}$ to the set of unipotent elements of $\cG$ is supported on $\overline {U'}$.) 

Define the virtual representation of $W_j\rtimes\langle \theta\rangle$
\begin{equation}
\pmb{\widetilde\rho}=\sum_{\rho'\in \widehat W_j^\theta} P_{j,\rho,\rho'} \widetilde \rho',\text{ where } P_{j,\rho,\rho'}=\sum_{m\in\ZZ} \tr (\theta^m_{i,i'}).
\end{equation}
In this virtual combination, $P_{j,\rho,\rho}=1$, and $P_{j,\rho,\rho'}\neq 0$ implies that $U\subset \overline{U'}$ if $(U,\cE)=\nu^{-1}(j,\rho)$ and $(U',\cE')=\nu^{-1}(j,\rho')$. 

\begin{ex}
When $\theta$ is the trivial automorphism of $\cG$ and $j=j_0$ (the case of the classical Springer correspondence), $\pmb{\widetilde\rho}$ can be identified with the reducible $W$-representation on the $\phi$-isotypic component ($\phi\in \widehat A_u$ corresponding to $\cE$) of the total cohomology of the Springer fiber of $u$. 
\end{ex}

Consider the $\theta$-elliptic pairing $(~,~)_{\theta-\ellip}^{W_j}$ on $\bigoplus_{j\in \mathbf J^\theta} R(W_j\rtimes \langle\theta\rangle)$, defined on each summand via the action of $ W_j\rtimes \langle\theta\rangle$ on ${\mathfrak z}_{\cM}$ and extended orthogonally to the direct sum. 

\begin{thm}[({\cite[Th\'eor\`eme p.~616]{Wa}})]\label{t:Wal} Let $i=(U,\cE)$, $i'=(U',\cE')$ be two elements of $\mathbf I^\theta$, and $\nu(i)=(j,\rho)$, $\nu(i')=(j',\rho')$. Let $(u,\phi)$, $(u',\phi')$, $\phi\in \widehat A_u$, $\phi'\in \widehat A_{u'}$, be representatives for $i,i'$, respectively. Then
\[(\widetilde\phi,\widetilde\phi')_{\theta-\ellip}=(\pmb{\widetilde\rho},\pmb{\widetilde\rho}')_{\theta-\ellip}^{W_j}.
\]
\end{thm}
The equality in the theorem does not depend on the choices involved in the construction.

\subsection{The proof of Theorem \ref{t:main-elliptic}: the case of adjoint groups} In this subsection, suppose that $G$ is a simple $F$-split group of adjoint type. This means that $G^\vee$ is simply connected, hence, for every $s\in T^\vee$, $\rZ_{G^\vee}(s)$ is connected. We may apply Theorem \ref{t:Wal} to
\[\cG=\rZ_{G^\vee}(s)\text{ and }\theta\text{ the trivial automorphism}.
\]
Let $\mathbf I^s=\mathbf I^{\rZ_{G^\vee}(s)}$, ${\bJ}^s={\bJ}^{\rZ_{G^\vee}(s)}$, and $\widetilde{\bJ}^s=\widetilde{\bJ}^{\rZ_{G^\vee}(s)}$, so that the generalized Springer correspondence for $\rZ_{G^\vee}(s)$ is the map
\[\nu_s\colon \mathbf I^s\to \widetilde{\bJ}^s,\ (U,\cE)\mapsto (j,\rho),
\]
and 
\[\pmb{\widetilde\rho}=\sum_{\rho'\in \widehat W_j} P_{j,\rho,\rho'}  \rho',\text{ where } P_{j,\rho,\rho'}=\sum_{m\in\ZZ}\dim \Hom(\cE,\cH^{2m+a_{U'}}(A_{j',\rho'})|_U).
\]
For convenience, let us also define
\begin{equation}\label{e:nu-tilde}
\widetilde\nu_s\colon \mathbf I^s\to \widetilde{\bJ}^s,\ (U,\cE)\mapsto (j,\pmb{\widetilde\rho}).
\end{equation}

Recall that for every semisimple element $s\in G^\vee$, $\cG_s^p=\rZ_{G^\vee}(s)$. 

\begin{prop}
Suppose $G$ is simple $F$-split group of adjoint type. The maps $\widetilde\nu_s$ from (\ref{e:nu-tilde}) induce an isometric isomorphism
\[\bigoplus_{s\in \cC(G^\vee)_{\mathsf{ss}}}\bigoplus_{u\in  \cC(\cG_s^p)_{\mathsf{un}}}\overline R(A_{\cG_s^p}(u))\cong \bigoplus_{s\in \cC(G^\vee)_{\mathsf{ss}}} \bigoplus_{j\in  {\bJ}^s} \overline R(W_j),\quad (\phi,\phi')_\ellip^{\delta_u^s}=(\widetilde\nu_s(\phi),\widetilde\nu_s(\phi'))_\ellip^{W_j}.
\]
\end{prop}

\begin{proof}
This is immediate from Theorem \ref{t:Wal} applied to each $\cG_s^p.$
\end{proof}

\subsection{Extending Theorem \ref{t:main-elliptic}}\label{s:extend} In order to extend the results to the case when $G$ is simple $F$-split but not adjoint, we first need some results about Mackey induction. We follow a construction from \cite[\S4.2]{CH2016}.
Suppose $H'$ is a finite group, $H$ a normal subgroup of $H'$, and $H'/H=\fR$ is abelian.  The groups $H',\fR$ act on $\widehat H$. For every $H$-character $\chi$, and $\gamma\in \fR$, denote by ${}^\gamma\chi$ the $H$-character ${}^\gamma\chi(h)=\chi(\gamma^{-1} h \gamma)$ (it doesn't depend on the choice of coset representative $\gamma$).

If $\sigma\in \widehat H$, let $\fR_\sigma$ and $H'_\sigma$ denote the corresponding isotropy groups of $\sigma$. For each $\gamma\in \fR_\sigma$, fix an isomorphism $\phi_\gamma\colon  {}^\gamma\sigma\to\sigma$ and define the twisted trace as $\tr_\gamma(\sigma)(h)=\tr(\sigma(h)\circ \phi_\gamma)$, $h\in H$. The choices of $\phi_\gamma$ (each unique up to scalar) define a factor set, or a 2-cocyle, $\beta_\sigma:\fR_\sigma\times \fR_\sigma\to \CC^\times$. 

\begin{rem}\label{r:triv-cocycle}
We assume that the action of $\fR$ can be normalized so that $\beta_\sigma$ is trivial. This is the case for example when $\fR$ is cyclic. 
\end{rem}

If $\tau$ is a (virtual) $\fR_\sigma$-representation, we may form the Mackey induced (virtual) $H'$-representation
\[\sigma\rtimes \tau=\Ind_{H'_\sigma}^{H'}(\sigma\otimes \tau).
\]
If $\tau$ is an irreducible $\fR_\sigma$-representation, then $\sigma\rtimes\tau$ is an irreducible $H'$-representation. In fact, $\widehat H'=\{\sigma\rtimes\tau\mid \sigma\in \fR\backslash \widehat H,\ \tau\in \widehat \fR_\sigma\}.$

Given $\gamma\in \fR$, if $\gamma\in \fR_\sigma$, define $\tau_{\sigma,\gamma}$ to be the virtual $\fR_\sigma$-representation whose character is the delta function on $\gamma$. Then $\{\sigma\rtimes\tau_{\sigma,\gamma}\mid \sigma\in \fR\backslash \widehat H,\ \gamma\in \fR\}$ is a basis of $R(H')$. As in \cite[Lemma 4.2.2]{CH2016}
\begin{equation}\label{e:twisted-trace}
\chi_{\sigma\rtimes \tau_{\sigma,\gamma}}(h)=\begin{cases} 0,&\text{ if } h\notin H\gamma,\\ 
\sum_{\gamma'\in \fR/\fR_\sigma} {}^{\gamma'}\! (\tr_\gamma(\sigma))(h\gamma^{-1}), &\text{ if } h\in H\gamma.
\end{cases}
\end{equation}
Notice that 
\[H'/H'_\sigma\cong \fR/\fR_\sigma
\]
indexes the $\fR$-orbit (equivalently, the $H'$-orbit) of $\sigma\in \widehat H$.
Suppose $H'$ is endowed with a representation $\delta$ and we define the corresponding elliptic pairing $(~,~)^{H'}_\ellip = (~, ~)^\delta_\ellip$. 

From now on, assume that $H'=H\rtimes \fR$, so that each $\gamma\in\fR$ acts on $H$ by automorphisms of $H$. (If $H$ is abelian, which is often the case for component groups, this assumption is not necessary.) Then we may define the twisted elliptic pairing $(~,~)^H_{\gamma-\ellip}$ (for each $\gamma\in \fR$).

The existence of the intertwiner $\phi_\gamma$ ($\gamma\in \fR_\sigma$) is equivalent to the existence of an extension of  $\sigma$ to a representation of $H\rtimes \langle\gamma\rangle$, by setting $\sigma(h\gamma)=\sigma(h)\circ \phi_\gamma$, so $\tr_\gamma(\sigma)(h)=\tr\sigma(h\gamma)$, $h\in H$. This is implicit in the following lemma.

\begin{lem}\label{l:ellip-compare}
For every $\gamma_1,\gamma_2\in \fR$ and every $\sigma_1,\sigma_2\in \widehat H$, the $H'$-elliptic pairing is given by
\[(\sigma_1\rtimes\tau_{\sigma_1,\gamma_1},\sigma_2\rtimes\tau_{\sigma_2,\gamma_2})^{H'}_\ellip=\begin{cases} \frac 1{|\fR|}\sum_{\gamma'\in \fR/\fR_{\sigma_1},\gamma''\in \fR/\fR_{\sigma_2}} ({}^{\gamma'}\!\sigma_1,{}^{\gamma''}\!\sigma_2)^H_{\gamma_1-\ellip},&\text{ if }\gamma_1=\gamma_2,\\
0,&\text{ if }\gamma_1\neq \gamma_2.
\end{cases}
\]
\end{lem}

\begin{proof}
The orthogonality of the two characters when $\gamma_1\neq \gamma_2$ follows at once since the first is supported on $\gamma_1H$ and the second on $\gamma_2H$. The first formula follows from (\ref{e:twisted-trace}) by the definition of the elliptic pairing.
\end{proof}

Lemma \ref{l:ellip-compare} allows us to extend the proof of Theorem \ref{t:Wal} to the case when $\cG'$ is disconnected as long as:

\smallskip

\noindent $(\star)$: the cocycles $\sharp_j$ that occur in the disconnected Springer correspondence (\ref{eqn:gen-Springer_disc})  can be trivialized.

With the notation from Section \ref{sec:GSC}, set \[A_u(j)=\rZ_{\fR_j\cG^\circ}(u)/\rZ_{\cG}(u)^\circ.\] This is a normal subgroup of $A_u$ containing $A_{\cG^\circ}(u)$.

In our case $\cG$ is not an arbitrary disconnected reductive group, but rather $\cG=Z_{G^\vee}(s)$, for some semisimple element $s\in T^\vee$. Let $W^\vee$ be the Weyl group of $T^\vee$ in $G^\vee$. If we fix $B(s)$ a Borel subgroup of $\cG^\circ$ with $T^\vee \subset B(s)$, and denote by $\Phi^+(\cG^\circ)$ the positive roots of $T^\vee$ in $\cG^\circ$, then $\cA=\cG/\cG^\circ$ can be identified with
\begin{equation}\label{e:action-A}
    \cA=\{w\in W^\vee\mid w(\Phi^+(\cG^\circ))=\Phi^+(\cG^\circ)\},
\end{equation}
see for example \cite[\S 1]{Bon}.
With this identification, $\cA$ acts on $T^\vee$ and $\Phi^+(\cG^\circ)$, hence by automorphisms of the root datum of $\cG^\circ$. We will use this ``global" action on $\cG^\circ$ in the proof of the next result in order to construct the extensions to the appropriate semidirect products and  apply Waldspurger's result to $\theta$-elliptic pairings.

\begin{prop}\label{p:Wal-gen} Retain the notation from Section \ref{sec:GSC} and suppose that $(\star)$ holds. Let $\nu: \mathbf I^{\cG}\to \widetilde{\mathbf J}^{\cG}$ be the generalized Springer correspondence (\ref{eqn:gen-Springer_disc}). Let $i=(U,\mathcal E)$, $i'=(U',\mathcal E')$ be two elements of $\mathbf I^{\cG}$, and $\nu(i)=(j,\rho)$, $\nu(i')=(j,\rho')$. Let $(u,\phi)$, $(u',\phi')$, $\phi\in \widehat A_u$, $\phi'\in \widehat A_{u'}$ be representatives for $i,i'$, respectively. Then:
\begin{enumerate}
\item If $U\neq U'$, $(\phi,\phi')_{\ellip}^{A_u}=0=(\pmb\rho,\pmb\rho')^{W_j}_{\ellip}$.
\item If $U=U'$, $(\phi,\phi')^{A_u(j)}_{\ellip}=(\pmb\rho,\pmb\rho')^{W_j}_{\ellip}$.
\end{enumerate}
\end{prop}

\begin{proof} Suppose $j=j'$, otherwise the claim is true by definition. Let $\rho_1^\circ,\rho_2^\circ\in \widehat{W_j^\circ}$ and suppose that they have unipotent supports $U_1^\circ,U_2^\circ$, respectively, in the connected generalized Springer correspondence such that $\cG\cdot U_1^\circ\neq \cG\cdot U_2^\circ$. Let $\pmb{\rho_i^\circ}$, $i=1,2$ be the corresponding reducible Springer representations, as in Section \ref{s:wald}. We assume that they all have appropriate twisted extensions as in Section \ref{s:wald} and drop $\widetilde{\ }$ from the notation. Using Lemma \ref{l:ellip-compare} applied to $H=W_j^\circ$, $H'=W_j$, $\fR=\mathfrak R_j$, for every $\gamma\in (\mathfrak R_j)_{\rho_1^\circ}\cap (\mathfrak R_j)_{\rho_2^\circ}$,
 \[(\pmb{\rho_1^\circ}\rtimes\tau_{\rho_1^\circ,\gamma},\pmb{\rho_2^\circ}\rtimes\tau_{\rho_1^\circ,\gamma})^{W_j}_\ellip=\frac 1{|\mathfrak R_j|}\sum_{\gamma',\gamma''} ({}^{\gamma'}\!\pmb{\rho_1^\circ},{}^{\gamma''}\!\pmb{\rho_2^\circ})^{W_j^\circ}_{\gamma-\ellip}=0,
\]
by Theorem \ref{t:Wal}. We used implicitly here that the stabilizers in $\fR_j$ of $\rho_i^\circ$ and $\pmb{\rho_i^\circ}$ are the same. In conjunction with the second claim in Lemma \ref{l:ellip-compare}, this implies that $(\pmb{\rho_1^\circ}\rtimes\tau_1,\pmb{\rho_2^\circ}\rtimes\tau_2)^{W_j}_\ellip=0$ for all $\tau_i\in \widehat {(\mathfrak R_j)}_{\rho_i^\circ}$, $i=1,2$.
Hence $(\pmb{\rho_1},\pmb{\rho_2})^{W_j}_\ellip=0$ whenever $\rho_1,\rho_2$ have distinct unipotent (disconnected) Springer support, which proves the first part of the claim.

Now assume that $u=u'$ and $\phi,\phi'\in \widehat A_u$. Suppose $\rho^\circ_i$ occurs in the restriction of $\rho_i$ to $W_j^\circ$ and that $\phi_i^\circ\in \widehat A_{\cG^\circ}(u)$ (which necessarily occurs in the restriction of $\phi_i$) corresponds to $\rho^\circ_i$ in the connected generalized Springer correspondence.
We observe that there is a natural injection
\begin{equation} \label{eqn:injection}\fR_j=W_j/W_j^\circ\simeq \Nor_{\cG}(j)/\Nor_{\cG^\circ}(\cM^\circ)\hookrightarrow \cG/\cG^\circ=\cA.
\end{equation}
Hence every $\gamma\in \fR_j$ can be regarded as an automorphism of $\cG^\circ$ via (\ref{e:action-A}), and in particular, Theorem \ref{t:Wal} can be applied with $\gamma$ in place of $\theta$. We wish to compare
$(\pmb{\rho_1^\circ}\rtimes\tau_{\rho_1^\circ,\gamma_1}, \pmb{\rho_2^\circ}\rtimes\tau_{\rho_2^\circ,\gamma_2})^{W_j}_\ellip
$
and $(\phi_1^\circ\rtimes\tau_{\phi_1^\circ,\gamma_1}, \phi_2^\circ\rtimes\tau_{\phi_2^\circ,\gamma_2})^{A_u}_\ellip$. By \cite[Lemma 4.4]{AMS1}, 
\[(\fR_j)_{\rho_i^\circ}\cong (W_j)_{\rho_i^\circ}/W_j\cong (A_u)_{\phi_i^\circ}/A_{\cG^\circ}(u),\quad i=1,2,
\]
which implies that there is an identification between $\gamma_i$, $i=1,2$ for the $\rho^\circ_i$'s and for the $\phi^\circ_i$'s in the setting of Lemma \ref{l:ellip-compare}. Hence if $\gamma_1\neq\gamma_2$, both elliptic products are zero.

Suppose $\gamma_1=\gamma_2=\gamma\in (\fR_j)_{\rho_1^\circ}\cap (\fR_j)_{\rho_2^\circ}= ((A_u)_{\phi_1^\circ}\cap  (A_u)_{\phi_2^\circ})/A_{\cG^\circ}(u)$. To simplify the formulas, set $n_i=|(\fR_j)_{\rho_i^\circ}|=|(A_u)_{\phi_i^\circ}/A_{\cG^\circ}(u)|$, $i=1,2$. Then, firstly by Lemma \ref{l:ellip-compare} and secondly by Theorem \ref{t:Wal},
\begin{align*}
n_1 n_2|\fR_j|(\pmb{\rho_1^\circ}\rtimes\tau_{\rho_1^\circ,\gamma}, &~\pmb{\rho_2^\circ}\rtimes\tau_{\rho_2^\circ,\gamma})^{W_j}_\ellip=\sum_{\gamma',\gamma''\in \fR_j} ({}^{\gamma'}\!\pmb{\rho_1^\circ},{}^{\gamma''}\!\pmb{\rho_2^\circ})^{W_j^\circ}_{\gamma-\ellip}\\
&=\sum_{u_0\in \fR_j \cG^\circ\cdot u/\cG^\circ}\sum_{\substack{\gamma',\gamma''\in \fR_j\\\cG^\circ\cdot ({}^{\gamma'}\!u_0)=
\cG^\circ\cdot ({}^{\gamma''}\!u_0)=\cG^\circ\cdot u_0}} ({}^{\gamma'}\!\phi_1^\circ,{}^{\gamma''}\!\phi_2^\circ)^{A_{\cG^\circ}(u_0)}_{\gamma-\ellip}.
\end{align*}
The first sum is over the representatives of the $\cG^\circ$-orbits that are conjugate to $\cG^\circ\cdot u$ via $\fR_j$. Since the corresponding summand for two different $\cG^\circ$-orbits $u_0,u_0'\in \cG\cdot u$ are equal (as they are related by an outer automorphism of $\cG^\circ$), it follows that 
\begin{equation}\label{e:Rj}
n_1n_2\frac{|\fR_j|}{N_u}(\pmb{\rho_1^\circ}\rtimes\tau_{\rho_1^\circ,\gamma},\pmb{ \rho_2^\circ}\rtimes\tau_{\rho_2^\circ,\gamma})^{W_j}_\ellip=\sum_{\substack{\gamma',\gamma''\in\fR_j\\\cG^\circ\cdot ({}^{\gamma'}\!u)=
\cG^\circ\cdot ({}^{\gamma''}\!u)=\cG^\circ\cdot u}} ({}^{\gamma'}\!\phi_1^\circ,{}^{\gamma''}\!\phi_2^\circ)^{A_{\cG^\circ}(u)}_{\gamma-\ellip},
\end{equation}
where $N_u$ is the number of $\cG^\circ$-conjugacy classes in $\fR_j\cG^\circ\cdot u$. It is easy to see (using orbit-stabilizer counting) that $\frac{|\fR_j|}{N_u}=|\rZ_{\fR_j\cG^\circ}(u)/\rZ_{\cG^\circ}(u)|$. Moreover, an element $\gamma\in \fR_j$ has the property that $\cG^\circ\cdot ({}^{\gamma}\!u)=\cG^\circ\cdot u$ if and only if $\gamma\in \rZ_\cG(u)$ mod $\cG^\circ$. Hence (\ref{e:Rj}) becomes:
\begin{equation}\label{e:Rj-2}
n_1n_2{|\rZ_{\fR_j\cG^\circ}(u)/\rZ_{\cG^\circ}(u)|}(\pmb{\rho_1^\circ}\rtimes\tau_{\rho_1^\circ,\gamma},\pmb{ \rho_2^\circ}\rtimes\tau_{\rho_2^\circ,\gamma})^{W_j}_\ellip=\sum_{\gamma',\gamma''\in\rZ_{\fR_j\cG^\circ}(u)/\rZ_{\cG^\circ}(u)} ({}^{\gamma'}\!\phi_1^\circ,{}^{\gamma''}\!\phi_2^\circ)^{A_{\cG^\circ}(u)}_{\gamma-\ellip}.
\end{equation}

On the other hand, applying Lemma \ref{l:ellip-compare} to $A_u(j)$, we get
\begin{equation}\label{e:A_u(j)}
n_1n_2|A_u(j)/A_{\cG^\circ}(u)|(\phi_1^\circ\rtimes\tau_{\phi_1^\circ,\gamma}, \phi_2^\circ\rtimes\tau_{\phi_2^\circ,\gamma})^{A_u(j)}_\ellip=\sum_{r',r''\in A_u(j)/A_{\cG^\circ}(u)}({}^{r'}\!\phi_1^\circ,{}^{r''}\!\phi_2^\circ)^{A_{\cG^\circ}(u)}_{\gamma-\ellip}.
\end{equation}
Notice that
\[ A_u(j)/A_{\cG^\circ}(u)\cong \rZ_{\fR_j\cG^\circ}(u)/\rZ_{\cG^\circ}(u)\hookrightarrow \cG/\cG^\circ.
\]
The claim follows by comparing (\ref{e:Rj-2}) and (\ref{e:A_u(j)}).

\end{proof}

\begin{rem}\label{r:assumption} \begin{enumerate}
\item In our case, $\cG=\rZ_{G^\vee}(s)$ for a semisimple element $s$, and hence all of the groups $W_j^\circ$ that occur in our setting are finite Weyl groups. This means that $(\star)$ holds by \cite[Proposition 4.3]{ABPSdisc}.
 \item If $\rZ_{\fR_j\cG^\circ}(u)/\rZ_{\cG^\circ}(u)=\rZ_{\cG}(u)/\rZ_{\cG^\circ}(u)$, it follows from Proposition \ref{p:Wal-gen}(ii) that, in fact, $(\phi,\phi')^{A_u}_{\ellip}=(\pmb\rho,\pmb\rho')^{W_j}_{\ellip}$. For example, when $j=j_0$ is the cuspidal datum associated to the trivial local system on the maximal torus of $\cG^\circ$, the $\fR_{j_0}=\cG/\cG^\circ$, hence this condition holds automatically. 
\end{enumerate}
\end{rem}

\begin{cor}\label{c:Iwahori}
Let $G$ is a semisimple split $F$-group. Then Conjecture \ref{conj:main-ell} holds for all Iwahori-spherical representations of the pure inner twists of $G$, in the sense of (\ref{e:ell-pure-LLC}).
\end{cor}

\begin{proof}
By  Proposition \ref{p:Wal-gen} and Remark \ref{r:assumption}(ii), the claim follows just as for adjoint groups.
\end{proof}

\section{\texorpdfstring{$\mathsf{Sp}_4(F)$}{\mathsf{Sp}4(F)}}\label{s:sp4}
As a useful example, we present the case $G=\mathsf{Sp}_4(F)$. Firstly, there are $6$ unipotent representations of the finite group $\mathsf{Sp}_4(\mathbb F_q)$: $5$ in bijection with  irreducible representations of the finite Weyl group of type $C_2$ and one cuspidal representation $\theta$. Using Lusztig's notation for the irreducible representations of the Weyl group of type $B/C$, there are $3$ families $\cF$ of unipotent representations with associated finite groups $\Gamma$ as follows:

\begin{itemize}
\item $\Gamma=\{1\}$, $\cF=\{2\times \emptyset\}$;
\item $\Gamma=\{1\}$, $\cF=\{0\times 11\}$;
\item $\Gamma=\ZZ/2\bZ$, $\cF=\{1\times 1, 11\times \emptyset, \emptyset\times 2, \theta\}$ with associated parameters, in order, $M(\Gamma)=\{(1,\mathbf 1),(-1,\mathbf 1), (1,\epsilon), (-1,\epsilon)\}$.
\end{itemize}
For the $\bZ/2\bZ$-family, the stable combinations are:
\begin{equation}
\begin{aligned}
&\sigma(1,1)=1\times 1+\emptyset\times 2,&\sigma(-1,1)=11\times\emptyset+\theta,\\
&\sigma(1,-1)=1\times 1-\emptyset\times 2, &\sigma(-1,-1)=11\times\emptyset-\theta,
\end{aligned}
\end{equation}
and Lusztig's Fourier transform acts by the flip  $\sigma(x,y)\mapsto \sigma(y,x)$. For the singleton families, the Fourier transform is the identity.

\medskip

Next, we consider the $p$-adic group $\mathsf{Sp}_4(F)$: the unipotent representations are parameterized by data in the dual group $G^\vee=\mathsf{SO}_5(\CC)$. In particular, the list of unipotent classes $u$ and their attached groups $\Gamma_u$ is:

\begin{center}
\begin{tabular}{|c|c|}
\hline
$u$ &$\Gamma_u$\\
\hline
$(5)$ &$1$\\
\hline
$(311)$ &$\mathsf{S}(\mathsf{O}_1\times \mathsf{O}_2)\cong \mathsf{O}_2$\\
\hline
$(221)$ &$\mathsf{Sp}_2$\\
\hline
$(1^5)$ &$\mathsf{SO}_5$\\
\hline
\end{tabular}
\end{center}

The interesting case is $u=(311)$. Write
\[\Gamma_u=\langle z,\delta\mid z\in \CC^\times,\ \delta^2=1,\ \delta z \delta^{-1}=z^{-1}\rangle.
\]
Then $\rZ_{\Gamma_u}(\pm\delta)=A_{\Gamma_u}(\pm\delta)=\{\pm 1,\pm\delta\}\cong C_2\times C_2$ and $A_{\Gamma_u}(\pm 1)=\{1,\delta\}\cong C_2$. There are $6$ conjugacy classes of elliptic pairs:
\begin{equation}
[(\pm 1,\delta)],\quad [(\delta,\pm 1)],\quad [(\delta,\pm\delta)],
\end{equation}
and the flip acts as
\begin{equation}
\mathsf{flip}([(\pm 1,\delta)])=[(\delta,\pm 1)],\quad \mathsf{flip}([(\delta,\pm\delta)])=[(\delta,\pm \delta)].
\end{equation}
There are three conjugacy classes of isolated semisimple elements in $T^\vee=\{(a,b)\mid a,b\in\CC^\times\}$ in $\mathsf{SO}_5(\CC)$. In this notation, the Weyl group $W(B_2)$ acts on $T$ by flips and inverses. The representatives of the three classes are:
\begin{itemize}
\item $s_0=(1,1)$, $\rZ_{G^\vee}(s_0)=\mathsf{SO}_5$;
\item $s_1=(-1,1)$, $\rZ_{G^\vee}(s_1)=\mathsf{S} (\mathsf{O}_2\times \mathsf{O}_3)$;
\item $s_2=(-1,-1)$, $\rZ_{G^\vee}(s_2)=\mathsf{S} (\mathsf{O}_1\times \mathsf{O}_4)\cong\mathsf{O}_4$.
\end{itemize}
All three $s_0,s_1,s_2$ occur in $\Gamma_u=\mathsf{O}_2$ and in the notation above for $\mathsf{O}_2=\langle z,\delta\rangle$, they are
\[s_0\leftrightarrow 1\in \mathsf{O}_2,\quad s_1\leftrightarrow -1\in \mathsf{O}_2,\quad s_2\leftrightarrow \delta\in \mathsf{O}_2.
\] 
Consequently, there are $8$ elliptic tempered representations of the form $\pi(s_i,u,\phi)$, $i=0,1,2$, $u=(311)$: $6$ are Iwahori-spherical, and $2$ are supercuspidal. Out of these, $4$ are discrete series representations, all those for $s_2=\delta$. The parahoric restrictions are given in Table \ref{ta:sp4}. We computed them using the same method as in \cite[(6.2)]{Re2}, but since in our case $G^\vee$ is not simply connected, we also need to involve the Mackey induction for graded affine Hecke algebras attached to disconnected groups.

\begin{table}
\begin{center}
\begin{tabular}{|c|c|c|c|}
\hline
$\pi(u,s,\phi)$ &$K_0\to \mathsf{Sp}_4(\mathbb F_q)$ &$K_1\to \mathsf{Sp}_2(\mathbb F_q)^2$ &$K_2\to \mathsf{Sp}_4(\mathbb F_q)$ \\
\hline
\hline
$(s_0,\mathbf 1)$ &$1\times 1+\emptyset\times 11$ &$\mathbf 1\boxtimes \epsilon+\epsilon\boxtimes \mathbf 1+\epsilon\boxtimes\epsilon$ &$1\times 1+\emptyset\times 11$\\
\hline
$(s_0,\epsilon)$ &$\emptyset\times 2$ &$\epsilon\boxtimes\epsilon$ &$\emptyset\times 2$\\
\hline
$(s_1,\mathbf 1)$ &$\emptyset\times 2+\emptyset\times 11$ &$\mathbf 1\boxtimes \epsilon+\epsilon\boxtimes\epsilon$ &$1\times 1$\\
\hline
$(s_1,\epsilon)$ &$1\times 1$ &$\epsilon\boxtimes \mathbf 1+\epsilon\boxtimes\epsilon$ &$\emptyset\times 2+\emptyset\times 11$\\
\hline
$(s_2,\mathbf 1\boxtimes\mathbf 1)$ &$\emptyset\times 11$ &$\mathbf 1\boxtimes\epsilon$ &$11\times \emptyset$\\
\hline
$(s_2,\epsilon\boxtimes\mathbf 1)$ &$\theta_{K_0}$ &$0$ &$0$\\
\hline
$(s_2,\mathbf 1\boxtimes\epsilon)$ &$0$ &$0$ &$\theta_{K_2}$\\
\hline
$(s_2,\epsilon\boxtimes\epsilon)$ &$11\times \emptyset$ &$\epsilon\boxtimes\mathbf 1$ &$\emptyset\times 11$\\
\hline
\end{tabular}
\end{center}
\smallskip
\caption{Elliptic $\mathsf{Sp}_4(F)$-representations attached to $u=(311)\in \mathsf{SO}_5$}\label{ta:sp4}
\end{table}

The corresponding stable combinations are
\begin{align*}
\Pi(u,1,\delta)&=\pi(u,s_0,\mathbf 1)-\pi(u,s_0,\epsilon),\\
\Pi(u,-1,\delta)&=\pi(u,s_1,\mathbf 1)-\pi(u,s_1,\epsilon),\\
\Pi(u,\delta,1)&=\pi(u,s_2,\mathbf 1\boxtimes \mathbf 1)+\pi(u,s_2,\mathbf 1\boxtimes \epsilon)+\pi(u,s_2,\epsilon\boxtimes \mathbf 1)+\pi(u,s_2,\epsilon\boxtimes \epsilon),\\
\Pi(u,\delta,-1)&=\pi(u,s_2,\mathbf 1\boxtimes \mathbf 1)+\pi(u,s_2,\mathbf 1\boxtimes \epsilon)-\pi(u,s_2,\epsilon\boxtimes \mathbf 1)-\pi(u,s_2,\epsilon\boxtimes \epsilon),\\
\Pi(u,\delta,\delta)&=\pi(u,s_2,\mathbf 1\boxtimes \mathbf 1)-\pi(u,s_2,\mathbf 1\boxtimes \epsilon)+\pi(u,s_2,\epsilon\boxtimes \mathbf 1)-\pi(u,s_2,\epsilon\boxtimes \epsilon),\\
\Pi(u,\delta,-\delta)&=\pi(u,s_2,\mathbf 1\boxtimes \mathbf 1)-\pi(u,s_2,\mathbf 1\boxtimes \epsilon)-\pi(u,s_2,\epsilon\boxtimes \mathbf 1)+\pi(u,s_2,\epsilon\boxtimes \epsilon).
\end{align*}
The corresponding parahoric restrictions are in Table \ref{ta:sp4-stable}. Here the column labeled $K_i$ contains $\res_{K_i}(\Pi(u, s, h))$.

\begin{table}
\begin{center}
\begin{tabular}{|c|c|c|c|}
\hline
$(s,h)$ &$K_0$ &$K_1$ &$K_2$ \\
\hline
\hline
$(1,\delta)$ &$(1\times 1-\emptyset\times 2)+\emptyset\times 11$ &$\mathbf 1\boxtimes \epsilon+\epsilon\boxtimes \mathbf 1$ &$(1\times 1-\emptyset\times 2)+\emptyset\times 11$\\
\hline
$(\delta,1)$ &$(11\times \emptyset+\theta_{K_0})+\emptyset\times 11$ &$\mathbf 1\boxtimes \epsilon+\epsilon\boxtimes \mathbf 1$ &$(11\times \emptyset+\theta_{K_2})+\emptyset\times 11$\\
\hline
$(-1,\delta)$ &$(-1\times 1+\emptyset\times 2)+\emptyset\times 11$ &$\mathbf 1\boxtimes \epsilon-\epsilon\boxtimes \mathbf 1$ &$(1\times 1-\emptyset\times 2)-\emptyset\times 11$\\
\hline
$(\delta,-1)$ &$(-11\times \emptyset-\theta_{K_0})+\emptyset\times 11$ &$\mathbf 1\boxtimes \epsilon-\epsilon\boxtimes \mathbf 1$ &$(11\times \emptyset+\theta_{K_2})-\emptyset\times 11$\\
\hline
$(\delta,\delta)$ &$ (-11\times\emptyset+\theta_{K_0})+\emptyset\times 11$ &$\mathbf 1\boxtimes \epsilon-\epsilon\boxtimes \mathbf 1$ &$(11\times \emptyset-\theta_{K_2})-\emptyset\times 11$\\
\hline
$(\delta,-\delta)$ &$ (11\times\emptyset-\theta_{K_0})+\emptyset\times 11$ &$\mathbf 1\boxtimes \epsilon+\epsilon\boxtimes \mathbf 1$ &$(11\times \emptyset-\theta_{K_2})+\emptyset\times 11$\\
\hline
\end{tabular}
\end{center}
\smallskip
\caption{Elliptic $\mathsf{Sp}_4(F)$ stable combinations attached to $u=(311)\in \mathsf{SO}_5$}\label{ta:sp4-stable}
\end{table}

One can easily verify by inspection using Table \ref{ta:sp4-stable} that the conjecture holds in this case.

\section{\texorpdfstring{$\SL_n(F)$}{\SL n(F)}}\label{s:type-A}

\subsection{Elliptic pairs for $G^\vee=\PGL_n(\CC)$}
Consider the case $G^\vee=\PGL_n(\mathbb C)$. Let $Z$ denote the centre of $\GL_n(\mathbb C)$. In the Weyl group of type $A_{n-1}$ ($W=S_n$), denote by $\dot w_n$ the permutation matrix corresponding to the $n$-cycle $(1,2,3,\dots,n)$. For every $n$-th root $z$ of $1$, let 
\[\Delta_n(z)=\text{diag}(1,z,z^2,\dots,z^{n-1}) Z\in \PGL_n(\mathbb C).
\]
Fix $\zeta_n$ a primitive $n$-th root of $1$ and set $s_n=\Delta_n(\zeta_n)$. Notice that $\dot w_n$ and $s_n$ commute in $\PGL_n(\CC)$.

\begin{lem}\label{l:pgl-princ} Suppose $\Gamma=\PGL_n(\mathbb C)$. Then
\[\cY(\Gamma)_\ellip=\bigsqcup_{k\in (\mathbb Z/n\mathbb Z)^\times}\Gamma\cdot (s_n,\dot w_n^k).
\]
In particular, there are $\varphi(n)$ $\Gamma$-orbits in $\cY(\Gamma)_\ellip$. The flip $(s,h)\to (h,s)$ induces the following map on $\Gamma$-orbits in $\cY(\Gamma)_\ellip$:
\[\mathsf{flip}\colon  ([(s_n,\dot w_n^k)])\to [(s_n, \dot w_n^{-k})],\quad  k\in (\mathbb Z/n\mathbb Z)^\times.
\]

\end{lem}

\begin{proof}
Let $T$ be the diagonal torus in $\Gamma$. By Proposition \ref{p:conn}, the only possible elliptic pairs are conjugate to $(s,\dot w)$ where $w$ is elliptic and $s$ is regular such that $s\in T^w$. If the group is semisimple of type $A_{n-1}$, then the only elliptic elements of the Weyl group are the $n$-cycles. We may assume that $\dot w=\dot w_n$. It is easy to see that
\[ T^{\dot w_n}=\{ \Delta_n(z)\mid z^n=1\}.
\]
Since $s\in T^{\dot w_n}$ needs to be regular, it follows that the corresponding $z$ must be a primitive root of $1$. 

Now fix $s=s_n$. Every other $\Delta(\zeta')$ with $\zeta'$ a primitive $n$-th root is conjugate in $\Gamma$ to $s_n$. The centralizer is $\rZ_\Gamma(s)=\langle T, \dot w_n^k\mid k\in \mathbb Z/n\mathbb Z\rangle\cong T\rtimes \mathbb Z/n\mathbb Z$. This means that $\dot w_n^i$ is conjugate to $\dot w_n^j$ in $\rZ_\Gamma(s)$ if and only if $i=j$. On the other hand, $\dot w_n^k$ is elliptic if and only if $k\in (\mathbb Z/n\mathbb Z)^\times$, hence the claim follows.

For the claim about the Fourier transform, let $x\in \GL_n(\CC)$ be such that $x^{-1} \dot w_n x=s_n$, where $x$ is the matrix corresponding to a basis of eigenvectors of $\dot w_n$. Then a calculation shows that
\[ x^{-1} s_n x=\dot w_n^{-1} \text{ in } \PGL_n(\CC).
\]
From this:
\[  \mathsf{flip}([(s_n,\dot w_n^k)])=[(w_n^k, s_n)]=[(x s_n^k x^{-1},s_n)]=[(s_n^k, x^{-1} s_n x)]=[(s_n^k,\dot w_n^{-1})].
\]
Finally, let $p$ be a permutation matrix such that $p^{-1} s_n^k p=s_n$ (this exists since $k$ is coprime to $n$). This has the effect $p^{-1} \dot w_n p=\dot w_n^k$, hence $[(s_n^k,\dot w_n^{-1})]=[(s_n,\dot w_n^{-k})]$.

\end{proof}

Now let $u$ be a unipotent element in $\PGL_n(\mathbb C)$. Via the Jordan canonical form, $u$ is parameterized by a partition $\lambda$ of $n$, where we write $\lambda=(\underbrace{1,\dots,1}_{r_1},\underbrace{2,\dots,2}_{r_2},\dots,\underbrace{\ell,\dots,\ell}_{r_\ell})$. As it well known (see for example \cite[Theorem 6.1.3]{CM}) 
\begin{equation}
\Gamma_u=\left(\prod_{i=1}^\ell \GL_{r_i}(\mathbb C)^i_\Delta\right)/Z,
\end{equation}
where $H^i_\Delta$ means $H$ embedded diagonally into the product of $i$ copies of $H$. In particular, $\Gamma_u$ is connected. Let $T_r$ denote the diagonal torus in $\GL_r$, $Z_r$ the center of $\GL_r$,  $\bar T_r =T_r/Z_r$, and $W_r\cong S_r$ the Weyl group. A maximal torus in $\Gamma_u$ is $T_u=\prod_{i=1}^\ell (T_{r_i})_\Delta^i/Z$ and the Weyl group is $W_u=\prod_{i=1}^\ell (W_{r_i})_\Delta^i$. 

Let $w=\prod_{i=1}^\ell (w_i)^i_\Delta\in W_u$, $w_i\in W_{r_i}$ be given. We need $T_u^w$ to be finite. The morphism
\[\pi\colon T_u\twoheadrightarrow \prod_{i=1}^\ell (\bar T_{r_i})_\Delta^i,\ (t_i)_i\text{ mod }Z\mapsto (t_i\text{ mod }Z_i)
\]
is surjective and $W_u$-equivariant. Since $(\bar 1,\dots,\bar 1,(\bar T_{r_i}^{w_i})_\Delta^i,\bar 1,\dots,\bar 1)\subset T_u^w$ for each $i$, it follows that $w_i$ is elliptic for $\PGL_{r_i}$, hence, each $w_i$ is an $r_i$-cycle.

\begin{prop}\label{p:pgl}
For $u\in \PGL_n(\mathbb C)$,  $\cY(\Gamma_u)_\ellip\neq \emptyset$ if and only if the partition $\lambda$ corresponding to $u$ is rectangular, i.e., $\lambda=(\underbrace{i,\dots,i}_{r_i})$ for some $i$. In this case, \[\Gamma_u=\GL_{r_i}(\mathbb C)^i_\Delta/Z\cong \PGL_{r_i}(\mathbb C),\] so $\cY(\Gamma_u)_\ellip$ is as in Lemma \ref{l:pgl-princ}.
\end{prop}

\begin{proof}
Let $w\in W_u$ be elliptic as above. We pass to the Lie algebra $\mathfrak t_u=\mathfrak s (\bigoplus (\mathfrak t_{r_i})^i_\Delta)$; here $\mathfrak s$ denotes the traceless matrices. Since $\mathfrak t_{r_i}^{w_i}=\mathbb C \cdot \Id_{r_i}$, we see that
\[\mathfrak t_u^w=\{(a_1 \Id_{r_1}, a_2 \Id_{r_2},a_2\Id_{r_2},\dots, \underbrace{a_i\Id_{r_i},\dots,a_i\Id_{r_i}}_{i},\dots)\mid \sum_{i=1}^\ell i a_i=0\}.
\]
The element $w$ is elliptic if and only if $\mathfrak t_u^w=0$. From the condition $\sum_{i=1}^\ell i a_i=0$, we that this can only happen if there exists a unique $i$ such that $r_i\neq 0$.
\end{proof}

\begin{cor}
The number of orbits of elliptic pairs for $\PGL_n(\mathbb C)$ is
\[\sum_{u \text{ unipotent class}} |\Gamma_u\backslash\cY(\Gamma_u)_\ellip|=n.
\]
\end{cor}

\begin{proof}
From Proposition \ref{p:pgl}, the only unipotent classes that contribute are the rectangular ones, which are in one-to-one correspondence with divisors $d$ of $n$. For the unipotent class $u=(\underbrace{n/d,\dots,n/d}_d)$, Lemma \ref{l:pgl-princ} says that there are $\varphi(d)$ orbits of elliptic pairs. Hence the total number is $\sum_{d|n}\varphi(d)=n$.

\end{proof}

\subsection{Elliptic unipotent representations of $\SL_n(F)$} 
 It is instructive to make explicit the elliptic correspondence (Conjecture \ref{conj:main-ell}) for $G=\SL_n(F)$.
Let $K_0=\SL_n({\mathfrak o}_F)$ and let $I\subset K_0$ be an Iwahori subgroup.  Let $\cH(G,I)=\{f\in C^\infty_c(G)\mid f(i_1 g i_2)=f(g),\ \text{for all }i_1,i_2\in I\}$ be the Iwahori--Hecke algebra (under convolution with respect to a fixed Haar measure). The algebra $\cH(G,I)$ is naturally isomorphic to the affine Hecke algebra $\cH=\cH(\cR,\sqrt q,1)$, where $q$ is the order of the residue field of $F$ and $\cR$ is the root datum for $\PGL_n(\CC)$.

Every irreducible unipotent $G$-representation has nonzero fixed vectors under $I$, in other words, $\fR_\un(G)=\fR_I(G)$, where $\fR_I(G)$ is the category of smooth representations generated by their $I$-fixed vectors. The functor
\[m_I\colon  \fR_\un(G)=\fR_I(G)\to \cH(G,I)\text{-mod}, \quad V\mapsto V^I,
\]
is an equivalence of categories. The Langlands parameterization in this case can be read off the Kazhdan--Lusztig classification of irreducible modules for $\cH(G,I)$ extended to this setting in \cite{Re}:
\begin{equation}
\Irr_\un \SL_n(F)\leftrightarrow \Irr~\cH(G,I)\leftrightarrow \PGL_n(\CC)\backslash\{(x,\phi)\mid x\in\PGL_n(\CC),\phi\in\widehat A_x\}.
\end{equation}
The exact functor $m_I$ induces an isomorphism
\[\Ext^i_G(V,V')=\Ext^i_{\cH(G,I)}(V^I,V'^I), \text{ for all }i,
\]
and therefore $\EP_G(V,V')=\EP_{\cH(G,I)}(V^I,V'^I)$ for all $V,V'\in \Irr~\fR_I(G)$. Since $\EP_G$ and $\EP_{\cH(G,I)}$ are additive, they extend to pairings on the Grothendieck groups of finite-length representations. Let $R_I(G)_\CC$ be the $\CC$-span of $\Irr ~\fR_I(G)$ and $\overline R_I(G)_\CC$ the quotient by the radical of $\EP_G$. Define $\overline R(\cH(G,I))_\CC$ similarly. Thus:
\begin{lem}
The equivalence of categories $m_I$ gives an isomorphism $\overline m_I\colon \overline R_I(G)_\CC\to \overline R(\cH(G,I))_\CC$ which is isometric with respect to $\EP_G$ and $\EP_{\cH(G,I)}$.
\end{lem}
The elliptic theory of affine Hecke algebras is well understood \cite{OS}, and we reviewed the basic facts in Section \ref{s:ell-Hecke}. In particular,  via (\ref{e:deform}) and (\ref{ind-iso})), we get that 
\begin{equation}
\overline R_I(G)_\CC \cong \overline R(\wti W)_\CC\cong \bigoplus_{s\in W\backslash T_{\iso}^\vee} \overline R(W_s)_\CC,
\end{equation}
where $\wti W=W^a$, $W$, $T^\vee$, $W_s$ are as in Example \ref{e:W-PGL}. If $n=dm$,  we consider $s_{d\times m}:=\text{diag}(\underbrace{1,\dots,1}_m,\underbrace{\zeta_d,\dots,\zeta_d}_m,\dots,\underbrace{\zeta_d^{d-1},\dots,\zeta_d^{d-1}}_m)$, where $\zeta_d$ is a primitive $d$-th root of $1$. This is $S_n$-conjugate to $\Delta_n(\zeta_d)$. In that case, 
\[\rZ_{G^\vee}(s_{d\times m})=\mathrm{P}(\GL_m(\CC)^d)\rtimes \bZ/d\bZ,
\]
which has component group $A_{G^\vee}(s_{d\times m})=\bZ/d\bZ$. The Lie algebra of the maximal (diagonal) torus in $\rZ_{G^\vee}(s_{d\times m})$ is 
\[\mathfrak t^\vee_{d\times m}=\{\underline x=(x_1,\dots,x_n)\mid \sum_i x_i=0\},
\]
on which $W_{s_{d\times m}}$ acts in the standard way: break $(x_1,\dots,x_n)$ into $m$-tuples \[\underline y_i=(x_{(i-1)m+1},x_{(i-1)m+2},\dots, x_{im}), \quad 1\le i\le d.\] Then the $i$-th $S_m$ acts by the natural permutation action on $\underline y_i$, whereas $\bZ/d\bZ$ permutes cyclically $(\underline y_1,\dots,\underline y_d)$. We consider the elliptic theory of $W_{s_{d\times m}}$ on $\mathfrak t^\vee_{d\times m}$ with respect to this action.

\begin{lem}\label{l:Ws-SL} There are $\varphi(d)$ elliptic conjugacy classes of $W_{s_{d\times m}}$ acting on $\mathfrak t^\vee_{d\times m}$ with representatives $g_\xi=(w_m,1,1,\dots,1) \xi$, where $\xi$ ranges over the elements of order $d$ in $\bZ/d\bZ$, and $w_m$ is a fixed $m$-cycle in $S_m$.
\end{lem}

\begin{proof}
We first show that each such element is elliptic. Without loss of generality, suppose that $\xi$ acts as the standard cycle $(1,2,\dots,d)$ and $w_m=(1,2,\dots,m)$. Then $g_\xi\cdot \underline x=\underline x$ implies $x_1=x_{(d-1)m+1}=x_{(d-2)m+1}=\dots=x_{m+1}$ which then equals to $x_m$ (because of the effect of $w_m$), then with $x_{dm}=x_{(d-1)m}=\dots=x_{2m}$, which then equals $x_{m-1}$ etc. It follows that all coordinates $x_i$ are the same and since the sum of the coordinates is $0$, we get that there are no nonzero fixed points.

Secondly, no two $g_{\xi}$'s are conjugate. This is clear because if $\xi,\xi'$ are distinct in $\bZ/d\bZ$, then $x\xi$ and $x'\xi'$ are in different conjugacy classes for all $x,x'$ in $S_m^d$.

It remains to show that these are all the elliptic conjugacy classes. Let $x\xi$ be an element with $x=(\sigma_1,\dots,\sigma_d)\in S_m^d$ and $\xi\in \bZ/d\bZ$. If $\xi$ does not have order $d$, then there exists points $\underline x=(\underline y_1,\dots,\underline y_d)\in \mathfrak t^\vee_{d\times m}$ (here, as above, each $\underline y_i$ is an $m$-tuple) fixed under the action of $\xi$ such that not all $y_i$'s are equal. This means in particular, that there exists $j$ such that $\underline y_j=(x_{(j-1)m+1},\dots,x_{jm})$ is arbitrary and $\sum_{l=(j-1)m+1}^{jm} x_l\neq 0$. But then every $\sigma_j\in S_m$ has a nonzero fixed point $\underline y_j$, for example taking all of the entries of $\underline y_j$ to be equal, and therefore $x\xi$ is not elliptic.

This means that necessarily $\xi$ has order $d$. We claim that the conjugacy classes of $x\xi$ are in one-to-one correspondence with conjugacy classes of $S_m$ via the correspondence
\[ w\in S_m\mapsto (w,1,\dots,1)\xi \in S_m^d\rtimes \bZ/d\bZ.
\]
Without loss of generality, suppose $\xi$ acts by shifting the indices $i\to i+1$ mod $d$. We show that every element $x\xi$, $x=(\sigma_1,\dots,\sigma_d)$, is conjugate to an element of the form $(w,1,\dots,1)\xi$. This is equivalent to the existence of permutations $z_1,\dots,z_d\in S_m$ such that 
\[\sigma_1=z_1 w z_2^{-1},\ \sigma_2=z_2 z_3^{-1},\dots,\sigma_{d-1}=z_{d-1} z_d^{-1},\ \sigma_d=z_d z_1^{-1}.
\]
This can be solved easily, by taking $z_1=1$, then $z_d=\sigma_d$, $z_{d-1}=\sigma_{d-1}\sigma_d,\dots,$ $z_2=\sigma_2\sigma_3\dots\sigma_d$, $w=\sigma_1\sigma_2\dots\sigma_d$.

A similar calculation shows that $(w,1,\dots,1)\xi$ and $(w',1,\dots,1)\xi$ are conjugate if and only if $w,w'$ are conjugate in $S_m$. (If $w'=zwz^{-1}$, then $(w',1,\dots,1)\xi$ and $(w,1,\dots,1)\xi$ are conjugate via $(z,z,\dots,z)$.) 

Finally, if an element $(w,1,\dots,1)\xi$, $\xi$ of order $d$, is elliptic then $w$ is elliptic in $S_m$, otherwise if $\underline y$ is a fixed point of $w$, $(\underline y,\dots,\underline y)$ is a fixed point of $(w,1,\dots,1)\xi$. This concludes the proof.
\end{proof}

On the other hand, we have unipotent classes $u$ in $\mathrm{P}(\GL_m(\CC)^d)$ and we need to look at the elliptic theory of $A_{G^\vee}(s_{d\times m} u)$ on the Lie algebra of the maximal torus in $\rZ_{\mathrm{P}(\GL_m(\CC)^d)\rtimes \bZ/d\bZ}(u)$. 
Let $u=u_{d\times m}$ be the unipotent element given by the principal Jordan normal form on each of the  $\GL_m$-blocks. Then the reductive part of the centralizer is
\[\rZ_{G^\vee}(s_{d\times m} u_{d\times m})^{\mathrm{red}}=\mathrm{P}(\rZ_{\GL_m(\CC)^d})\rtimes \bZ/d\bZ,
\]
hence $A_{G^\vee}(s_{d\times m} u_{d\times m})=\bZ/d\bZ$ and this acts on the Cartan subalgebra
\[\mathfrak t^\vee(s_{d\times m} u_{d\times m})=\left\{(z_1 \Id_m,\dots,z_d\Id_m)\mid \sum_i{z_i}=0\right\}.
\]
In particular, $\overline R(A_{s_{d\times m} u_{d\times m}})_\bC$ has dimension $\varphi(d)$ and can be identified with the class functions on the elements of order $d$ in $\bZ/d\bZ$. Thus, in the case of $\mathsf{SL}_n(F)$, the elliptic correspondence for unipotent representations takes the following very concrete form.

\begin{prop}\label{p:ell-SLn} Let $G=\SL_n(F)$. The local Langlands correspondence for unipotent representations induces an isometric isomorphism
\[\overline{\mathsf{LLC}^p}_\un\colon \bigoplus_{d\mid n} \overline R(A_{x_{d\times m}}) \longrightarrow  \overline R_\un(\SL_n(F)),\quad \phi\mapsto \pi(x_{d\times m},\phi)\]
where $x_{d\times m}=s_{d\times m} u_{d\times m}\in \PGL_n(\CC)$ is as above and $A_{x_{d\times m}}=\bZ/d\bZ$.
\end{prop}

The connection with the elliptic pairs for $G^\vee=\PGL_n(\CC)$ from Proposition \ref{p:pgl} is:
\[\bigoplus_{u\in \cC(\PGL_n(\CC))_\un} \CC[\cY(\Gamma_u)_\ellip]^{\Gamma_u}=\bigoplus_{d|n} \CC[\cY(\Gamma_{u_{d\times m}})_\ellip]^{\Gamma_{u_{d\times m}}} \cong  \bigoplus_{d\mid n} \overline R(A_{x_{d\times m}}).
\]

\begin{rem}
Note that by taking dimensions in Proposition \ref{p:ell-SLn}, we recover the well-known formula $\sum_{d \mid n} \varphi(d) = n$, where, as above, $\varphi$ denotes the Euler phi function. 
\end{rem}

\subsection{The elliptic Fourier transform for $\SL_n(F)$} The results so far imply that we have an equivalence 
\[\overline R_\un(\SL_n(F))_\bC\cong \overline R(\cH)_\bC\cong \overline R(\wti W)_\bC\cong \bigoplus_{u\in \cC(\PGL_n(\CC))_\un} \Gamma_u\backslash \cY(\Gamma_u)_\ellip.
\]
The spaces involved are all $n$-dimensional and we describe the basis of $\overline R_\un(G)_\bC$ given by the virtual characters $\Pi(u,s,h)$.

First consider the two extremes. At one extreme, we have the regular unipotent class $u_\reg$. Then $\Gamma_{u_\reg}=\{1\}$ and $\pi(u_\reg,1,1)=\St$. At the other, end, $u=1$, $\Gamma_1=\PGL_n(\CC)$, and there are $\varphi(n)$ orbits of elliptic pairs $(s_n,\dot w_n^k)$, $k\in (\ZZ/n\ZZ)^\times$, as in Lemma \ref{l:pgl-princ}. The component group is $A_{s_n}=\langle \dot w_n\rangle\cong \bZ/n\bZ$. Let $\pi(1,s_n)$ denote the tempered unramified principal series of $G$ with Satake parameter $s_n\in W\backslash T^\vee$. Since $W_{s_n}=A_{s_n}=\bZ/n\bZ$, the theory of (analytic) R-groups provides a well-known decomposition
\[\pi(1,s_n)=\bigoplus_{\phi\in \widehat A_{s_n}} \pi(1,s_n,\phi),
\]
where each $\pi(1,s_n,\phi)$ is an irreducible tempered $G$-representation. Identifying $\widehat {\bZ/n\bZ}$ with $\ZZ/n\ZZ$ (via a choice $\zeta_n$ of primitive $n$-th root of unity), we get
\begin{equation}
\Pi(1,s_n,\dot w_n^k)=\sum_{\ell\in \ZZ/n\ZZ}\zeta_n^{\ell k} ~ \pi(1,s_n,\phi_\ell),\qquad k\in (\ZZ/n\ZZ)^\times,
\end{equation}
where $\phi_\ell(\zeta_n)=\zeta_n^\ell$. Moreover, as an $\cH$-module, $\pi(1,s_n,\phi_\ell)^I$ is the (unique) irreducible tempered $\cH$-module with central character $W\cdot s_n$ such that 
\[\sigma_0(\pi(1,s_n,\phi_\ell)^I)=\Ind_{\bZ/n\bZ\ltimes X}^{S_n\ltimes X} (\phi_\ell\otimes s_n).
\]
Now, more generally, by Proposition \ref{p:pgl}, $\cY(\Gamma_u)_\ellip\neq \emptyset$ if and only if $u = u_{d\times m}$ is labelled by a rectangular partition $(\underbrace{m,\dots,m}_d)$ of $n$. In this case $\Gamma_u=\PGL_d(\bC)$. Recall
$s_{d\times m}$
and $x_{d\times m}=s_{d\times m} u_{d\times m}$. 
Consider the parabolically induced tempered $G$-representation
\[\pi(u_{d\times m},s_{d\times m})=\Ind_{P_{d\times m}}^{\SL_n(F)}((\St_m\otimes\CC_1)\boxtimes (\St_m\otimes \CC_{\zeta_d})\boxtimes\dots\boxtimes (\St_m\otimes \CC_{\zeta_d^{d-1}})),
\] 
where $P_{d\times m}$ is the block-upper-triangular parabolic subgroup with Levi subgroup $M_{d\times m}=S(\GL_m(F)^d)$, $\St_m$ is the Steinberg representation of $\GL_m(F)$, and $\CC_z$ is the unramified character of $\GL_m(F)$ corresponding to the semisimple element $z\Id_m$ in the dual complex group $\GL_m(\CC)$. The R-group in this case is $\bZ/d\bZ$ which coincides with $A_{x_{d\times m}}$. We have a decomposition into irreducible tempered $G$-representations:
\[
\pi(u_{d\times m},s_{d\times m})=\bigoplus_{\phi\in \widehat A_{x_{d\times m}}} \pi(u_{d\times m},s_{d\times m},\phi).
\]
Taking $I$-fixed vectors and the deformation $\sigma_0$, we have
\[\sigma_0((\pi_{d\times m},s_{d\times m},\phi_\ell)^I)=\Ind_{W_{s_{d\times m}\ltimes X}}^{S_n\ltimes X}(\sgn_{d\times m}\phi_\ell\otimes s_{d\times m}),\
\]
where recall that $W_{s_{d\times m}}=S_m^d\rtimes \bZ/d\bZ$ and $ \sgn_{d\times m}$ is the sign character of $S_m^d$.

Define
\begin{equation}
\Pi(u_{d\times m},s_{d\times m},\dot w_{d\times m}^k)=\sum_{\ell\in \ZZ/d\ZZ}\zeta_d^{\ell k} ~ \pi(u_{d\times m},s_{d\times m},\phi_\ell),\qquad k\in (\ZZ/d\ZZ)^\times.
\end{equation}
The elliptic Fourier transform in this case is
\begin{equation}\label{e:FT-sl}
\FT^\vee_\ellip(\Pi(u_{d\times m},s_{d\times m},\dot w_{d\times m}^k))=\Pi(u_{d\times m},s_{d\times m},\dot w_{d\times m}^{-k}), \qquad k\in (\ZZ/d\ZZ)^\times.
\end{equation}
The maximal compact open subgroups of $\SL_n(F)$ are maximal parahoric subgroups $K_i$, one for each vertex $i$ of the affine Dynkin diagram. With this notation, $K_0=\SL_n(\mathfrak o_F)$. Moreover, $\InnT^p\overline K_i=\{\overline K_i\}$. All $K_i$ are isomorphic to $K_0$ (conjugate in $\GL_n(F)$), hence for all $i$, the nonabelian Fourier transform of $\overline K_i$ is the identity. Let $W_i\cong S_n$ denote the finite parahoric subgroup of $W^a$ corresponding to $K_i$, so that $W_0=W$. 
The isomorphism $W_i\cong W$ is given by the map $s_j\mapsto s_{(j-i)\text{ mod }n}$. 
By Mackey induction/restriction
\begin{equation}\label{e:restr}
\Ind_{W_{s_{d\times m}\ltimes X}}^{S_n\ltimes X}(\sgn_{d\times m}\phi_\ell\otimes s_{d\times m})|_{W_i}\cong \Ind_{(W_{s_{d\times m}})_i}^{W_i} (\sgn_{d\times m}\phi_\ell\otimes s_{d\times m}),
\end{equation}
where $(W_{s_{d\times m}} )_i=(W_{s_{d\times m}}\ltimes X)\cap W_i.$
Let $\gamma=\epsilon_1-\epsilon_n$ be the highest root of type $A_{n-1}$, in the standard coordinates, so that $s_0=s_\gamma t_\gamma$, denoting by $t_\gamma\in X\subset W^a$ the corresponding translation. Then one can see that $W_{s_{d\times m}} \cong (W_{s_{d\times m}} )_i$ is given by sending 
\[s_j\mapsto \begin{cases} s_j,&\text{ if }j\neq i \\ s_i t_{\epsilon_i-\epsilon_{i+1}},&\text{ if }j=i,\end{cases}\quad 1\le j<n.
\]
\begin{lem}\label{l:Sn-restr}
For every $0\le i<n$, 
\[ \Ind_{(W_{s_{d\times m}})_i}^{W_i} (\sgn_{d\times m}\phi_\ell\otimes s_{d\times m})\cong \Ind_{W_{s_{d\times m}}}^{S_n}(\sgn_{d\times m}\phi_{\ell+\lfloor \frac im\rfloor}).
\]
\end{lem}
\begin{proof}
In light of the observation before the statement of the lemma, we only need to trace how the inducing character changes on the generator corresponding to $i$. Denote by $(S_m^d)_i$ the image of $S_m^d$ inside $(W_{s_{d\times m}})_i$, and similarly for $(\bZ/d\bZ)_i$. 
If $s_i$ is a generator of $(S_m^d)_i$, then the value of the character $s_{d\times m}$ on $t_{\epsilon_i-\epsilon_{i+1}}$ is $1$. On the other hand, if $s_i$ is not a generator of $(S_m^d)_i$, then there is also no change. This means that the inducing character on the $(S_m^d)_i$ is still $\sgn_{d\times m}$.

The generator $\xi$ of $\bZ/d\bZ$ is, in cycle notation, a product of the disjoint cycles $(l,m+l,2m+l,\dots,(d-1)m+l)$, where $l$ ranges from $0$ to $m-1$ (when $l=0$, we mean the cycle $(m,2m,\dots,dm)$). Then the simple reflection $i$ contributes to the cycle $l$ for $i=jm+l$, $j=\lfloor\frac i m\rfloor$. In $(\bZ/d\bZ)_i$, we then get a $\theta_{\epsilon_i-\epsilon_{i+1}}$, which we need to move to the end of the product of cycles, and we get that the image $(\xi)_i$ in $(\bZ/d\bZ)_i$ is $(\xi)_i=\xi t_{\epsilon_i-\epsilon_{(d-1)m+l}}$. The character $s_{d\times m}$ acts on $ t_{\epsilon_i-\epsilon_{(d-1)m+l}}$ by $\zeta_d^{j}$, which means that $\phi_\ell s_{d\times m}$ acts on $(\xi)_i$ by $\zeta_d^{\ell+j}$, which proves the claim.
\end{proof}

\begin{prop}\label{p:sl}
Conjecture \ref{c:elliptic} holds for $G=\SL_n(F)$. More precisely, for each $0\le i<n$, 
\[\res_{K_i}\circ \FT^\vee_{\ellip}(\Pi(u_{d\times m},s_{d\times m},\dot w_{d\times m}^k))=\zeta_d^{2k\lfloor\frac i m\rfloor} \FT_{\cpt,\un}\circ \res_{K_i} (\Pi(u_{d\times m},s_{d\times m},\dot w_{d\times m}^k)).
\]
\end{prop}

\begin{proof}
To verify Conjecture \ref{c:elliptic}, given (\ref{e:FT-sl}), it is sufficient to compare the restrictions to $W_i$ of 
$\sigma_0(\res_{K_i}(\Pi(u_{d\times m},s_{d\times m},\dot w_{d\times m}^k)^I))$ and $\sigma_0(\res_{K_i}(\Pi(u_{d\times m},s_{d\times m},\dot w_{d\times m}^{-k})^I)$ as virtual $W_i$-characters. For this, we apply (\ref{e:restr}) and Lemma \ref{l:Sn-restr} and get with $j=\lfloor \frac im \rfloor$:
\begin{align*}
\sigma_0(\res_{K_i}(\Pi(u_{d\times m},s_{d\times m},\dot w_{d\times m}^k)^I))&\cong\sum_{\ell\in \bZ/d} \zeta_d^{k\ell}\Ind_{W_{s_{d\times m}}}^{S_n}(\sgn_{d\times m}\phi_{\ell+j})\\
&=\zeta_d^{-kj} \sum_{\ell} \zeta_d^{k\ell}\Ind_{W_{s_{d\times m}}}^{S_n}(\sgn_{d\times m}\phi_{\ell}).
\end{align*}
On the other hand,
\begin{align*}
\sigma_0(\res_{K_i}&(\Pi(u_{d\times m},s_{d\times m},\dot w_{d\times m}^{-k})^I))\cong\sum_{\ell\in \bZ/d} \zeta_d^{-k\ell}\Ind_{W_{s_{d\times m}}}^{S_n}(\sgn_{d\times m}\phi_{\ell+j})\\
&=\sum_{\ell\in \bZ/d} \zeta_d^{k\ell}\Ind_{W_{s_{d\times m}}}^{S_n}(\sgn_{d\times m}\phi_{-\ell+j})=\sum_{\ell\in \bZ/d} \zeta_d^{k\ell}\Ind_{W_{s_{d\times m}}}^{S_n}(\sgn_{d\times m}\phi_{\ell-j})\\
&=\zeta_d^{kj} \sum_{\ell} \zeta_d^{k \ell}\Ind_{W_{s_{d\times m}}}^{S_n}(\sgn_{d\times m}\phi_{\ell}).
\end{align*}
Here we have used that \[\Ind_{W_{s_{d\times m}}}^{S_n}(\sgn_{d\times m}\phi_{\ell'})\cong \Ind_{W_{s_{d\times m}}}^{S_n}(\sgn_{d\times m}\phi_{-\ell'}),
\]
because $\phi_{-\ell'}$ is the $\bZ/d\bZ$-representation contragredient to $\phi_{\ell'}$. This implies that the two sides are contragredient to each other, but all irreducible $S_n$-representations are self-contragredient. 

\end{proof}

\section{\texorpdfstring{$\PGL_n(F)$}{\PGL n(F)}}\label{s:pgl}

Now suppose $\bG = \PGL_n$. The dual group $G^\vee$ is $\SL_n(\bC)$. 
Each unipotent element $u \in G^\vee$ corresponds to a partition $\lam_u$ of $n$, and if $\lam_u = \lambda=(\underbrace{1,\dots,1}_{r_1},\underbrace{2,\dots,2}_{r_2},\dots,\underbrace{\ell,\dots,\ell}_{r_\ell})$, then 
\begin{equation}\label{eqn-GammauPGL}
\Gamma_u \simeq \left\{ (x_1, \dots, x_\ell) \in \prod_{i = 1}^\ell \GL_{r_i}( \bC) \mid \prod_{i = 1}^\ell \det(x_i)^i = 1\right\}.
\end{equation}

\begin{lem}\label{l:pgl-pairs}
The group
$\Gamma_u$ contains elliptic pairs if and only if $u$ is regular unipotent. In this case $\cY(\Gamma_u) = \cY(\Gamma_u)_{\ellip} = \{(s, h) \mid s, h \in \rZ_{\SL_n(\bC)}\}$.
\end{lem}

\begin{proof}The proof is very similar to that of Proposition \ref{p:pgl}.
Note that if $u$ is not rectangular, then $\Gamma_u$ has infinite center: for example, with notation as in (\ref{eqn-GammauPGL}), given $t \in \bC^\times$, the element
$(t\Id_{r_1}, t\Id_{r_2}, \dots, t\Id_{r_{\ell-1}}, t^{r_\ell - \frac{n}{\ell}}\Id_{r_\ell}) \in \rZ_{\Gamma_u}$. So if $\Gamma_u$ contains an elliptic pair, then $\lam_u$ is of the form $(k, k, \dots, k)$ for some $k$ dividing $n$, and $\Gamma_u \simeq \{x \in \GL_{\frac{n}{k}}(\bC) \mid \det(x)^k = 1\}$. Explicitly, we can think of $\Gamma_u$ as a split extension
\begin{equation*}
1 \to \SL_{\frac{n}{k}}(\bC) \to \Gamma_u \to \mu_{k} \to 1
\end{equation*}
where the first inclusion is the natural one, and the map to $\mu_{k}$ is given by the determinant. 
Now, given semisimple elements $s, h \in \Gamma_u$ such that $sh = hs$, there exists $g \in \SL_{\frac{n}{k}}(\bC)$ such that $gsg^{-1}, ghg^{-1}$ are both diagonal in $\GL_{\frac{n}{k}}(\bC)$. So a maximal torus of $\SL_{\frac{n}{k}}( \bC)$ centralizes both $s$ and $h$, and if $(s, h)$ is an elliptic pair, we must have $k = n$. 
\end{proof}

We can now easily prove Conjecture \ref{c:elliptic} in this case. 

\begin{thm}
Conjecture \ref{c:elliptic} holds when $\bG = \PGL_n$. More precisely, when $\bG = \PGL_n$, we have 
\begin{equation*}
    \res_{\cpt, \un} \circ \FT^\vee_\ellip = \FT_{\cpt, \un} \circ \res_{\cpt, \un}.
\end{equation*}
\end{thm}

\begin{proof}
Using Lemma \ref{l:pgl-pairs}, the proof of the theorem reduces to Proposition \ref{p:regular}.
\end{proof}

To illustrate the theorem, we explicitly describe the case when $\bG = \PGL_2$. Note that even this low-rank example shows that certain choices were necessary in our set up: to relate $\FT^\vee_{\ellip}$ to a finite Fourier transform for the non-split inner twist of $G$, first, we must consider maximal compact subgroups instead of just parahorics (otherwise the restrictions of $\Pi(u, 1, 1)$ and $\Pi(u, -1, 1)$ would be the same, though $\FT^\vee_{\ellip}$ fixes the first but not the second). Second, $\FT_{\cpt, \un}$ must mix subspaces corresponding to distinct inner twists to give a well-defined linear map. 

\begin{ex}\label{ex:pgl2}

Now let $\bG = \PGL_2$. Then $G$ has a unique non-split inner twist $G'$, which 
we can describe explicitly as follows: 
Let 
\begin{equation*}
  D =   \left\{ \begin{pmatrix} a & \varpi b\\\overline{b} & \overline{a} \end{pmatrix} \mid a, b \in F_{(2)} \right\},
\end{equation*}
where $F_{(2)}$ is the degree-$2$ unramified extension of $F$. Then $D$ is a $4$-dimensional division algebra over $F$, and we can take $G':= D^\times/F^\times$. 

Let $\chi_0$ be the unramified character of $F^\times$ given by $\varpi \mapsto -1$. 
Then the nontrivial weakly unramified character of $G$ (resp. $G'$) is given by $\chi := \chi_0 \circ \det$ (resp. $\chi' := \chi_0 \circ \det$).
Let $\St_G$ denote the Steinberg representation of $G$ (and similarly for $\St_{G'}$, which is the trivial representation of $G'$), and let $u \in \SL_2(\bC)$ be regular unipotent. Then the virtual representations corresponding to our 4 elliptic pairs are
\begin{eqnarray*}
\Pi(u, 1, 1) &=& \St_G + \St_{G'}\\
\Pi(u, 1, -1) &=& \St_G - \St_{G'}\\
\Pi(u, -1, 1) &=& (\St_G \otimes \chi) + (\St_{G'}\otimes \chi')\\
\Pi(u, -1, -1) &=& (\St_{G} \otimes \chi) - (\St_{G'}\otimes \chi').
\end{eqnarray*}
The involution $\FT^\vee_{\ellip}$ switches $\Pi(u, 1, -1)$ and $\Pi(u, -1, 1)$, and fixes the other two sums.

Let $I$ be the Iwahori subgroup of $G$ given by 
\begin{equation*}
    I = \left\{ \begin{pmatrix} a & \varpi b\\c & d \end{pmatrix} \mid a, d \in \mathfrak{o}_F^\times, b, c \in \mathfrak{o}_F\right\}.
\end{equation*}
With notation as in Section \ref{s:maximal}, we have $\Omega_G \simeq \bZ/2\bZ$. 
The set $S_{\max}(G)$ contains two elements $(A, \cO)$: one corresponding to $A = \Omega_G$, and one corresponding to $A$ trivial.
Thus the group $G$ has two conjugacy classes of maximal compact open subgroups: the maximal parahoric subgroup $K_0 := \PGL_2(\mathfrak{o}_F)$ (which corresponds to $A$ trivial) and $K_1 := \rN_G(I)$ (which corresponds to $A = \Omega_G$). Note that $K_1$ contains $I$ with index 2: it is generated by $I$ and $\sigma:= \begin{psmallmatrix} 0 & \varpi\\1 & 0\end{psmallmatrix}$. 
The reductive quotients are given by $\overline{K}_0 \simeq \PGL_2(\mathbb{F}_q)$ and $\overline{K}_1 \simeq \mathbb{F}_q^\times \rtimes \bZ/2\bZ$. Note that $\chi$ is trivial on $K_0$ and on $I$, but $\chi(\sigma) = -1$, so $\chi$ induces the sign character on the component group of $\overline{K}_1$.

In the notation of Section \ref{s:maximal}, we have $G' = G_x$, where $x$ is the nontrivial element of $\Omega_G$. Thus $G'$ has a unique conjugacy class of maximal compact subgroups, corresponding to the element $(A, \cO)$ of $S_{\max}(G)$ with $A = \Omega_G$. Explicitly, $G'$ itself is compact. The 
only parahoric subgroup of $G'$ is $I' := \mathfrak{o}_D^\times/\mathfrak{o}_F^\times$, where $\mathfrak{o}_D$ is the ring of integers of $D$, and this parahoric is normal in $G'$.
The group $G'$ is generated by $I'$ and $\sigma$ (defined as above), which again has order 2 in $G'$. The reductive quotient $\overline{G'} \simeq  (\mathbb{F}_{q^2}^\times/\mathbb{F}_q^\times) \rtimes \bZ/2\bZ$. Note that as above $\chi'$ is trivial on $I'$ but takes the value $-1$ on $\sigma$, so $\chi'$ factors through the sign character of the component group of $\overline{G'}$. 

The space $\cC(G)_{\cpt, \un}$ is given by
\begin{equation*}
    \cC(G)_{\cpt, \un} = R_{\un}(\overline{K}_0) \oplus R_{\un}(\overline{K}_1) \oplus R_{\un}(\overline{G'}),
\end{equation*}
and the map $\res_{\cpt, \un}$ is defined on the virtual representations above by
\begin{eqnarray*}
\Pi(u, 1, 1) &\mapsto& \St_{K_0} + \St_{I} + \St_{I'}\\
\Pi(u, 1, -1) &\mapsto& \St_{K_0} + \St_{I} - \St_{I'}\\
\Pi(u, -1, 1) &\mapsto& \St_{K_0}  + (\St_{I}\otimes \sgn)  + (\St_{I'}\otimes \sgn)\\
\Pi(u, -1, -1) &\mapsto& \St_{K_0} + (\St_{I} \otimes \sgn) - (\St_{I'}\otimes \sgn),
\end{eqnarray*}
where, as in the proof of Proposition \ref{p:regular}, $\St_K$ denotes the Steinberg representation of $\overline{K}$, and $\sgn$ denotes the sign representation of the relevant component group. 

If $(A, \cO) \in S_{\max}(G)$ with $A$ trivial, then $\res_{\cO}(\Pi(u, s, h)) = \St_{K_0}$ for all elliptic pairs $(s, h)$. In this case, 
$\FT_{\cpt, \un}$ restricts to the identity map on $R_{\un}(\overline{K}_0)$.

Now suppose $(A, \cO) \in S_{\max}(G)$ with $A = \Omega_G$. Then $\res_{\cO}$ is given by projection onto $R_{\un}(\overline{K}_1) \oplus R_{\un}(\overline{G'})$. In the notation of Section \ref{s:disconn}, and with $\cU$ the (one-element) family consisting of the Steinberg representation of $\overline{K}_1$, we have $\widetilde{\Gamma}_{\cU}^A = A$, so $\cM(\widetilde{\Gamma}_{\cU}^A)$ consists of four elements: $(1, \triv), (1, \sign), (x, \triv), (x, \sgn)$ (where, as above, $x$ is the nontrivial element of $\Omega_G$). These correspond to following elements of $R_\un (\overline{K}_1) \oplus R_\un (\overline{G'})$:
\begin{eqnarray*}
(1, \triv) &\longleftrightarrow& \St_I\\
(1, \sgn) &\longleftrightarrow& \St_I \otimes \sgn\\
(x, \triv) &\longleftrightarrow& \St_{I'}\\
(x, \sgn) &\longleftrightarrow& \St_{I'} \otimes \sgn.
\end{eqnarray*} 
With notation as in (\ref{e:pi-U}), we have
\begin{eqnarray*}
\res_\cO (\Pi(u, 1, 1)) &=& \Pi_{\widetilde{\cU}}(\triv, \triv)\\
\res_\cO (\Pi(u, 1, -1)) &=& \Pi_{\widetilde{\cU}}(\triv, \sgn)\\
\res_\cO (\Pi(u, -1, 1)) &=& \Pi_{\widetilde{\cU}}(\sgn, \triv)\\
\res_\cO (\Pi(u, -1, -1)) &=& \Pi_{\widetilde{\cU}}(\sgn, \sgn),
\end{eqnarray*}
where $\widetilde{\cU}$ is the family indexed by $\Gamma_{\cU}^A$. Thus the proof of Proposition \ref{p:regular}, and Conjecture \ref{c:elliptic}, may be easily verified.
\end{ex}



\begin{thebibliography}{AAAA}

\bibitem[Ar06]{ArNote} J.~Arthur, {\em A note on $L$-packets},
Pure Appl. Math. Quaterly {\bf 2.1} (2006), 199--217.

\bibitem[Ar13]{ArClass}
\bysame, 
The endoscopic classification of representations: orthogonal and symplectic groups, 
Colloquium Publications volume {\bf 61}, American Mathematical Society, 2013.

\bibitem[Ar89]{ArAutom} \bysame, \emph{Unipotent automorphic representations: conjectures}, Ast\' erisque {\bf 171-172} (1989), 13--71.

\bibitem[Ar93]{ArEll} \bysame, \emph{On elliptic tempered characters}, Acta. Math. {\bf 171} (1993), 73--130.

\bibitem[ABPS16]{ABPS}  A.-M.~Aubert, P.~ Baum, R.~Plymen, M.~Solleveld, {\em The local Langlands correspondence for inner forms of $\SL_n$}, Res. Math. Sci. \textbf{3} (2016), Paper No. 32, 34 pp.

\bibitem[ABPS17a]{ABPSdisc} \bysame,
{\em The principal series of $p$-adic groups with disconnected centre}, Proc. London Math. Soc. \textbf{114} (2017) 798--854.

\bibitem[ABPS17b]{ABPSconj}\bysame,
{\em Conjectures about $p$-adic groups and their non commutative geometry}, 5--21, in
 ``Around Langlands correspondences'', Contemp. Math., \textbf{691}, Amer. Math. Soc., Providence, RI, 2017.

\bibitem[AMS18]{AMS1}  A.-M. Aubert, A. Moussaoui, M. Solleveld, \emph{Generalizations of the Springer correspondence and cuspidal Langlands 
parameters},  Manuscripta Math. \textbf{157} (2018), no. 1-2, 121--192. 

    
\bibitem[AMS17]{AMS3}  \bysame, \emph{Affine Hecke algebras for Langlands parameters}, \href{https://arxiv.org/abs/1701.03593}{arXiv:1701.03593}.

\bibitem[Be84]{Be} J.~Bernstein, {\em Le ``centre'' de Bernstein}, edited by P. Deligne, in Travaux en Cours, Representations of reductive groups over a local field, 1--32, Hermann, Paris, 1984.



\bibitem[Bo05]{Bon} C. Bonnaf\' e, \emph{Quasi-isolated elements in reductive groups}, Comm. Alg. {\bf 33} (2005), 2315--2337.

\bibitem[BT72]{BT1} F. Bruhat, J. Tits, \emph{Groupes r\'eductifs sur un corps local: I. Donn\'ees radicielles valu\'ees}, Inst. Hautes \'Etudes Sci. Publ. Math. \textbf{41} (1972), 5--251. 

\bibitem[BT84]{BT2} \bysame, \emph{Groupes r\'eductifs sur un corps local: II. Sch\'emas en groupes. Existence d'une donn\'ee radicielle valu\'ee}, Inst. Haut. \'Etudes Sci. Publ. Math. \textbf{60} (1984), 197--376.

\bibitem[Ca93]{Car}
R.~Carter,
Finite groups of Lie type: conjugacy classes and complex characters, Wiley Classics Library, A Wiley-Interscience Publication. John Wiley \& Sons, Ltd., Chichester, 1993


\bibitem[Ci20]{Ciu} D.~Ciubotaru, {\em The nonabelian Fourier transform for elliptic unipotent representations of exceptional $p$-adic groups}, \texttt{arXiv:2006.13540v2}, 2020.

\bibitem[CO15]{CO} D.~Ciubotaru, E.~Opdam,
\emph{Formal degrees of unipotent discrete series and the exotic Fourier transform}, Proc. Lond. Math. Soc. (3) {\bf 110} (2015), no. 3, 615--646.

\bibitem[CO17]{CO-ell} \bysame,
 \emph{On the elliptic nonabelian Fourier transform for unipotent representations of $p$-adic groups}, Representation theory, number theory, and invariant theory, 87--113, Progr. Math., {\bf 323}, Birkh\" auser/Springer, Cham, 2017.


\bibitem[CH16]{CH2016} D.~Ciubotaru, X.~He, \emph{The cocenter of the graded affine Hecke algebra and the density theorem}, J. Pure Appl. Algebra 220 (2016), no. 1, 382–410.

\bibitem[CH17]{CH} \bysame,
\emph{Cocenters and representations of affine Hecke algebras},  J. Eur. Math. Soc. (JEMS) {\bf 19} (2017), no. 10, 3143--3177. 

\bibitem[CH21]{CH2}\bysame,
{\em{Cocenters of p-adic groups, III: elliptic and rigid concenters}}, Peking Math. J. {\bf 4} (2021) no. 2, 159--186.

\bibitem[CM93]{CM} D.~Collingwood, W.~McGovern,
Nilpotent orbits in semisimple Lie algebras, 
Van Nostrand Reinhold Mathematics Series., New York, 1993. xiv+186 pp. 


\bibitem[DL76]{DL} P.~Deligne, G.~Lusztig,
{\em Representations of reductive groups over finite fields}, Ann. of Math. (2) {\bf 103} (1976), no. 1, 103--161.

\bibitem[DM90]{DM} F.~Digne, J.~Michel,
{\em On Lusztig's parametrization of characters of finite groups of Lie type},  Ast\' erisque No. {\bf 181-182} (1990), 6, 113--156.

\bibitem[DM18]{DM2}\bysame,
\emph{Quasi-semisimple elements}, Proc. Lon. Math. Soc. (3) {\bf 116} (2018), no. 5, 1301--1328.

\bibitem[FOS20]{FOS} Y.~Feng, E.~Opdam, M.~Solleveld,
{\em Supercuspidal unipotent representations: L-packets and formal degrees}, J. Éc. polytech. Math. {\bf 7} (2020), 1133–1193.


\bibitem[GM20]{GM} M.~Geck, G.~Malle,
The character theory of finite groups of Lie type. A guided tour, Cambridge Studies in Advanced Mathematics, {\bf 187}, Cambridge, 2020.

\bibitem[Ha14]{Hai} T. Haines, {\em The stable Bernstein center and test functions for Shimura varieties},  Automorphic forms and Galois representations Vol. 2, 118--186, London Math. Soc. Lecture Note Ser. \textbf{415}, Cambridge Univ. Press, Cambridge, 2014.


\bibitem[He23]{hetz-2023} J.D. Hetz, \emph{Characters and character sheaves of finite groups of Lie type}, Ph.D. thesis, Universit\" at Stuttgart 2023, \texttt{https://elib.uni-stuttgart.de/bitstream/11682/12951/1/DissertationJHetz.pdf}

\bibitem[HS12]{HS} K. Hiraga, H. Saito, {\em On $L$-packets for inner forms of $\SL_n$},  Mem. Am. Math. Soc. \textbf{215}, No.~1013, 2012.


\bibitem[IM65]{IM} N. Iwahori, H. Matsumoto, {\em On some Bruhat decomposition and the structure of the
              Hecke rings of $p$-adic Chevalley groups}, Inst. Hautes \'{E}tudes Sci. Publ. Math. \textbf{25} (1965), 5--48.

\bibitem[Kal16]{KaLLC} T. Kaletha, {\em The local Langlands conjectures for non-quasi-split groups}, Families of automorphic forms and the trace formula, 217--257,
Simons Symp., Springer, [Cham], 2016.

\bibitem[Kal18]{Ka} \bysame,
{\em Global rigid inner forms and multiplicities of discrete automorphic representations}, Invent. Math. \textbf{213} (2018),  271--369. 


\bibitem[Kaz86]{Kaz}
D.~Kazhdan,
\emph{Cuspidal geometry of $p$-adic groups},
J. Analyse Math. {\bf 47} (1986), 1--36.

\bibitem[KL87]{KL}
D.~Kazhdan, G.~Lusztig, {\em Proof of the Deligne-Langlands conjecture for Hecke algebras}, Invent. Math. \textbf{87} (1987), no. 1, 153--215. 

\bibitem[Kn65]{Kn}M.~Kneser, \emph{Galois-Kohomologie halbeinfacher algebraischer Gruppen \"uber $p$-adischen K\" orpern. II}, Math. Z. {\bf 89} (1965), 250--272. 

\bibitem[Ko84]{Ko} R. Kottwitz, {\em Stable trace formula: cuspidal tempered terms}, Duke Math. J. 
\textbf{51} (1984), no. 3, 611--650.

\bibitem[Lu84a]{Lubook} G.~Lusztig, {Characters of reductive groups over a finite field}, Ann. Math. Stud., Princeton, 1984.

\bibitem[Lu84b]{LuIC} \bysame, {\em Intersection cohomology complexes on a reductive group}, Invent. Math. {\bf 75.2} (1984), 205--272.

\bibitem[Lu86]{Lu86} \bysame, {\em Character sheaves, IV}, Adv. Math. {\bf 59} (1986), 1--63.

\bibitem[Lu95]{LuI} \bysame, {\em Classification of unipotent representations of simple $p$-adic groups} Internat. Math. Res. Notices \textbf{1995}, no. 11, 517--589.

\bibitem[Lu02]{LuII} \bysame, {\em Classification of unipotent representations of simple $p$-adic groups. II}, Represent. Theory \textbf{6} (2002), 243--289. 


\bibitem[Lu15]{LualmostI} \bysame, {\em Unipotent almost characters of simple $p$-adic groups}, Ast\'erisque \textbf{370} (2015), 243--267. 

\bibitem[Lu14]{LualmostII} \bysame, {\em Unipotent almost characters of simple $p$-adic groups, II},  Transform. Groups \textbf{19} (2014), no. 2, 527--547.

\bibitem[Lu18]{LualmostBirk} \bysame, {\em On the definitions of almost characters} in ``Lie groups, geometry, and representation theory", 367--379, Progr. Math. \textbf{326}, Birkh\"auser/Springer, Cham, 2018. 

\bibitem[MW03]{MW}  C. M{\oe}glin, J.-L. Waldspurger,  {\em Paquets stables de repr\'esentations temp\'er\'ees et de r\'eduction unipotente pour $\SO(2n+1)$}, Invent. Math. \textbf{152} (2003), 461--623.




\bibitem[OS09]{OS} E. Opdam, M. Solleveld, {\em Homological algebra for affine Hecke algebras},  Adv. Math. \textbf{220} (2009), no. 5, 1549--1601.

\bibitem[OS13]{OS2} \bysame, \emph{Extensions of tempered representations}, Geom. Funct. Analysis {\bf 23} (2013), 664--714.


\bibitem[Re02]{Re} M. Reeder, {\em Isogenies of Hecke algebras and a Langlands correspondence for ramified principal series representations}, 
Represent. Theory \textbf{6} (2002), 101--126.

\bibitem[Re00]{Re2} \bysame, \emph{Formal degrees and L-packets of unipotent discrete series representations of exceptional $p$-adic groups. With an appendix by Frank L\" ubeck}, J. Reine Angew. Math. {\bf 520} (2000), 37--93. 

\bibitem[Re01]{Re3} \bysame, \emph{Euler--Poincar\'e pairings and elliptic representations of Weyl groups and p-adic groups}, Compositio Math. {\bf 129} (2001), no. 2, 149--181. 

\bibitem[Sh95]{shoji-unipotent} T. Shoji, 
\emph{Character sheaves and almost characters of reductive groups I, II}, Adv. Math. {\bf 111} (1995), 244--313 and 314--354.


\bibitem[SS97]{SS} P. Schneider, U. Stuhler, {\em Representation theory and sheaves on the Bruhat-Tits building},  Inst. Hautes \'Etudes Sci. Publ. Math. No. \textbf{85} (1997), 97--191.

\bibitem[Se02]{Se} J.-P. Serre, Galois cohomology, Springer Monographs in Mathematics, Springer 2002.

\bibitem[So23]{So1} M. Solleveld, {\em A local Langlands correspondence for unipotent representations}, Amer. J. Math. {\bf 145} (2023) no.3, 673--719.

\bibitem[So23]{So2} \bysame,  {\em  On unipotent representations of ramified $p$-adic groups},  Repres. Theory {\bf 27} (2023), 669--716.

\bibitem[Som01]{Som} E.~Sommers, \emph{Lusztig's canonical quotient and generalized duality}, 
J. Algebra {\bf 243} (2001), no. 2, 790--812. 

\bibitem[SpSt70]{SpSt} T.A. Springer, R. Steinberg,
\emph{Conjugacy classes. 1970 Seminar on Algebraic Groups and Related Finite Groups (The Institute for Advanced Study, Princeton, N.J., 1968/69)}, 167--266, 
Lecture Notes in Mathematics, Vol. {\bf 131} Springer, Berlin.

\bibitem[St68]{St} R.~Steinberg,
\emph{Endomorphisms of linear algebraic groups}, Mem. Amer. Math. Soc. {\bf 80} (1968).

\bibitem[Ti65]{Ti} J. Tits, {\em Classification of algebraic semisimple groups}, Algebraic Groups and Discontinuous Subgroups (Proc. Sympos. Pure Math., Boulder, Colo., 1965) 
pp. 33--62 Amer. Math. Soc., Providence, R.I., 1966.


\bibitem[Vo93]{Vo} D.~Vogan, {\em The local Langlands conjecture}, Representation Theory of Groups and Algebras, Contemp. Math. \textbf{145} Amer. Math. Soc., Providence (1993)


\bibitem[Wa07]{Wa} J.-L. Waldspurger,  {\em Produit scalaire elliptique},  Adv. Math. \textbf{210} (2007), no. 2, 607--634.

\bibitem[Wa18]{Wa2} \bysame, {\em Repr\'esentations de r\'eduction unipotente pour $\SO(2n+1)$, I: une involution}, J. Lie Theory \textbf{28} (2018), no. 2, 381--426.

\end{thebibliography}
\end{document}